\documentclass{amsart}

\usepackage[T5,T1]{fontenc}

\usepackage{amsmath}
\usepackage{amsfonts}
\usepackage{amssymb}
\usepackage{amsthm}
\usepackage{ifpdf}
\usepackage{ifthen}
\usepackage{enumerate}
\usepackage{aliascnt}
\ifpdf
\usepackage[colorlinks=true]{hyperref}
\else
\usepackage[draft=true]{hyperref}
\fi
\usepackage{cleveref}

\usepackage[all]{xy}

\usepackage{pifont}

\ifpdf
\usepackage{graphicx,color}
\else
\usepackage{graphicx}
\fi

\usepackage{setspace}

\usepackage{comment}

\title{Local Newton nondegenerate Weil divisors in toric varieties}

\author[A. N\'emethi]{Andr\'as N\'emethi}
\address{Alfr\'ed R\'enyi Institute of Mathematics, ELKH\newline \hspace*{4mm}
Re\'altanoda utca 13-15, H-1053, Budapest, Hungary \newline \hspace*{4mm}
ELTE - University of Budapest, Dept. of Geometry, Budapest, Hungary \newline
\hspace*{4mm}
BCAM - Basque Center for Applied Math.,
Mazarredo, 14 E48009 Bilbao, Basque Country, Spain}
\email{nemethi.andras@renyi.hu}
\thanks{
The first author was partially  supported by NKFIH Grant ``\'Elvonal
(Frontier)'' KKP 126683.
}

\author[B. Sigur\dh sson]{Baldur Sigur\dh sson}
\address{
Institute of Mathematics,\newline \hspace{4mm}
{{\fontencoding{T5}\selectfont
18 Đường Hoàng Quốc Việt,
Quận Cầu Giấy,
10307,
Hanoi,
Vietnam}}}
\email{baldursigurds@gmail.com}

\date{\today}

\swapnumbers

\newtheorem{thm}{Theorem}[section]
\crefname{thm}{theorem}{theorems}
\Crefname{thm}{Theorem}{Theorems}

\newaliascnt{lemma}{thm}
\newtheorem{lemma}[lemma]{Lemma}
\aliascntresetthe{lemma}

\newaliascnt{prop}{thm}
\newtheorem{prop}[prop]{Proposition}
\aliascntresetthe{prop}
\crefname{prop}{proposition}{propositions}
\Crefname{prop}{Proposition}{Propositions}

\newaliascnt{cor}{thm}
\newtheorem{cor}[cor]{Corollary}
\aliascntresetthe{cor}
\crefname{cor}{corollary}{corollaries}
\Crefname{cor}{Corollary}{Corollaries}

\theoremstyle{definition}

\newaliascnt{warning}{thm}

\aliascntresetthe{warning}
\crefname{warning}{warning}{warningollaries}
\Crefname{warning}{Warning}{Corollaries}

\newaliascnt{rem}{thm}
\newtheorem{rem}[rem]{Remark}
\aliascntresetthe{rem}
\crefname{rem}{remark}{remarks}
\Crefname{rem}{Remark}{Remarks}

\newaliascnt{example}{thm}
\newtheorem{example}[example]{Example}
\aliascntresetthe{example}

\newaliascnt{definition}{thm}
\newtheorem{definition}[definition]{Definition}
\aliascntresetthe{definition}

\newaliascnt{war}{thm}

\aliascntresetthe{war}
\crefname{war}{warning}{warnings}
\Crefname{war}{Warning}{Warnings}

\newaliascnt{block}{thm}
\newtheorem{block}[block]{}
\aliascntresetthe{block}
\crefname{block}{}{}
\Crefname{block}{}{}

\newaliascnt{conv}{thm}

\aliascntresetthe{conv}
\crefname{conv}{convention}{conventions}
\Crefname{conv}{Convention}{Conventions}

\numberwithin{equation}{section}
\crefformat{equation}{eq.~#2(#1)#3}
\Crefformat{equation}{Eq.~#2(#1)#3}

\crefname{subsection}{subsection}{subsections}
\Crefname{subsection}{Subsection}{Subsections}

\crefformat{block}{#2#1#3}
\Crefformat{block}{#2#1#3}

\crefformat{dummy}{#2#1#3}

\crefformat{enumi}{#2#1#3}
\crefrangeformat{enumi}{(#3#1#4-#5#2#6)}

\newcommand{\C}{\mathbb{C}}
\newcommand{\CP}{\mathbb{CP}}
\newcommand{\N}{\mathbb{N}}
\newcommand{\Z}{\mathbb{Z}}
\newcommand{\R}{\mathbb{R}}
\newcommand{\Q}{\mathbb{Q}}

\newcommand{\V}{\mathcal{V}}
\newcommand{\UU}{\mathcal{U}}

\newcommand{\Nd}{\mathcal{N}}

\newcommand{\Dc}{D_{\mathrm{c}}}

\newcommand{\X}{\tilde X}

\renewcommand{\O}{\mathcal{O}}
\newcommand{\Lb}{\mathcal{L}}
\newcommand{\Qb}{\mathcal{Q}}
\newcommand{\T}{\mathbb{T}}

\newcommand{\Yloc}{Y^{\rm loc}}
\newcommand{\Ytloc}{\tilde Y^{\rm loc}}
\newcommand{\Dloc}{D^{\rm loc}}

\newcommand{\Yt}{\tilde Y}

\newcommand{\rk}{\mathop{\mathrm{rk}}}

\newcommand{\length}{\mathop{\mathrm{length}}}

\newcommand{\Vol}{\mathop{\rm Vol}\nolimits}
\newcommand{\convx}{\mathop{\rm conv}\nolimits}

\newcommand{\Spec}{\mathop{\rm Spec}\nolimits}
\newcommand{\Pic}{\mathop{\rm Pic}\nolimits}

\newcommand{\supp}{\mathop{\rm supp}\nolimits}
\newcommand{\Hom}{\mathop{\rm Hom}\nolimits}
\newcommand{\wt}{\mathop{\rm wt}\nolimits}

\renewcommand{\div}{\mathop{\rm div}\nolimits}

\renewcommand{\mod}{\mathop{\rm mod}\nolimits}

\newcommand{\Div}{\mathop{\rm Div}\nolimits}
\newcommand{\Star}{\mathop{\rm Star}\nolimits}

\newcommand{\mult}{\mathop{\rm mult}\nolimits}

\renewcommand{\tilde}{\widetilde}

\newcommand{\set}[2]{\left\{ #1 \,\middle\vert\, #2 \right\}}
\newcommand{\gen}[2]{\left\langle #1 \, \middle| \, #2 \right\rangle}

\renewcommand{\epsilon}{\varepsilon}

\newcounter{dummy}
\renewcommand{\thedummy}{\roman{dummy}}
\crefformat{dummy}{#2(#1)#3}

\newcommand{\fa}[2]{\forall #1 :\, #2}

\newenvironment{blist}
{
  \begin{list}{(\thedummy)}
  {
	\setlength\labelsep{4pt}
	\setlength\itemindent{4pt}
	\setlength\leftmargin{0pt}
	\setlength\labelwidth{0pt}
	\setlength\parsep{0pt}
    \usecounter{dummy}
  }
}
{
  \end{list}
}

\newcommand{\mynd}[4]
{
    \begin{figure}[#2]\begin{center}
        \ifpdf
            \input{#1.pdf_t}
        \else
            \input{#1.pstex_t}
        \fi
		\ifthenelse{\equal{#3}{}}
        {
		}
		{
            \caption{#3\label{#4}}
        }
    \end{center}
    \end{figure}
}

\hypersetup{colorlinks=false}

\makeindex

\begin{document}

\maketitle

\begin{abstract}
We introduce and develop the theory of Newton nondegenerate local Weil
divisors $(X,0)$ in toric affine varieties.
We characterize in terms of the toric combinatorics of the Newton diagram
different properties of such singular germs: normality, Gorenstein property,
or being an
Cartier divisor in the ambient space. We discuss certain properties of their
(canonical) resolution $\widetilde{X}\to X $ and the corresponding canonical
divisor.
We provide combinatorial formulae for the delta--invariant $\delta(X,0)$ and
for the
cohomology groups $H^i(\X,\O_{\X})$ for $i>0$.
In the case
$\dim(X,0)=2$, we provide the (canonical) resolution graph from the Newton
diagram
and we also prove that  if such a Weil divisor is normal and Gorenstein,
and the link is a rational homology sphere, then the geometric genus
is given by the minimal path cohomology, a topological invariant of the link.
\end{abstract}

\tableofcontents

\section{Introduction}

\begin{block}
Hypersurface (or complete intersection) germs with nondegenerate Newton
principal part constitute
a very important family of singularities. They provide a bridge
between toric
geometry and
the combinatorics of polytopes.
The computation of their analytic and
topological invariants
serve as guiding models for the general cases, and also as testing ground  for
different
general conjectures and ideas.

On the other hand, from the point of view of the general classification
theorems
in algebraic/analytic  geometry and singularity theory, these hypersurface
germs are rather restrictive.
In particular, it is highly desired to extend such germs to a more general
setting.
Besides the algebraic/analytic motivations there are also several topological
ones too: one has to create
a flexible family, which is able to follow at analytic level different
inductive (cutting and pasting
procedures) of the topology. For example, the link of a surface singularity is
an
oriented  plumbed 3--manifold associated with a graph. In inductive proofs and
constructions it is very
efficient  to consider their splice or JSJ decomposition. This would
correspond to cutting the Newton diagram
by linear planes though their 1--faces, in this way creating non-regular cones
as well, as completely general
toric 3--folds  as ambient spaces for our germs.

The first goal of the present work is to introduce and develop the theory of
Weil divisors
in general affine toric varieties with additional Newton nondegeneracy
condition.
By such extensions we wish to cover non--Gorenstein singularities as well, or
germs which
are not necessarily Cartier divisors in their canonical ambient toric spaces.
In the toric presentation two combinatorial/geometrical  packages  are
needed:
the fan and geometry of the ambient toric variety, and the `dual fan' (as a
subdivision of the previous one)
together with the Newton polytope associated with the equations of the Weil
divisor.

In fact, we will focus on three level of invariants.

The first level is the analytic geometry of the abstract or embedded singular
germs, e.g. normality,
or being Gorenstein or isolated singularity, or being Cartier (or
${\mathbb Q}$--Cartier)
in the ambient toric variety. Furthermore, at this level we wish to
understand/determine
 several numerical sheaf--cohomological invarints as well.

The second level is the toric combinatorics. In terms of this we wish to
characterize
the above analytic properties and provide formulae for the numerical
invariants.

The third level appears explicitly in the case of curve and surface
singularities.
In the case of surfaces we construct the resolution graph (as the plumbing
graph of the link, hence as a complete topological invariant).
It is always a very interesting and difficult task to decide whether
the numerical analytic invariants can be recovered from the resolution graph.
(This is much harder than the formulae via the toric combinatorics:
the Newton polytope preserves considerably more information from the
structure of the equations than
the resolution graph.) In the last part we prove that the geometric genus of
the resolution
can be recovered from the graph. This is a new substantial step in a
project which aims
to provide topological interpretations for sheaf--cohomological invariants,
see \cite{Nem_Nico_SW,NS-hyper,Nemethi_OzsSzInv,Nemethi_ICM}
\end{block}

\begin{block}
Next we provide some additional concrete comments  and
the  detailed presentation of the sections.

After recalling some notation and results from toric geometry,
we generalize the notion of a Newton nondegenerate hypersurface in
$\C^r$ to an arbitrary Weil divisor in an affine toric variety
in \cref{s:divisors}.
These Newton nondegenerate Weil divisors can be resolved using toric
geometry similarly as in the classical case \cite{Oka}, or in
a different generalization \cite{AMGS}.
In \cref{s:curves} we consider Newton nondegenerate curves.
In \cref{s:isol} we provide conditions for Newton nondegenerate
surface singularities to be isolated, and
in \cref{s:res} we generalize Oka's algorithm \cite{Oka} to
construct a resolution of a Newton nondegenerate Weil divisor,
along with an explicit description of its resolution graph.

In \cref{s:geom_genus}, we give a formula for the
$\delta$-invariant and dimensions of cohomologies of the structure
sheaf on a resolution of a Newton nondegenerate germ
in terms of the Newton polyhedron, see \cref{thm:geom_genus}, whose
statement should have independent interest.
In particular, this yields a formula for the geometric genus.
In the classical case, this formula was given by Merle and Teissier
in \cite[Th\'eor\`eme 2.1.1]{Merle_Teissier}.

In \cref{s:can},
we give a formula for a canonical divisor on a resolution
of a Newton nondegenerate Weil divisor, as well as the canonical cycle
in the surface case, in terms of the Newton diagram, see \cref{s:can}.
This formula generalizes results of Oka \cite[\S9]{Oka}.
In the surface case,
we also prove in \cref{s:Gor} that the Gorenstein property
is identified by the
Newton polyhedron, \cref{thm:Gor}.
A similar, but weaker, condition implies that the
singularity is $\Q$-Gorenstein, but is not sufficient, as shown by
an example in \cref{rem:Gor}.

Using the above results, and a technical result verfied in
\cref{s:rem_face},
we generalize a previous result \cite{NS-hyper}
for the classical
case of Newton nondegenerate hypersurface singularities in $\C^3$, namely
that the geometric genus is determined by a computation sequence, and
is therefore topologically determined:
\end{block}

\begin{thm} \label{thm:thm}
Let $(X,0) \subset (Y,0)$ be a two-dimensional
Newton nondegenerate Weil divisor in
the affine toric ambients space $Y$. Assume that $(X,0)$ is normal
and Gorenstein, and that its link is a rational homology sphere.
Then the geometric genus $p_g(X,0)$ equals the
minimal path lattice cohomology associated with the link of $(X,0)$.
In particular, the geometric genus is determined by the topology
of $(X,0)$.
\end{thm}

\section{Toric preliminaries} \label{s:toric}

In this section, we will recall some definitions and statements from
toric geometry. For an introduction, see e.g.
\cite{Fulton_toric} and \cite{Danilov_toric}.

\begin{block}
Let $N$ be a free Abelian group of rank $r \in \N$ and set
$M = N^\vee = \Hom(N,\Z)$, as well as $M_\R = M \otimes \R$ and
$N_\R = N \otimes \R$. If $\sigma \subset N_\R$ is a
cone, the \emph{dual cone} is defined as
\[
  \sigma^\vee = \set{u\in M_\R}{\fa{v\in\sigma}{\langle u,v \rangle \geq 0}}.
\]
We also set
\[
  \sigma^\perp = \set{u\in M_\R}{\fa{v\in\sigma}{\langle u,v \rangle = 0}}.
\]
We will always assume cones to be finitely generated and rational.
To a cone $\sigma\subset N_\R$ we associate
the semigroup $S_\sigma$, the algebra $A_\sigma$ and the affine
variety $U_\sigma$ by setting
\[
  S_\sigma = \sigma^\vee\cap M,  \quad
  A_\sigma = \C[S_\sigma],       \quad
  U_\sigma = \Spec(A_\sigma).
\]
A variety of the form $U_\sigma$ is called an \emph{affine toric variety}. It
has a canonical action of the $r$-torus $\T^r = (\C^*)^r$.
\end{block}

\begin{block}
A \emph{fan} $\triangle$ in $N$ is a collection of cones in $N_\R$
satisfying the following two conditions.
\begin{enumerate}
\item
Any face of a cone in $\triangle$ is in $\triangle$.
\item
The intersection of two cones in $\triangle$ is a face of each of them.
\end{enumerate}
The \emph{support} of a fan $\triangle$ is defined as
$|\triangle| = \cup_{\sigma\in\triangle} \sigma$.
If $\tau,\sigma\in\triangle$ and $\tau \subset \sigma$, then we
get a morphism $U_\tau \to U_\sigma$. These morphisms form a direct system,
whose limit is denoted by $Y_\triangle$ and called the associated
\emph{toric variety}.
The actions of $\T^r$ on the affine varieties $U_\sigma$ for
$\sigma\in\triangle$ glue together to form an action on $Y_\triangle$.
Note that the canonical maps $U_\sigma \to Y_\triangle$ are open inclusions
(note also that the notation $Y_\triangle$ differs from \cite{Fulton_toric}).

Let $\tilde\triangle$ be another fan in a lattice $\tilde N$
and let $\phi:\tilde N \to N$ be a linear map.
Assume that for any
$\tilde\sigma\in \tilde\triangle$ there is a $\sigma \in \triangle$ so that
$\phi(\tilde\sigma)\subset\sigma$.
This induces maps $U_{\tilde\sigma} \to U_\sigma \to Y_\triangle$,
which glue together to form a map
$Y_{\tilde \triangle} \to Y_{\triangle\vphantom{\tilde \triangle}}$.
\end{block}

\begin{lemma}[Proposition, p. 39, \cite{Fulton_toric}] \label{lem:proper}
Let $\tilde\triangle$ and $\triangle$ be fans as above. The
induced map
$Y_{\tilde\triangle} \to Y_{\triangle\vphantom{\tilde \triangle}}$
is proper if and only
if $\phi^{-1}(|\triangle|) = |\tilde\triangle|$.
\end{lemma}

\begin{block}
For any $p\in M$, there is an associated rational function on $U_\sigma$.
These glue together to form a rational function $x^p$ on $Y_\triangle$. We
refer to these functions as \emph{monomials}. A monomial
$x^p$ is a regular function on $Y_\triangle$ if
$p\in |\triangle|^\vee = \cap_{\sigma\in\triangle} \sigma^\vee$.
A map $\phi:\tilde N \to N$ as above induces $\phi^*:M \to \tilde M$.
The monomial $x^p$ on $Y_\triangle$ then pulls back to
$x^{\phi^*(p)}$.
\end{block}

\begin{block} \label{block:transverse}
For $\sigma \in \triangle$, let $O_\sigma$ be the closed subset
of $U_\sigma$ defined by the ideal generated by monomials $x^p$
where $p\in(\sigma^\vee \setminus \sigma^\perp)\cap M$. We identify this
set with its image in $Y_\triangle$. The closure of $O_\sigma$ in
$Y_\triangle$ is denoted by $V(\sigma)$.
In the case when $\sigma$ is a ray, $V(\sigma)$ is a Weil divisor and we
write $D_\sigma = V(\sigma)$.
The orbits of the $\T^r$ action on $Y_\triangle$ are precisely the sets
$O_\sigma$ for $\sigma\in \triangle$. Furthermore, we have (as sets)
\[
  U_\sigma  = \coprod_{\tau \subset \sigma} O_\tau,\quad
  V(\sigma) = \coprod_{\sigma \subset \tau} O_\tau,\quad
  O_\sigma  = V(\sigma) \setminus \bigcup_{\sigma\subsetneq\tau} V(\tau).
\]
Let $N_\sigma$ be the subgroup of $N$ generated by $\sigma\cap N$ and
define
\[
  N(\sigma) = N/N_\sigma,\quad
  M(\sigma) = \sigma^\perp \cap M,\quad
  M_\sigma  = M/M(\sigma).
\]
Note that this way we have $M_\sigma^{\phantom{\vee}} \cong N_\sigma^\vee$ and
$M(\sigma) \cong N(\sigma)^\vee$.
Let $\pi_\sigma:N_\R \to N_\R(\sigma)$ be the canonical projection and set
\[
  \Star(\sigma) = \set{\pi_\sigma(\tau)}{\sigma \subset \tau \in \triangle}.
\]
This set is a fan in $N(\sigma)$, whose associated toric variety
is identified canonically with the orbit closure $V(\sigma)$.
Similarly, let $\varpi_\sigma:M \to M_\sigma$ be the canonical
projection. Assuming $\sigma \in \triangle$ has dimension
$s$, we have
$(U_\sigma, O_\sigma)
\cong
(Y_\sigma\times(\C^*)^{r-s}, (\{0\}\times(\C^*)^{r-s})$. In particular,
$O_\sigma \subset Y_\triangle$ has $Y_\sigma$ as a transverse type.
\end{block}

\begin{definition}
\begin{blist}

\item
For a cone $\Sigma\subset N_\R$, let $\triangle_\Sigma$ denote the fanfan
consisting of all the faces of $\Sigma$. We also write
$Y_\Sigma$ instead of $Y_{\triangle_\Sigma}$.

\item
If $\triangle$ is a fan and $i\in\N$, define
\[
  \triangle^{(i)} = \set{\sigma\in\triangle}{\dim \sigma = i}.
\]

\item
A \emph{regular cone} (resp. \emph{simplicial cone}) is a cone generated
by a subset of an integral (resp. rational) basis of $N$.

\item
A \emph{subdivision} of a fan $\triangle$ is a fan $\tilde\triangle$
so that $|\tilde\triangle| = |\triangle|$ and
each cone in $\triangle$ is a union of cones in $\tilde\triangle$.
A \emph{regular subdivision} is a subdivision consisting of regular cones.

\item
If $\Sigma\subset N_\R$ is a cone and $\triangle$ is a subdivision
of $\triangle_\Sigma$, denote by $\triangle^*$ the fan consisting of
$\sigma\in\triangle$ for which $\sigma\subset\partial\Sigma$.
Here we see $\partial \Sigma$ as the union of the proper faces of $\Sigma$.
As a result, $\triangle^*$ is a subdivision of the fan
$\triangle_\Sigma \setminus \{\Sigma\}$.

\item
Let $\triangle_1,\triangle_2$ be subdivisions of a fan $\triangle$. We say
that $\triangle_2$ \emph{refines} $\triangle_1$ if $\triangle_2$ is
a subdivision of $\triangle_1$, or that $\triangle_2$ is a
\emph{refinement} of $\triangle_1$.

\item
Let $\triangle$ be a fan with a subdivision $\triangle_1$ and let
$\sigma \in \triangle$. The \emph{restriction} of $\triangle_1$ to $\sigma$
is defined as
\[
  \triangle_1|_\sigma = \set{\tau \in \triangle_1}{\tau \subset \sigma}.
\]

\end{blist}
\end{definition}

\section{Analytic Weil divisors in affine toric varieties} \label{s:divisors}

\begin{block}
Throughout this section, as well as the following sections, we will assume
that $N$ has rank $r$ and
that $\Sigma$ is an $r$-dimensional, rational, finitely generated, strictly
convex cone in $N_\R$. This means that $\Sigma \subset N_\R$
is generated over $\R_{\geq 0}$ by a finite set of elements from
$N$, which generate $N$ as a vectorspace, and that $\Sigma^\perp = \{0\}$.
In particular, the orbit $O_\Sigma$ consists of a single point, which we denote
by $0$, and refer to as the \emph{origin}.
Let $Y_\Sigma$ be the affine toric variety
associated with $\Sigma$.

Any subdivision $\triangle$ of $\triangle_\Sigma$ induces a
modification $\pi_\triangle:Y_\triangle \to Y_\Sigma$.

In the sequel we denote by $(Y_\Sigma,0)$ the analytic germ of $Y_\Sigma$ at 0,
and usually we will denote by $Y$ a (small Stein) representative of $(Y_\Sigma,0)$.
(Hence $(Y,0)=(Y_\Sigma,0)$.) If $\pi_\Delta $ is a toric modification, in the
discussions regarding the local analytic germ $(Y,0)$, we will use the same notation
$Y_\Delta $ for
 $\pi_\Delta^{-1}(Y)$  and $D_\sigma$ for $D_\sigma\cap \pi_\Delta^{-1}(Y)$.
Similarly, $O_\sigma$  might stay for $O_\sigma\cap Y \subset Y$ as well.
If in some argument we really wish to stress the differences, we write $Y^{loc}_\Delta$,
$D^{loc}_\sigma$, $O^{loc}_{\sigma}$ for the local objects.

Assume that   $f\in \O_{Y,0}$ is the germ of  a  holomorphic function at the origin,
which has an expansion
\begin{equation} \label{eq:expansion}
  f(x) = \sum_{p\in S_\Sigma} a_p x^p,\quad
  a_p \in \C.
\end{equation}
Then $(\{f=0\},0)\subset (Y,0)$  is the germ of an analytic space.
We set $\supp(f) = \set{p\in S_\Sigma}{a_p \neq 0}$  too.
\end{block}

\begin{definition}
The \emph{Newton polyhedron} of $f$ with respect to $\Sigma$ is the polyhedron
\[
  \Gamma_+(f) = \convx(\supp(f) + \Sigma^\vee),
\]
where $\convx$ denotes the convex closure in $M_\R$. The union of compact
faces of $\Gamma_+(f)$ is denoted by $\Gamma(f)$ and is called the
\emph{Newton diagram} of $f$ with respect to $\Sigma$.
\end{definition}

\begin{block} \label{block:fan_comb}
{\bf The fan $\triangle_f$ and some combinatorial properties.}
It follows from definition
that $\Sigma$ is precisely the set of those linear functions on $M_\R$ having
a minimal value on $\Gamma_+(f)$. Denote by $F(\ell)$ the minimal set
of $\ell\in \Sigma$ on $\Gamma_+(f)$.
For $\ell_1, \ell_2 \in \Sigma$, say that $\ell_1 \sim \ell_2$ if
and only if $F(\ell_1) = F(\ell_2)$. Then $\sim$ is an equivalence
relation on $\Sigma$ having finitely many equivalence classes, each of whose
closure is a finitely generated rational strictly convex cone.
These cones form a fan, which we will denote by $\triangle_f$.
We refer to $\triangle_f$ as the \emph{dual fan} associated with $f$
and $\Sigma$.
Note that $\triangle_f$ refines $\triangle_\Sigma$.

For any $\sigma \in \triangle_f$, the face $F(\ell)$ is independent of
the choice of $\ell \in \sigma^\circ$, where $\sigma^\circ \subset \sigma$
is the relative interior, that is, the topological interior of $\sigma$
as a subset of $N_{\sigma,\R}$.
For $\sigma \in \triangle_f^{(1)}$, the set $\sigma\cap N$ is
a semigroup generated by a unique element, which we denote by $\ell_\sigma$.
For a series
\[
  g\in\O_{Y,0}[x^M] = \set{x^p h}{p\in M,\,h\in\O_{Y,0}},
\]
the support $\supp(g)$ is defined similarly as above,
and for $\sigma \in \triangle_f^{(1)}$ we set
\[
  \wt_\sigma(g) = \min\set{\ell_\sigma(p)}{p\in\supp(g)}.
\]
One verifies that for any such $g$
\begin{equation} \label{eq:van_order}
\mbox{the vanishing order of $g$  along $D_\sigma \subset Y_{\triangle_f}$
is exactly  $\wt_\sigma(g)$}.
\end{equation}
\end{block}

\begin{definition} \label{def:Fflm}
Let $\sigma \in \triangle_f$ and $\ell\in \sigma^\circ$.
Define
\[
  F_\sigma = F(\ell),\quad
  f_\sigma = \sum_{p\in F_\sigma} a_p x^p.
\]
\end{definition}
If $\sigma' \subset N_\R$ is a cone, and $\sigma'^\circ \subset \sigma^\circ$
(for example, if $\sigma'$ is an element of a refinement of $\triangle_f$),
then we set $F_{\sigma'} = F_\sigma$.

If $\sigma \subset \Sigma$ is one dimensional,
set $m_\sigma = \wt_\sigma(f)$.
Thus, $\ell_\sigma|_{F_\sigma} \equiv m_\sigma$.
Note that we have
\[
  \Gamma_+(f) = \set{ u\in M_\R }
                    {\fa{\sigma \in \triangle_f^{(1)}}
                        {\ell_\sigma(u) \geq m_\sigma}}.
\]
This can be compared with the following set.
\begin{definition} \label{def:Gamma_star}
Let
\[
  \Gamma_+^*(f) = \set{ u\in M_\R }
                      {\fa{\sigma \in \triangle_f^{*(1)}}
                          {\ell_\sigma(u) \geq m_\sigma}},
\]
where $\triangle_f^{*(1)}$ denotes the set of rays in $\triangle_f$
contained in the boundary of $\Sigma$.
\end{definition}

\begin{definition}
Denote by $(X,0) \subset (Y,0)$ the union of those local primary components of the germ
defined by $f$ (with their non-reduced structure),  which are not invariant by the torus action.
If $f$ is  reduced along the non-invariant components, this means the following.  Let
$U\subset Y$ be a neighbourhood of the origin on which $f$ converges and
let $X' \subset U$ be defined by $f = 0$. Then $X$ is the closure of
$X' \setminus \cup\set{D_\sigma}{\sigma\in\triangle_\Sigma^{(1)}}$ in $U$.
\end{definition}

\begin{rem} (i)
For any $p\in M$, the function $x^p f$ defines the same germ $(X,0)$.
Thus, we may allow
$f \in \O_{Y,0}[x^M] = \set{x^p g}{p\in M,\,g\in\O_{Y,0}}$ as well.

(ii)
 Since the divisors $\{D_\sigma\,:\, \sigma\in \triangle_\Sigma^{(1)}\}$ are torus-invariant,
 the divisor of $f$ in $Y_{\Sigma}$ is $X+\sum_\sigma m_\sigma D_\sigma$.
\end{rem}

\begin{prop} \label{lem:ideal_gen}
\begin{blist}
\item \label{it:contains_orig_orig}
We have $\Gamma_+(f) = p + \Sigma^\vee$
for some $p\in M$
if and only if $\triangle_f = \triangle_\Sigma$
if and only if the germ
$X$ at $0$ is the empty germ.

\item \label{it:contains_orig_orbit}
For a $\sigma \in \triangle_\Sigma$,
we have $O_\sigma \subset X$ if and only if the normal fan $\triangle_f$
subdivides $\sigma$ into smaller cones, i.e.
$\triangle_f|_\sigma \neq \triangle_\sigma$.

\item \label{it:ideal_orbit}
The ideal $I_X \subset \O_{Y,0}$ which defines $(X,0)$ in $(Y,0)$,
is generated
by the functions $x^p f$ for $p\in M$ satisfying
$\ell_\sigma(p) + m_\sigma \geq 0$ for all
$\sigma\in \triangle^{(1)}_{\Sigma}$.
\end{blist}
\end{prop}
\begin{proof}
Statement \cref{it:contains_orig_orig}
is clear, since $\Gamma_+(f)$ is of the form $p+\Sigma^\vee$
if and only if $f$ is a product of a monomial and a unit in $\O_{Y,0}$.

Statement \cref{it:contains_orig_orbit} follows from
\cref{it:contains_orig_orig}, and the fact that the intersection
of $X$ and a generic transverse space $Y_\sigma$
to $O_\sigma$
has Newton polygon $\varpi_\sigma(\Gamma(f))$, cf. \cref{block:transverse}.

\cref{it:ideal_orbit}
Assuming the given conditions on $p$, the function $x^p f$ is
meromorphic and has no poles. Since $Y$ is normal, $x^p f$ is
analytic and vanishes on $X$. As a result, $x^p f \in I_X$.

To show that these generate $I_X$, take $g\in I_X$.
We must show that $g = hf$, with $h\in \O_{Y,0}[x^M]$ and
$\ell_\sigma(p) + m_\sigma \geq 0$ for $p\in \supp(h)$.

Let $I_{X,M}$ be the localization of $I_X$ along the invariant divisors,
that is, the ideal of meromorphic function germs on $(Y,0)$,
regular on the open torus and vanishing on $X$.
It follows that $I_{X,M} = f\cdot\O_{Y,0}[x^M]$
and $I_X = I_{X,M} \cap \O_{Y,0}$.

Thus, $g =x^r hf$ for some $h\in\O_{Y,0}$ and $r\in M$.
Then, there exist finite families $(h_i)_i$ of units in $\O_{Y,0}$ and
exponents $(p_i)_i$ in $M$ so that $x^r h = \sum_i x^{p_i} h_i$
and the support of $x^r h$ is the disjoint union of the supports of
$x^{p_i} h_i$.
Let us take any $\sigma\in \triangle ^{(1)}_\Sigma$.
The condition on disjointness of supports gives
\[
  \min_i \wt_\sigma x^{p_i} h_i f
  = \wt_\sigma x^r h f
  = \wt_\sigma g
  \geq 0.
\]
As a result, we have $\ell_\sigma(p_i) + m_\sigma \geq 0$ for all $i$.
The result follows.
\end{proof}

\begin{definition} \label{def:pointed}
Let $f$ and $\triangle_f$ be as above.
We say that $\Gamma_+(f)$, or $f$, is ($\Q$-)\emph{pointed} if
there exists a $p\in M$ ($p\in M_\Q)$ such that
$\ell_\sigma(p) = m_\sigma$ for all $\sigma \in \triangle_\Sigma^{(1)}$.
\end{definition}

\begin{prop} \label{prop:pg_Cartier}
\begin{blist}

\item
If $\Sigma$ is regular (resp. simplicial),
then any Newton polyhedron (w.r.t. $\Sigma$) is pointed
at some $p\in M$
(resp. $p \in M_\Q$).

\item \label{it:pg_Cartier}
$f$ is pointed at $p\in M$ if and only if
 $(X,0)$ in  $(Y,0)$ is defined by
a single equation $x^{-p}f$ (cf.  \cref{lem:ideal_gen}). In other words,  $f$ is pointed  if and only if
$(X,0)$ is a Cartier divisor in $(Y,0)$.

\item $f$ is pointed at $p\in M_{\Q}$ if and only if
$(X,0)$ is a $\Q$-Cartier divisor in $(Y,0)$.

\end{blist}
\end{prop}
\begin{proof} (i) Use the fact that $\{\ell_\sigma\,:\, \sigma\in\triangle_\Sigma^{(1)}\}$
is an integral (resp. rational) basis.

(ii)
If $f$ is pointed at $p\in M$ then by  \cref{lem:ideal_gen}, $x^{-p}f\in I_X$.
Moreover, if  $x^{-q}f\in I_X$ for some $q\in M$, then $\ell_\sigma(p-q)\geq 0$ for any $\sigma\in
\triangle_\Sigma^{(1)}$, hence $p-q\in\Sigma\cap M$ and $x^{p-q}\in\O_{Y,0}$.

Conversely, assume that $(X,0)\subset (Y,0)$ is an (analytic) Cartier divisor.
Let $\tilde\triangle_f$ be a smooth subdivision of $\triangle_f$, and
set $\Yt = Y_{\tilde\triangle_f}$.
This is a smooth variety, and the map
$\pi: \Yt \to Y_\Sigma $ is a resolution of $Y$.
Take a small Stein representative $\Yloc \subset Y$, and
set $\Ytloc=\pi^{-1}(\Yloc)$.
Then we have the vanishing
$H^{\geq 1}(\Yt, \O_{\Yt})=0$ 
(see e.g. \cite[Corrollary, p. 74]{Fulton_toric}
or \cite[\S8.5]{Danilov_toric}), and also
 its local analogue
$H^{\geq 1}(\Ytloc, \O_{\Ytloc})=0$
(since the local analytic germ $(Y,0)$ is rational too).
Thus, from the exponential exact sequence,
$\Pic(\Yt)=H^2(\Yt,\Z)$ and
$\Pic(\Ytloc)=H^2(\Ytloc,\Z)$.
On the other hand, $Y$ is weighted homogeneous (as any affine toric variety),
hence $H^2(\Yt,\Z)=H^2(\Ytloc,\Z)$.
In particular,
$\Pic(\Yt) \cong \Pic(\Ytloc)$.
Here the first group is the Picard group of the algebraic
variety, while the second one is the
Picard group of the analytic manifold.

Next, consider the Chow group $A_{r-1}(Y)$
of codimension one, i.e. the group freely generated
by Weil divisors, modulo linear equivalence.
Note that since $\Yt$ is smooth,
we have
$A_{r-1}(\Yt) \cong \Pic(\Yt)$
and
$A_{r-1}(\Ytloc) \cong \Pic(\Ytloc)$.
If we factor these isomorphic groups by the subgroups generated by the
exceptional divisors, we find that the restriction induces
an isomorphism
$A_{r-1}(Y) \cong A_{r-1}(\Yloc)$.

Denote by $\Dloc_\sigma$ the restriction image of $D_\sigma$ under the
above isomorphism.
Since $(X,0)\subset (Y,0)$ is local analytic Cartier,
and the local divisor of $f$ in $Y$ is $X+\Dloc_f$, where
$\Dloc_f = \sum_{ \sigma\in \triangle_\Sigma^{(1)}}m_\sigma \Dloc_\sigma$,
we find that the class of  $\Dloc_f$
is zero in $A_{r-1}(\Yloc)$.
But then, by the above isomorphisms, the class of
$D_f = \sum_{ \sigma\in \triangle_\Sigma^{(1)}}m_\sigma D_\sigma$
is zero in $A_{r-1}(Y)$.

Finally note that $A_{r-1}(Y)$ 
can be computed as follows \cite[3.4]{Fulton_toric}.
Consider the group
$\Div_{\T}(Y) = \Z\gen{ D_\sigma}{\sigma \in \triangle_\Sigma}$
of invariant divisors
and the inclusion $M \hookrightarrow \Div_{\T}(Y)$ sending
$p\in M$ to $\sum_\sigma \ell_\sigma(p) D_\sigma$. Along with the
map $\Div_{\T} \to A_{r-1}(Y)$, this gives a short exact
sequence
\[
  0
  \to
  M
  \to
  \Div_{\T}(Y)
  \to
  A_{r-1}(Y)
  \to
  0.
\]
Since $D_f\in A_{r-1}(Y)$ maps to zero in $A_{r-1}(\Yloc)$
under the above isomorphism, and $D_f \in \Div_{\T}(Y)$, we find
that $D_f$ is in the image of $M$.
But this means exactly that there exists $p\in M$ such that
$\ell_\sigma(p)=m_\sigma$ for all
$\sigma\in \triangle^{(1)}_\Sigma$.

(iii) Use part (ii) for a certain power of $f$.
\end{proof}

\begin{definition}
We say that $f$ has \emph{Newton nondegenerate
principal part} with respect to $\Sigma$
(or simply that $f$ or $(X,0)$ is \emph{Newton nondegenerate})
if for every $\sigma \in \triangle_f$
with  $F_\sigma$ compact, the variety
$
  \Spec(\C[M]/(f_\sigma))$  (that is,  $ \set{x\in \T^r}{f_\sigma(x) = 0 }$ with its non-reduced structure)
is smooth.
Note that $f_\sigma$ is a polynomial since $F_\sigma$ is compact.
\end{definition}

\begin{lemma} \label{lem:contains_orig}
Assume that $(X,0) \subset (Y,0)$ is Newton nondegenerate and
let $\sigma \in \triangle_\Sigma$.
If $O_\sigma \subset X$,
then the generic transverse type of $X$ to
$O_\sigma$ is a Newton nondegenerate singularity with Newton polyhedron
$\varpi_\sigma(\Gamma_+(f)) \subset M_\sigma$.
\end{lemma}

\begin{proof}
The  statement
follows  by restricting $f$ to a toric subspace transverse to $O_\sigma$,
see \cref{block:transverse}.
\end{proof}

\begin{block} \label{block:resolution} {\bf The fan $\tilde{\triangle}_f$ and the associated resolution.}
Assume that $f$ is Newton nondegenerate.
Let $\tilde\triangle_f$ be a regular subdivision of $\triangle_f$.
Then $\tilde Y =  Y_{\tilde{\triangle}_f}$ is a smooth variety, and we have
a modification $\pi:\tilde Y \to Y$.
As a result of the nondegeneracy of $f$, the strict transform $\X$
of $X$ in $\tilde Y$ intersects all orbits in $\tilde Y$ smoothly.
In particular, $\tilde X$ is smooth, and $\pi$ is an embedded resolution
of $(X,0) \subset (Y,0)$.
\end{block}

\begin{lemma} \label{lem:smooth_open}
Assume $(X,0) \subset (Y,0)$ is a Newton nondegenerate Weil divisor.
Then, the singular locus of the germ $(X,0)$
is contained in the union of codimension $\geq 2$ orbits in $(Y,0)$.
\end{lemma}

\begin{proof}
Let $Y^{(\leq r-2)}$ be the union of orbits of dimension $\leq r-2$, that is,
codimension $\geq 2$, in $Y$.
Let $\pi:\tilde Y\to Y$ be as in \cref{block:resolution}.
The restriction
$\pi^{-1}(Y\setminus Y^{(\leq r-2)}) \to Y\setminus Y^{(\leq r-2)}$
is an isomorphism,
and $\X$ is smooth. Therefore, $X \setminus Y^{(\leq r-2)}$ is smooth.
\end{proof}

\section{Newton nondegenerate curve singularities} \label{s:curves}

In this section, we will assume that $\rk N = 2$  and that
$\Sigma\subset N_\R$ is a two dimensional finitely generated
strictly convex rational cone.
Nondegenerate rank 2 singularities appear naturally in the $r=3$  case
as transversal types of certain orbits.

We will introduce the {\it canonical subdivision} and we establish
criterions for irreducibility and smoothness.
They will be used in the context of  rank $r=3$ cones in the
definition of their canonical subdivision and in the characterization of
Newton nondegenerate
isolated surface singularities.

\begin{block} \label{block:two_can}
{\bf Canonical primitive sequence.}
Assume first that $\Sigma$ is nonregular.
Then there exists a sequence of vectors
$\ell_0,\ldots, \ell_{s+1} \in \Sigma\cap N$, called the
\emph{canonical primitive sequence} \cite{Oka} and integers
$b_1,\ldots, b_s \geq 2$, called the associated
\emph{selfintersection numbers}, so that:
\begin{enumerate}

\item \label{it:two_can_bas}
If $0\leq j \leq s$, then $\ell_j, \ell_{j+1}$ form an integral basis for $N$.

\item \label{it:two_can_int}
If $0 < j \leq s$, then $b_j \ell_j = \ell_{j-1} + \ell_{j+1}$.

\item \label{it:two_can_gen}
The set $\{\ell_0, \ldots, \ell_{s+1}\}$ is a minimal set of generators
for the semigroup $\Sigma\cap N$.

\end{enumerate}
This data is uniquely determined up to reversing the order
of $(\ell_j)_j$ and $(b_j)_j$.
It can, in fact, be determined as follows. Let $\alpha$ be the absolute
value of the determinant
of the $2\times 2$ matrix whose columns $\ell, \ell'$
are the primitive generators of the
one dimensional faces of $\Sigma$, given in any integral basis.
Then, there exists a unique integer $0\leq \beta < \alpha$ so that
$\beta \ell + \ell' \in \alpha N$.
The selfintersection numbers are determined as the
\emph{negative continued fraction expansion}
\[
  \frac{\alpha}{\beta} = b_1 - \frac{1}{ b_2 - \ddots - \frac{1}{b_s}}.
\]
We use the notation $[b_1, \ldots, b_s]$ for the right hand side above.
We have
\[
  \ell_0 = \ell,\quad
  \ell_1 = \frac{\beta \ell + \ell'}{\alpha}.
\]
Along with condition \ref{it:two_can_int}, this determines the
canonical primitive sequence recursively and we have $\ell_{s+1} = \ell'$.

\begin{figure}[ht]
\begin{center}
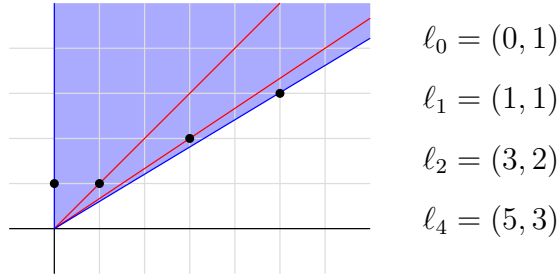
\caption{In this example, $\Sigma$ is generated by $(0,1)$ and $(5,3)$.
The canonical primitive sequence consists of four elements, including the
generators of the cone.}
\label{fig:can_cone}
\end{center}
\end{figure}

Alternatively, the vectors $\ell_0,\ell_1,\ldots,\ell_{s+1}$ are the
integral points lying on compact faces of the convex closure of the
set $\Sigma\cap N \setminus \{0\}$. For a detailed discussion of this
construction, see \cite[1.6]{Oda_book}.

If $\Sigma$ is regular, then we prefer to modify the minimality of
the resolution considered
above, and set $s = 1$, $\ell_1 = \ell$ and $\ell_2 = \ell'$
and $\ell_1 = \ell_0 + \ell_2$.
Accordingly, in \cref{it:two_can_int}, we will have $-b_1 = -1$.
In particular, the set $\{\ell_0, \ell_1, \ell_2\}$ is
not a minimal set of generators of the semigroup $\Sigma\cap N$.
We make this choice here mostly for technical reasons
(directed by properties of the induced reslution), which will
appear in \cref{s:comp_seq}. The same choice is made in
\cite{Oka}, Definition (3.5).
\end{block}

\begin{definition} \label{def:canon_subdiv}
Let $\Sigma$ be a two dimensional rational strictly convex cone with
a canonical primitive sequence $\ell_0, \ell_1, \ldots, \ell_{s+1}$.
The \emph{canonical subdivision} of $\triangle_\Sigma$ is the unique
subdivision $\tilde\triangle_\Sigma$ for which
\[
  \tilde\triangle_\Sigma^{(1)}
    = \set{\R_{\geq 0}\langle\ell_i \rangle}{0\leq i \leq s+1}.
\]
For each $i=1,\ldots, s$, there is a unique number $-b_i \in \Z_{\leq -1}$
satisfying $\ell_{i-1} - b_i\ell_i + \ell_{i+1} = 0$. We define
$\alpha(\ell_0, \ell_{s+1})$ and
$\beta(\ell_0, \ell_{s+1})$ as the numerator and denominator, respectively,
of the negative continued fraction
\[
  [b_1, \ldots, b_s] = b_1 - \frac{1}{b_2 - \frac{1}{\cdots}},
\]
(we require $\gcd(\alpha(\ell_0, \ell_{s+1}),\beta(\ell_0, \ell_{s+1})) = 1$,
and $\beta(\ell_0, \ell_{s+1}) \geq 0$, so that these numbers
are well defined).
The number $\alpha(\ell_0, \ell_{s+1})$ is referred to as the
\emph{determinant} of $\Sigma$.
\end{definition}

\begin{rem} \label{rem:alpha}
Let $\ell_1, \ell_2 \in N$ be two linearly independent elements.
Then we have $\alpha(\ell_1, \ell_2) = 1$ if and only if $\ell_1, \ell_2$
form part of an integral basis of $N$.
In general, $\alpha = \alpha(\ell_1, \ell_2)$ can be computed as the content of
the restriction of $\ell_2$ to the kernel of $\ell_1$. 
In other words, let $K \subset N$ be the kernel of $\ell_1$. Then
$\ell_2|_K$ is divisible by $\alpha$, and $(\ell_2|_K) / \alpha$ is primitive.
\end{rem}

\begin{lemma} \label{lem:surf_res}
If $\Sigma$ is not a regular cone, then $Y_\Sigma$ has a cyclic quotient
singularity at the origin and the map
$Y_{\tilde\triangle_\Sigma} \to Y_\Sigma$ induced by the identity map
on $N$ is the minimal resolution.
\end{lemma}
\begin{proof}
See  Proposition 1.19 and Proposition 1.24 of \cite{Oda_book}.
\end{proof}

\begin{prop} \label{lem:curves}
Assume that $\rk N = 2$, and that $f$ is Newton nondegenerate
with respect to $\Sigma \subset N_\R$ defining a germ $(X,0)$.
\begin{enumerate}

\item \label{it:curves_irred}
The germ $(X,0)$ is irreducible if and only if $\Gamma(f)$ is a single
interval with no integral interior points. In fact, in general,
the number of
components in $(X,0)$ is precisely the combinatorial length of $\Gamma(f)$.

\item \label{it:curves_smooth}
Assume that $(X,0)$ is irreducible and let $\sigma\in\triangle_f^{(1)}$ so
that $\Gamma(f) = F_\sigma$.
Then $(X,0)$ is smooth if and only if $\ell_\sigma$ lies on the boundary
of the convex hull of the set $\Sigma^\circ \cap N$.
In other words,
let $\ell_0, \ldots, \ell_{s+1}$ be the canonical primitive sequence of
$\Sigma$.
Then
either $\ell_\sigma$ is one of $\ell_1, \ldots, \ell_s$,
or there is an $a\in\Z_{>0}$ such that either
\[
  \ell_\sigma = a\ell_0 + \ell_1 \quad or \quad
  \ell_\sigma = a\ell_{s+1} + \ell_s.
\]

\item \label{it:curves_third}
The curve $(X,0)$ is smooth if and only if the following condition
holds: If $p\in M$ and $\ell_\sigma(p) > m_\sigma$ for all
$\sigma\in\triangle_\Sigma^{(1)}$, then $\ell_\sigma(p) > m_\sigma$
for all $\sigma\in\triangle_f^{(1)}$.

\end{enumerate}
\end{prop}
\begin{rem}
One can ask why the vectors $\ell_0$
and $\ell_{s+1}$ do not appear in the list of (ii).
The answer
is that the corresponding divisors $D_\sigma$,
though they intersect $E$ transversaly,
they
are $\T$–invariant, hence they are eliminated by the convention of the
definition 3.6.
\end{rem}
\begin{proof}[Proof of \cref{lem:curves}]
We start with the following observations.
Write $\sigma_i = \R_{\geq 0} \langle \ell_i \rangle$.
Let $\triangle'$ be a regular subdivision
of $\triangle_\Sigma$ which refines both $\triangle_f$ and the canonical
subdivision of $\triangle_\Sigma$. The map $Y_{\triangle'}\to Y_\Sigma$
is then a resolution of $Y_\Sigma$ with exceptional divisor $E'$.
We can write $E' = \cup_{i=1}^{s'}E'_i$,  where each $E_i'$ is a rational
curve. Furthermore, if $i\neq j$, then
$E_i'$ and $E_j'$ intersect if and only if $|i-j| = 1$.
In fact, we can write
\[
  \triangle'^{(1)} = \{ \sigma_1',\ldots,\sigma_{s'}' \}
                     \cup \{ \sigma, \tau \}
\]
where $\sigma,\tau$ are the two faces of $\Sigma$ and
$E_i' = V(\sigma_i')$.

Similarly as in \cite{Oka}, we see that $Y_{\triangle'}$ resolves
$(X,0)$ and that the strict transform $X'$ of $X$ in
$Y_{\triangle'}$ intersects the exceptional divisor $E'$ transversally
in smooth points of $E'$. In fact, these intersection points
lie in the open orbit $O_{\sigma_i'}\subset E'_i$. Therefore, we have
(see \cite[Theorem 5.1]{Oka})
\[
  |X'\cap E'_i| = \chi(X'\cap O_{\sigma'_i}) = \Vol_1(F(\ell_i'))
\]
where $\ell_i'$ is the primitive generator of
$\sigma_i'$. Now, the components of $(X,0)$ are in bijection with the
intersection points $X'\cap E'$, which proves \ref{it:curves_irred}.

For \ref{it:curves_smooth}, let $\tilde\triangle_\Sigma$ be the canonical
subdivision, and $\pi:\tilde Y\to Y$ the associated modification,
which is a resolution of $Y$. Let $\X \subset \tilde Y$ be the strict
transform of $X$. The minimal cycle of the resolution
$\tilde Y \to Y$ is the reduced exceptional divisor $E \subset \tilde Y$
and $(Y,0)$
is rational. By \cite{Artin_iso}, the pullback of the maximal ideal
of $0\in Y$ is the reduced exceptional divisor in $\tilde Y$,
and the maximal ideal has no base points in $\tilde Y$.
It follows that the multiplicity of $(X,0)$ is the intersection number
between $\tilde X$ and $E$. In particular, $(X,0)$ is smooth if and
only if $E\cup \tilde X$ is a normal crossing divisor.
If $\sigma = \sigma_i$ for some $1\leq i \leq s$, then this is indeed the
case. Otherwise, there is an $0\leq i \leq s$ so that
$\ell_\sigma = a\ell_i+b\ell_{i+1}$. In a neighbourhood of
$E_i \cap E_{i+1}$ we have coordinates $u,v$ so that
$E_i = \{x=0\}$, $E_{i+1} = \{y=0\}$ and we have some generic
coefficients $c,d$ so that the strict transform of $X$ is defined by
$cx^b + dy^a$.
Thus, $(X,0)$ is not smooth if
$1<i<s$. In the case $i=0$ (the case $i=s$ is similar), $\tilde X$ is
smooth and transverse to $E_1$ if and only if $b=1$.

The condition in \cref{it:curves_third} is equivalent with the equality
\begin{equation} \label{eq:curve_points}
 (\Gamma^*_+(f)\setminus \Gamma_+(f))\cap M=\emptyset.
\end{equation}
Choose a basis for $N$, inducing an isomorphism $N \cong M$ via the
dual basis, as well as an inner product on $N_\R \cong M_\R$.
If we rotate the segment $\Gamma(f)$ by $\pi/2$ and translate it,
then it can be identified with the vector $\ell_i$
(segment $t\ell_i$, $t\in [0,1]$).
Consider the parallelogram $P(\ell_i)$
whose sides are parallel to $\ell_0$ and $\ell_{s+1}$,
and it has $\ell_i$ as diagonal.
It is divided by $\ell_i$ into two triangles,
each of them can be identified by $\Gamma^*_+(f)\setminus \Gamma_+(f)$.
Hence, \cref{eq:curve_points}
holds if and only if
$P(\ell_i)^\circ\cap N=\emptyset$.

Clearly, $P(\ell_i)^\circ\cap N$ is empty if
$\ell_\sigma \in \partial \convx(\Sigma^\circ \cap N)$.
The converse can be seen as follows. Let
$(\ell^{\mathrm b}_i)_{i\in \Z}$ be a family consisting of integral points on
$\partial\convx(\Sigma^\circ \cap N)$, ordered according one of the
orientation of this boundary. Two consecutive elements of this family
form a basis of $N$, and
\[
  \Sigma^\circ \cap N
  = \bigcup_{i\in \Z}
    \Z_{\geq 0}\langle \ell^{\mathrm b}_i, \ell^{\mathrm b}_{i+1} \rangle
    \setminus\{0\}.
\]
It follows that the set of irreducible elements in the semigroup
$\Sigma^\circ \cap N$ are presicely the elements on the boundary
$\partial\convx(\Sigma^\circ \cap N)$. In particular, if
$\ell_\sigma \in (\convx(\Sigma^\circ \cap N))^\circ$, then
$\ell_\sigma = \ell' + \ell''$ for some
$\ell', \ell'' \in \Sigma^\circ \cap N$. It follows that
$\ell', \ell'' \in P(\ell_i)^\circ$.
\end{proof}

\begin{figure}[ht]
\begin{center}
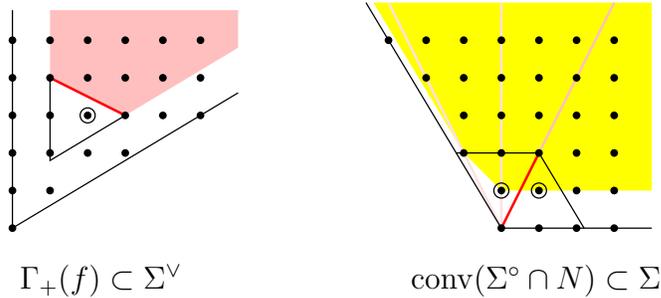
\caption{The integral points in the interior of the
parallellogram $P(\ell_{\sigma})$.}
\label{fig:delta}
\end{center}
\end{figure}

\begin{cor}
Consider  the notation from the proof of \cref{lem:curves}(ii),
that is, $(X,0)$ irreducible and
$\ell_\sigma = a\ell_i+b\ell_{i+1}$ with $\gcd(a,b) = 1$. Then
the multiplicity of $(X,0)$ is
\[
\pushQED{\qed}
  \mult(X,0) =
\begin{cases}
  b     & i = 0,    \\
  a + b & 0 < i < s, \\
  a     & i = s.
\end{cases}
\qedhere
\popQED
\]
\end{cor}

\begin{rem} \label{rem:alpha_beta}
Let $\ell, \ell'$ be any two linearly independent integral vectors in any free $\mathbb Z$ module, and let $N$
be the free $\mathbb Z$ module generated by them. Then the definitions from \ref{block:two_can} and \ref{def:canon_subdiv}
can be repeated in $N$. Then
the determinant of two such vectors can be seen as the greatest
common divisor of the maximal minors of the matrix having the
coordinate vectors of $\ell, \ell'$ as rows, see \cite{Oka}.
Note that   $\alpha(\ell, \ell') = \alpha(\ell', \ell)$. Moreover,
 $\beta(\ell_0, \ell_{s+1})\beta(\ell_{s+1}, \ell_0) \equiv 1
\,(\mod \alpha(\ell_0, \ell_{s+1}))$, cf.
\cite[Proposotion 5.6]{P-P_cfrac}.
\end{rem}

\section{Isolated surface singularities} \label{s:isol}

In the next theorem we give necessary and
sufficient conditions for a Newton nondegenerate surface
singularity to be isolated,
in terms of the Newton polyhedron.
In particular, we assume that $r=3$ in this section.
This is a (non-direct) generalization of a result of Kouchnirenko
valid in the classical case \cite{Kouchnirenko}.

\begin{thm}\label{thm:isol}
Let $(X,0)$ be a Newton nondegenerate singularity and assume $\rk N = 3$.
The following are equivalent
\begin{enumerate}

\item \label{it:isol_isol}
$(X,0)$ has an isolated singularity.

\item \label{it:isol_p}
If $p \in M$ satisfies
$\ell_\sigma(p) > m_\sigma$ for all
$\sigma\in\triangle_\Sigma^{(1)}$, then
$\ell_\sigma(p) > m_\sigma$ for all
$\sigma\in\triangle_f^{*(1)}$.

\item \label{it:isol_pp}
Let $\sigma_1,\sigma_2\in\triangle_\Sigma^{(1)}$ and
$\sigma =
\R_{\geq 0}\langle \sigma_1,\sigma_2\rangle \in \triangle_\Sigma^{(2)}$
and assume that
$\tau\in\triangle_f^{(1)}$ with $\tau \subset \sigma$.
If $p\in M$ so that $\ell_{\sigma_1}(p) > m_{\sigma_1}$ and
$\ell_{\sigma_2}(p) > m_{\sigma_2}$, then $\ell_\tau(p) > m_\tau$.

\item \label{it:isol_e}
Let $\sigma_1,\sigma_2\in\triangle_\Sigma^{(1)}$ and
$\sigma =
\R_{\geq 0}\langle \sigma_1,\sigma_2\rangle \in \triangle_\Sigma^{(2)}$.
Then there is at most one
$\tau\in\triangle_f^{(1)}$ with $\tau \subset \sigma$ and
$\sigma_1 \neq \tau \neq \sigma_2$.
If such a $\tau$ exists, then
$\ell_\tau$ is one of the following
\begin{equation} \label{eq:smooth_tau}
  \ell_1, \ldots, \ell_s, \quad
  a\ell_0 + \ell_1,\quad
  \ell_s + a\ell_{s+1},
  \qquad
  a \in \Z_{\geq 0}
\end{equation}
and, furthermore, there exists an $e \in \Q$ so that
\begin{equation} \label{eq:e}
  e\ell_\tau + \frac{\ell_{\sigma_1} }{ \alpha(\ell_\tau, \ell_{\sigma_1})}
             + \frac{\ell_{\sigma_2} }{ \alpha(\ell_\tau, \ell_{\sigma_2})}
  = 0,
\end{equation}
(see \cref{def:canon_subdiv} for $\alpha(\cdot,\cdot)$) and
\begin{equation} \label{eq:m}
  em_\tau + \frac{m_{\sigma_1} }{ \alpha(\ell_\tau, \ell_{\sigma_1})}
          + \frac{m_{\sigma_2} }{ \alpha(\ell_\tau, \ell_{\sigma_2})}
  = -1.
\end{equation}
\end{enumerate}
\end{thm}

\begin{proof}
By \cref{lem:smooth_open}, the singular locus of the punctured
germ $X\setminus \{0\}$ is a union of orbits $O_\sigma$ for some
$\sigma\in\triangle_\Sigma^{(2)}$.
For such a $\sigma$, we have $(V(\sigma),0)\subset(X,0)$ if and only if
the projection of $\Gamma_+(f)$ in $M_\sigma$ is nontrivial,
by \cref{lem:contains_orig}. By the same lemma,
if $(V(\sigma),0)\subset(X,0)$, then the generic transverse type
to $V(\sigma)$ in $(X,0)$ is a Newton nondegenerate curve with Newton
polyhedron the projection of $\Gamma_+(f)$ to $M_\sigma$.
Therefore \cref{it:isol_isol}$\Leftrightarrow$\cref{it:isol_pp} follows
from \cref{lem:curves}. The equivalence of
\cref{it:isol_p} and \cref{it:isol_pp} is an exercise.

The generic transverse type to $(V(\sigma),0)$ in $(X,0)$ is smooth if and
only if its diagram has a single face corresponding to a $\tau$ as
in \cref{eq:smooth_tau}, and this face has length one.
\cref{it:isol_isol}$\Leftrightarrow$\cref{it:isol_e} follows,
once we show that
given such a $\tau$, an $e\in \Q$ satisfying \cref{eq:e} exists and is
unique, and that, furthermore, the left hand side of
\cref{eq:m} is minus the combinatorial length of the face $F$ of the Newton
diagram corresponding to $\tau$.

Take a smooth subdivision of $\sigma$ containing $\tau$ as a ray, and
let $\tau_i$ be the ray adjacent to $\tau$ between $\tau$ and $\sigma_i$.
Then there exists a $-b \in \Z$ so that
\begin{equation} \label{eq:minus_b}
  -b\ell_\tau + \ell_{\tau_1} + \ell_{\tau_2} = 0.
\end{equation}
Furthermore, for $i=1,2$, we may assume that
\[
  \ell_{\tau_i} = \frac{\beta_i\ell_{\tau} + \ell_{\sigma_i}}{\alpha_i}
\]
where $\alpha_i / \beta_i$ is the continued fraction associated with
$\ell_\tau$ and $\ell_{\sigma_i}$. As a result,
\cref{eq:minus_b} can be rewritten as (\ref{eq:e})
with
$e = -b + \beta_1/\alpha_1 + \beta_2/\alpha_2$.
Let $p_1, p_2$ be the endpoints of $F$ so that
$\ell_{\tau_1}(p_2 - p_1) > 0$ and $\ell_{\tau_2}(p_1 - p_2) > 0$.
Since $\ell_{\tau_i}$ is a primitive function on the affine hull of
the face of $F$,
$\ell_{\tau_1}(p_2 - p_1) =\ell_{\tau_2}(p_1 - p_2) =$ the length
of $F$. We find
\[
\begin{split}
  em_\tau + \frac{m_{\sigma_1}}{\alpha_1}
          + \frac{m_{\sigma_2}}{\alpha_2}
  &= e\ell_\tau(p_1) + \frac{\ell_{\sigma_1}(p_1)}{\alpha_1}
                     + \frac{\ell_{\sigma_2}(p_2)}{\alpha_2} \\
  &= -b\ell_\tau(p_1) + \ell_{\tau_1}(p_1) + \ell_{\tau_2}(p_2)
  = \ell_{\tau_2}(p_2 - p_1).\qedhere
\end{split}
\]
\end{proof}

\section{Resolution of Newton nondegenerate surface singularities} \label{s:res}

In this section, we retain the notation introduced in \cref{s:divisors},
with the assumption that $\rk N = 3$.
We describe Oka's algorithm which describes explicitly the graph of a
resolution of a Newton nondegenerate Weil divisor of dimension $2$.
This algorithm was originally described by Oka \cite{Oka}
for Newton nondegenerate hypersurface singularities in $(\C^3,0)$.
The general methods for resolving Newton nondegenerate hypersurface
singularities have been used in e.g. \cite{Var_zeta} and
\cite[Chapter 8]{AGZVII}.

\begin{definition} \label{def:cansubdiv}
A \emph{canonical subdivision} of $\triangle_f$ is a subdivision
$\tilde\triangle_f$ satisfying the following.
\begin{enumerate}

\item
$\tilde\triangle_f$ is a regular subdivision of $\triangle_f$.

\item
If $\sigma\in\triangle^{(2)}_f \setminus \triangle_f^*$, then
$\tilde\triangle_f|_\sigma$ is the canonical subdivision
$\tilde\triangle_\sigma$ of $\triangle_\sigma$ given
in \cref{def:canon_subdiv}.
\end{enumerate}
\end{definition}

\begin{block}
The existence of a canonical subdivision is proved
in \cite[\S3]{Oka}. We fix such a subdivision
$\tilde\triangle_f$. We will denote by $\tilde Y$ the toric variety
associated with $\tilde\triangle_f$. The map $\tilde Y \to Y$ is denoted by
$\pi$, and the strict transform of $X$ under this map is denoted by $\tilde X$.
We denote by $\pi_X$ the restriction $\pi|_{\tilde X}:\tilde X \to X$.
By \cref{lem:proper}, the map $\tilde Y \to Y$ is proper, hence
$\tilde X \to X$ is proper as well.
\end{block}

\begin{definition} \label{def:Fflmt}
For $i,d\in\N$, define
\[
\begin{split}
  \tilde\triangle_f^{(i,d)}
    &= \set{\sigma\in\tilde\triangle_f^{(i)}}
          {\dim(F_\sigma \cap \Gamma(f)) = d}\\
  \tilde\triangle_f^{*(i,d)}
    &= \tilde\triangle_f^{(i,d)} \cap \tilde\triangle_f^*.
\end{split}
\]
\end{definition}

\begin{definition} \label{def:okas_graph}
We start by defining a graph $G^*$ as follows. Index the set
$\tilde\triangle_f^{(1,2)}$ by a set $\Nd$, i.e. write
$\tilde\triangle_f^{(1,2)} = \set{\sigma_n}{n\in \Nd}$ in such a way that
the map $\Nd\to\tilde\triangle_f^{(1,2)}$, $n\mapsto \sigma_n$ is bijective.
Similarly, index the set
$\widetilde{\Delta}_f^{(1,2)} \cup \widetilde{\Delta}_f^{*(1,1)}$ by
$\mathcal {N}^*$. Hence $\mathcal {N}\subset \mathcal {N}^*$.
The elements of $\mathcal {N}^*$ are referred to as {\it extended nodes},
while ${\mathcal N}$ as {\it nodes}.

Denote by $F_n$ the face of $\Gamma_+(f)$ corresponding to $\sigma_n$
and by $\ell_n$ the primitive integral generator of $\sigma_n$.
Note that $n \in \Nd$ if and only if $F_n$ is bounded.
For $n,n'\in \Nd^*$, let $t_{n,n'}$ be the length of the segment
$F_n\cap F_{n'}$ if this is a bounded segment of dimension 1.
If $F_n\cap F_{n'}$ is unbounded, or has dimension 0,
then we set $t_{n,n'} = 0$.
Now, for every pair $n$ and $n'\in \Nd^*$, we join $n,n'$ by $t_{n,n'}$ bamboos
of type $\alpha(\ell_n, \ell_{n'}) / \beta(\ell_n, \ell_{n'})$, as in
\cref{fig:bamboo}. This finishes the construction of the graph $G^*$.
Denote its set of vertices $\V^*$.

Define the graph $G$ as the induced full subgraph of $G^*$
on the set of vertices
$  \V = \V^* \setminus (\Nd^*\setminus \Nd)$.

\begin{figure}[ht]
\begin{center}
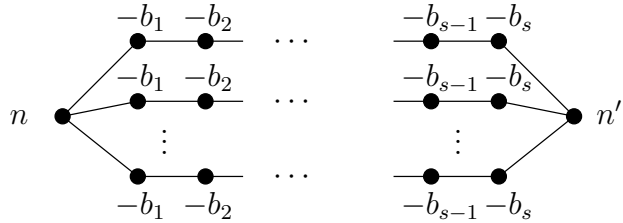
\caption{We join $n,n'\in \Nd$ by $t_{n,n'}$ bamboos of the above form,
where the sequence $b_1, \ldots, b_s$ is defined as $b_1 = 1$ if
$\alpha(\ell_n, \ell_{n'}) = 1$, and by a negative continued
fraction expansion
$\alpha(\ell_n,\ell_{n'}) / \beta(\ell_n,\ell_{n'}) = [b_1,\ldots, b_s]$
otherwise.}
\label{fig:bamboo}
\end{center}
\end{figure}

In order to have a plumbing graph structure on $G$, we must specify
an Euler number and a genus for each vertex, as well as a sign for each
edge. All edges are positive. Vertices appearing on bamboos have genus
zero, whereas the genus $g_n$ associated with $n\in \Nd$ is defined
as the number of integral interior points in the polygon $F_n$.

To every extended node $n\in \Nd^*$ we have associated the cone $\sigma_n$
and its primitive integral generator $\ell_n$. If $v_1, \ldots, v_s$
are the vertices appearing on a bamboo, in this order, from $n$ to
$n'\in\Nd^*$, let $\ell_0, \ell_1, \ldots, \ell_{s+1}$ be
the canonical primitive sequence associated with $\ell_n, \ell_{n'}$.
We then set $\ell_v = \ell_i$ for $v=v_i$, $i=1,\ldots, s$, and
$\sigma_v = \R_{\geq 0}\langle \ell_i \rangle$. This induces
a map $\gamma:\V\to \tilde\triangle_f^{(1)}$ with the property that
$\gamma(n) = \sigma_n$ for $n\in \Nd^*$, and
$\ell_v, \ell_w$ generate an element of $\tilde\triangle_f^{(2)}$
if $v,w$ are adjacent in $G^*$.

For any $v\in \V$, let $\V_v$ and$\V_v^*$ be the set of neighbours of
$v$ in $G$ and $G^*$, respectively.
Then there exists a unique $-b_v\in \Z_{\leq -1}$ satisfying
\[
  -b_v \ell_v + \sum_{u\in \V_v^*} \ell_u = 0 \ \ \mbox{in} \ \   N,
\]
The number $-b_v$ is the Euler number associated with $v \in \V$.
We note that if $v$ lies on a bamboo, with the notation of the previous
paragraph, $v = v_i$, then $-b_v = -b_i$ and $-b_i \leq -2$ unless
$\alpha(\ell_n, \ell_{n'}) = 1$.
\end{definition}

\begin{rem} \label{rem:qhs3_oka}
The link of an isolated surface singularity is a
rational homology sphere if and only if it has a resolution whose
graph is a tree and all vertices have genus zero, see e.g.
\cite{Nemethi_resolution}. The above construction produces such
a graph if and only if all integral points on $\Gamma(f)$ lie on
its  boundary $\partial \Gamma(f)$.

Indeed,
if $P \subset \Gamma(f)$ is a vertex which is not on the boundary, then
the nodes corresponding to faces of $\Gamma(f)$ containing $P$ lie
on an embedded cycle.
Similarly,
if $S \subset \Gamma(f)$ is a face of dimension $1$ which is
not a subset of the boundary, and $S$ contains integral interior points, then
the nodes corresponding to the two faces containing $S$ are joined by
more than one bamboo, inducing an embedded cycle in $G$.
Finally, if $F \subset \Gamma(f)$ is a two dimensional
face containing interior integral interior points, then the corresponding node
has nonzero genus. The converse is not difficult.

The classical case $Y = \C^3$ is discussed in details in
\cite{Newton_nondeg}.
\end{rem}

\begin{example}
Let $\Sigma = \R^3_{\geq 0}$, and consider standard coordinates
$x,y,z$ on $Y = \C^3$, and the function
\[
  f(x,y,z) = x^5 + x^2y^2 + y^7 + z^{10}.
\]
The Newton diagram $\Gamma(f)$ consists of two triangular faces, whose
intersection is a segment of length two. The diagram, as well as the
graph obtained by Oka's algorithm can be seen in \cref{fig:ex6}.
\begin{figure}[ht]
\begin{center}
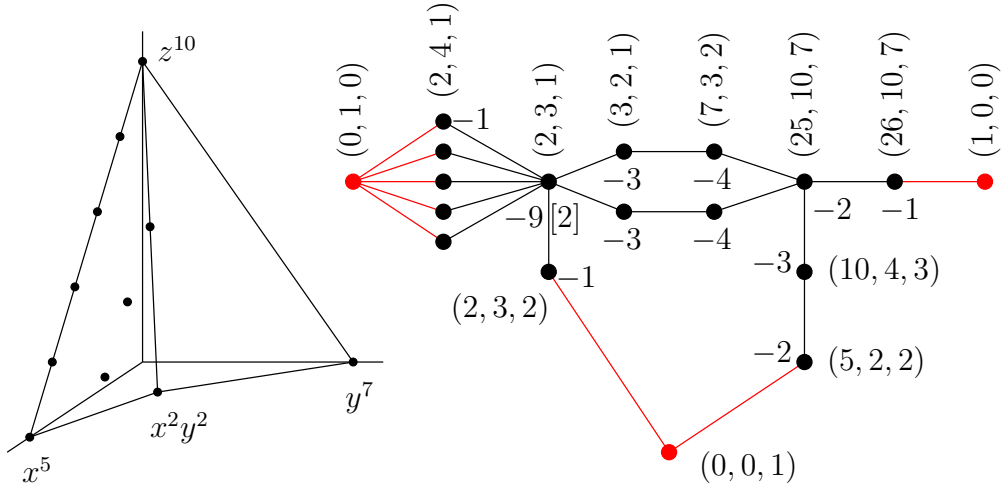
\caption{A Newton diagram, and the graph $G^*$, with the subgraph
$G$ in black.}
\label{fig:ex6}
\end{center}
\end{figure}
\end{example}

\begin{prop}
Let $(X,0)$ be a Newton nondegenerate surface singularity. Then
the map $\X\to X$ is a resolution of $(X,0)$ whose
resolution graph is $G$.

More precisely, $\tilde X$ is smooth and
the exceptional set $E \subset \tilde X$
is a normal crossing divisor.
For each $\sigma\in\tilde\triangle_f^{(1)}$, we can enumerate the
irreducible components of $E_\sigma$ by $\gamma^{-1}(\sigma)$ so that
$E_\sigma = \amalg_{v \in \gamma^{-1}(\sigma)} E_v$, where $E_v$ is a smooth
curve.

If $\gamma(v) \in \tilde\triangle_f^{(1)} \setminus \tilde\triangle^*$, then
$E_v$ is compact, has
genus $g_v$, and its normal bundle in $\tilde X$ has Euler number $-b_v$.
If $\gamma(v) \in \tilde\triangle^{(1)*}$, then
$E_v$ is a smooth germ, transverse to a smooth point of the exceptional
divisor.

Furthermore, if $v,w\in\V$, then the number of intersection points
$|E_v\cap E_w|$ equals
the number of edges between $v$ and $w$ in $G$.
\end{prop}
\begin{proof}
The proof goes exactly as in \cite{Oka}
\end{proof}

\begin{definition}
For $v\in \V^*$, (recall \cref{block:fan_comb} and \cref{def:Fflm}) let
\[
  F_v = F_{\gamma(v)},\quad
  \ell_v = \ell_{\gamma(v)},\quad
  m_v = m_{\gamma(v)}.
\]
\end{definition}

\begin{lemma} \label{lem:bv}
For $v\in \V$, we have
\[
  -b_v \ell_v + \sum_{u\in\V^*_v} \ell_u = 0,\quad
  -b_v m_v + \sum_{u\in\V^*_v} m_u = -2\Vol_2(F_v).
\]
\end{lemma}
\begin{proof} The first equality follows from construction, see also
\cite[\S6]{Oka}. The second equality follows from
\cite[Prop. 4.4.4]{Newton_nondeg} and the formula
$\alpha \ell_1 = \beta \ell_0 + \ell_{s+1}$, where
$\ell_0,\ell_1,\ldots, \ell_{s+1}$ is a primitive sequence.
\end{proof}

\begin{rem} \label{rem:exceptional_div}
\begin{blist}

\item
The exceptional divisor $E$ is the union of $E_\sigma$ for which
$\sigma \in \tilde\triangle_f^{(1)}$ is a cone which is not contained in
$\partial\Sigma$, or, equivalently, $F_\sigma$ is compact.

\item
If $\sigma \in \tilde\triangle_f^{(1,2)}$,
then $E_\sigma$ is a compact smooth irreducible curve.
If $\sigma \in \tilde\triangle_f^{(1,1)} \setminus \tilde\triangle^*_f$,
then $E_\sigma$ is the union of $t_\sigma$ disjoint smooth compact rational curves.
For $\sigma \in \tilde\triangle_f^{*(1,1)}$, the intersection
$E_\sigma = V(\sigma) \cap \X$ is the disjoint union of $t$
smooth curve germs, where $t$ is the length of the segment
$F_\sigma \cap \Gamma(f)$.
If $\sigma \in \tilde\triangle_f^{(1,0)}$,
then $E_\sigma = \emptyset$ (the global divisor $D_\sigma$ does not intersect $\X$).
\end{blist}
\end{rem}

\begin{definition}
We denote by $L =  \Z\gen{ E_v }{v\in\V}$ the lattice of integral
cycles in $\tilde X$ supported on the exceptional divisor $E$.
\end{definition}

\begin{definition} \label{def:wt_div}
Let $g \in \O_{Y,0}$ and denote its restriction by $\overline g \in \O_{X,0}$.
For any $v\in\V^*$, we define
\[
\begin{split}
   \wt_v  (g)      = \min\set{\ell_v(p)}{p\in\supp(g)},  \quad
  &\wt(g) = \sum_{v\in\V} \wt_v(g) E_v \in L,          \\
   \wt_v  (\overline g) = \max\set{\wt_v(g+h)}{h\in I_X},     \quad
  &\wt(\overline g) = \sum_{v\in\V} \wt_v(\overline g) E_v \in L.
\end{split}
\]
For $\sigma = \gamma(v)$, we also write $\wt_\sigma$ instead of
$\wt_v$, as this is independent of $v\in\gamma^{-1}(\sigma)$.

Similarly, for any $v\in\V$, let $\div_v$ be the valuation on
$\O_{X,0}$ associated with the divisor $E_v$, that is, for
$\overline g\in\O_{X,0}$, denote by $\div_v(\overline g)$
the order of vanishing of the
function $\pi_X^*(\overline g)$ along $E_v$. Set also
\[
  \div(\overline g) = \sum_{v\in\V} \div_v(\overline g) E_v \in L.
\]
\end{definition}

\begin{rem}
\begin{blist}

\item
If $\sigma = \gamma(v)$ and $|\gamma^{-1}(\sigma)| > 1$, then
$\div_v$ is not independent of the choice of $v\in\gamma^{-1}(\sigma)$.

\item
For $\sigma\in\tilde\triangle_f^{(1)}$, the function
$\wt_\sigma:\O_{Y,0} \setminus \{0\} \to \Z$ is the valuation on
$\O_{Y,0}$ associated with the irreducible divisor $V(\sigma) \subset \tilde Y$, cf. \cref{eq:van_order}.

\item
In general,
the functions $\wt_v$ and $\div_v$ do not coincide on $\O_{X,0}$. However,
$\wt_v(\overline{g})\leq \div_v(\overline g)$ for any $\overline{g}\in \O_{X,0}$ and $v\in \V$.
 Furthermore,
if  $p\in M$ and
$\gamma(v) \in \tilde\triangle_f^{(1,>0)} \setminus \tilde\triangle_f^*$,
then $\div_v(x^p) = \wt_v(x^p) = \ell_v(p)$.
In particular, this defines a group homomorphism $M \to L$,
$p\to \wt(x^p)$.

\end{blist}
\end{rem}

\section{The geometric genus} \label{s:geom_genus}

In this section we provide a formula for the delta invariant and geometric
genus for an arbitrary generalized Newton nondegenerate singularity in
terms of its Newton polyhedron.
In this section, the rank $r$ of $N$ is under no restriction.
Recall that
we say that $f$ (or $\Gamma_+(f)$)
is \emph{pointed} at $p\in M_\Q$, if for any
$\sigma \in \triangle_\Sigma^{(1)}$ we have $m_\sigma = \ell_\sigma(p)$,
see \cref{def:pointed}.

\begin{rem}
In the proof of \cref{thm:geom_genus}, one of the main steps consists
of computing the cohomology of a line bundle on a toric variety. To do this,
we build on classical methods \cite{Fulton_toric,Danilov_toric}.
A more general method to compute such cohomology has been described
by Altmann and Ploog in \cite{Altmann_Ploog_coh}.
\end{rem}

\begin{definition} \label{def:geom_genus}
For a point $x$ in an analytic variety $X$, denote by $\overline\O_{X,x}$ the
normalization of its local ring $\O_{X,x}$. The \emph{delta invariant} associated with
$x\in X$ is defined as
\[
  \delta(X,x) = \dim_{\C} \overline\O_{X,x} / \O_{X,x}.
\]
Let $\X\to X$ be a resolution of the singularity $x\in X$ and assume that
$X$ has dimension $d$.
Assume, furthermore, that $\delta(X,x) < \infty$, and that the higher
direct image sheaves $R^i\pi_* \O_{\X}$, $i>0$, are concentrated at $x$.
The \emph{geometric genus} $p_g = p_g(X,0)$ is defined as
\[
  (-1)^{d-1} p_g(X,x)
    = \delta(X,x) + \sum_{i=1}^{d-1} (-1)^i h^i(\X, \O_{\X}).
\]
We say that $(X,x)$ is \emph{rational} if $\delta(X,x) = 0$ and
$h^i(\X,\O_{\X}) = 0$ for $i>0$.
\end{definition}

\begin{thm} \label{thm:geom_genus}
Let $(X,0) \subset (Y,0)$ be a Newton nondegenerate Weil divisor
of dimension $d = r-1$.
\begin{enumerate}

\item \label{it:geom_genus}
We have the following canonical identifications
\[
\begin{split}
  \overline\O_{X,0}/\O_{X,0}
  &\cong
  \bigoplus_{p\in M}
    \tilde H^0(\Gamma_+(x^pf) \setminus \Sigma^\vee, \C),\\
  H^i(\X, \O_{\X})
  &\cong
  \bigoplus_{p\in M}
    \tilde H^i(\Gamma_+(x^pf) \setminus \Sigma^\vee, \C),
  \quad
  i>0.
\end{split}
\]
In particular, if these vector spaces have finite dimension, then
\[
\begin{split}
  \delta(X,0)
  &=
  \sum_{p\in M} \tilde h^0(\Gamma_+(x^pf) \setminus \Sigma^\vee, \C),\\
  p_g(X,0)
  &=
  (-1)^{d-1}\sum_{p\in M} \tilde \chi(\Gamma_+(x^pf) \setminus \Sigma^\vee, \C),
\end{split}
\]
where $\tilde\chi$ denotes the \emph{reduced Euler characteristic},
that is, the
alternating sum of ranks of reduced singular cohomology groups.

\item \label{it:geom_genus_d-1}
We have
\[
  \tilde h^{d-1}(\Gamma_+(x^p f) \setminus \Sigma^\vee, \C) =
\begin{cases}
  1 &  \textrm{if} \quad 0 \in \Gamma_+^*(x^pf)^\circ \setminus
                               \Gamma_+(x^pf)^\circ,\\
  0 &  \textrm{else}.
\end{cases}
\]
In particular,
$h^{d-1}(\X, \O_{\X}) =
|M\cap \Gamma_+^*(f)^\circ \setminus\Gamma_+(f)^\circ|$
(recall \cref{def:Gamma_star}).

\item \label{it:geom_genus_pointed}
Assume that $f$ is $\Q$-pointed, that $d\geq 2$,
and that $(X,0)$ has only rational singularities outside the origin.
Then $(X,0)$ is normal and
$h^i(\X, \O_{\X}) = 0$ for $1\leq i < d-1$.
\end{enumerate}
\end{thm}

\begin{cor} \label{cor:pg_normal_surface}
Assume that  $d=2$ and $(X,0)$  is normal. Then
\[
\pushQED{\qed}
  p_g(X,0) = |M \cap \Gamma_+^*(f)^\circ \setminus \Gamma_+(f)^\circ|.\qedhere
\popQED
\]
\end{cor}

This  generalizes a result of Merle and Teissier
\cite{Merle_Teissier} valid for  the
classical case $\Sigma={\mathbb R}_{\geq 0}^3$.

\begin{cor} \label{cor:partition}
Assume that $d = 1$ and  $(X,0)$ is an irreducible germ of a curve,
and that $\sigma\in \tilde\triangle^{(1)}_f $ satisfies $F_\sigma=\Gamma(f)$
(cf. \cref{lem:curves}\cref{it:curves_irred}).
Then $\delta(X,0)$ is the number of unordered pairs
$\ell', \ell'' \in \Sigma^\circ \cap N$ satisfying
$\ell' + \ell'' = \ell_\sigma$.
\end{cor}
\begin{proof}
Let $P(\ell_\sigma)$ be the parallelogram introduced in the proof of
\cref{lem:curves}.
The diagonal splits $P(\ell_\sigma)$ into two triangles,
$T_1$ and $T_2$, say.
If $\ell' \in T_1^\circ$, then $\ell_\sigma - \ell' \in T^\circ_2$.
This induces a bijection between elements $\ell' \in T^\circ_1 \cap N$
and unordered pairs $\{\ell', \ell''\} \subset \Sigma^\circ \cap N$
adding up to $\ell_\sigma$. By rotating by $\pi/2$ as in the proof of
\cref{lem:curves},  $T_1^\circ\cap N$ is  in bijection
with $M \cap \Gamma_+^*(f)^\circ \setminus \Gamma_+(f)^\circ$.
\end{proof}

\begin{rem}
Assume that $d \geq 2$, and that $X$ is rational outside $\{0\}$.
Then, for $0<i<d-1$, we have
\[
  H^i(\X,\O_{\X})
  \cong
  H^i(\X\setminus E,\O_{\X})
  \cong
  H^i(X\setminus \{0\}, \O_X)
  \cong
  H^{i+1}_{\{0\}}(X,\O_X).
\]
Here, the first isomorphism comes from the long exact sequence for
cohomology with support in $E$, and the vanishing
$H^i_E(\X, \O_{\X}) = 0$, for $i < d$
\cite[Corollary 3.3]{Karras_duality}.
The second isomorphism follows from the rationality assumption, and
the Leray spectral sequence.
The third isomorphism comes from
the similar long exact sequence for cohomology with support in $\{0\}$,
and the fact that $H^j(X,\O_X) = 0$ for $j > 0$,
if we choose a Stein representative $X$ of the germ $(X,0)$.
This last long exact sequence furthermore gives
\[
  H^1_{\{0\}}(X,\O_X)
  \cong
  \frac{H^0(X\setminus \{0\}, \O_X)}{H^0(X, \O_X)}
  \cong
  \overline\O_{X,0} / \O_{X,0}.
\]
Therefore, in this case, the groups described in \cref{thm:geom_genus}
are closely related with the depth of $\O_{X,0}$.
In particular, the conclusion of
\cref{thm:geom_genus}\cref{it:geom_genus_pointed} is that
$(X,0)$ is a Cohen--Macaulay ring.

If $f$ is pointed at $p\in M$, then this statement can be proved as follows.
Since $(X,0)$ is a Cartier divisor in $(Y,0)$,
cf.  \cref{prop:pg_Cartier}\cref{it:pg_Cartier},
and $(Y,0)$ is Cohen--Macaulay
\cite[Theorem 6.3.5]{Bruns_Herzog}
so is $(X,0)$
\cite[Theorem 2.1.3]{Bruns_Herzog}.
\end{rem}

\begin{proof}[Proof of \cref{thm:geom_genus}]
To prove \cref{it:geom_genus}, we use results and notation from
\cite[\S7]{Danilov_toric}, see also \cite[3.5]{Fulton_toric}.
Define
\[
  D_m = \sum \set{ m_\sigma D_\sigma}{\sigma\in\tilde\triangle_f^{(1)}}.
\]
Then $D_m + \X$ is the divisor of the pullback of $f$ to $\tilde Y$,
and we have a short exact sequence
\[
  0
  \to \O_{\Yt}(D_m)
  \stackrel{\cdot f}{\to} \O_{\Yt}
  \to \O_{\X}
  \to 0.
\]
By \cite[Corollary 7.4]{Danilov_toric}, we have $H^i(\Yt,\O_{\Yt}) = 0$
for all $i>0$.
Furthermore,
$H^0(\X,\O_{\X}) \cong \overline\O_{X,0}$, and the image of
$H^0(\Yt,\O_{\Yt}) = \O_{Y,0}$ in $\overline\O_{X,0}$
is $\O_{X,0}$. Therefore,
\[
\begin{split}
  \overline\O_{X,0} / \O_{X,0}
  &\cong
  H^1(\Yt, \O_{\Yt}(D_m)), \ \ \mbox{and} \\
  H^i(\X,\O_{\X})
  &\cong
  H^{i+1}(\Yt, \O_{\Yt}(D_m)),
  \quad
  i>0.
\end{split}
\]
Denote by $g$ the \emph{order function} defined in
\cite[\S6]{Danilov_toric} (using the natural trivialization of
$\O_{\Yt}(D_m)$ on the open torus)
\[
  g:|\tilde\triangle_f| \to \R,\quad
  g(\ell) = - \min\set{ \ell(q)}{q \in \Gamma_+(f)}
\]
and define the sets
\[
  Z_p = \set{\ell \in |\tilde\triangle_f|}
            {\ell(p) \geq g(\ell)},
  \quad p\in M.
\]
We note that $Z_p$ is a convex cone and that $0\in Z_p$ for all $p\in M$.
By \cite[Theorem 7.2]{Danilov_toric}, we have isomorphisms
\[
  H^{i+1}(\Yt,\O_{\Yt}(D_m))
  \cong
  \bigoplus_{p\in M} H^{i+1}_{Z_p}(|\tilde\triangle_f|, \C).
\]
Since $|\tilde\triangle_f| = \Sigma$ is a convex set, the long exact
sequence associated with cohomology with supports provides, for any
$p\in M$
\[
  0
  \cong \tilde H^i(|\tilde\triangle_f|,\C)
  \to   \tilde H^i(|\tilde\triangle_f|\setminus Z_p,\C)
  \cong H^{i+1}_{Z_p}(|\tilde\triangle_f|,\C)
  \to   \tilde H^{i+1}(|\tilde\triangle_f|,\C)
  \cong 0.
\]

To finish the proof of \ref{it:geom_genus}, we will show that for any $p\in M$,
the spaces
$|\tilde\triangle_f|\setminus Z_p = \Sigma\setminus Z_p$ and
$\Gamma_+(x^pf) \setminus \Sigma^\vee$ are in fact homotopically equivalent.
We start by noting that the
the condition
$Z_p \subset \partial|\tilde\triangle_f|$ (including the case when $Z_p=\emptyset$)
is equivalent to
$0 \in \Gamma_+(x^pf) \setminus \Gamma(x^pf)$. If this happens then we can choose
a $q\in \Sigma^\vee$ small so that
$-q \in \Gamma_+(x^pf) \setminus \Gamma(x^pf)$ as well,
and so
$\Gamma_+(x^pf) \setminus \Sigma^\vee$ is star-shaped with
center $-q$. In particular, in this case,
\[
  \Sigma\setminus Z_p
  \sim \{\mbox{a point}\} \sim
  \Gamma_+(x^pf) \setminus \Sigma^\vee,
\]
where $\sim$ denotes the homotopy equivalence.
Thus, in what follows, we assume that $Z_p$ contains an interior
point in $\Sigma$, equivalently,
$0 \notin \Gamma_+(x^pf) \setminus \Gamma(x^p f)$.

Choose $\ell_0 \in \Sigma^\circ$ and $q_0 \in (\Sigma^\vee)^\circ$
satisfying $\ell_0(q_0) = 1$ and define the hyperplanes
\[
  H = \set{\ell \in N_\R}{\ell(q_0) = 1},\quad
  H^\vee = \set{q \in M_\R}{\ell_0(q) = 1}.
\]
Then, seeing $H$ and $H^\vee$ as linear spaces by choosing origins
$\ell_0, q_0$,
the pairing $H\times H^\vee \ni (\ell, q) \mapsto \ell(q) - 1$ is
nondegenerate and the polyhedrons $H\cap \Sigma$ and $H^\vee\cap\Sigma^\vee$
are each others polar sets as in \cite[1.5]{Fulton_toric}.

Since $0\in Z_p$, we have
\[
  \Sigma \setminus Z_p
  \,\sim\,
  (H \cap \Sigma \setminus Z_p) \times \R
  \,\sim\,
  H \cap \Sigma \setminus Z_p.
\]
By the assumptions made above, there is an $\ell \in Z_p \cap \Sigma^\circ$.
Both $\Sigma \cap H$ and $Z_p\cap H$ are compact convex polyhedrons in $H$.
Projection away from $\ell$ onto $\partial(H\cap \Sigma)$ then induces
a homotopy equivalence
\[
  H\cap \Sigma \setminus Z_p \sim H \cap \partial \Sigma \setminus Z_p.
\]
By projection, we mean that any element in a ray
$r = \ell + \R_{>0}\ell' \subset H$
maps to the unique element in $r\cap \partial(H\cap \Sigma)$.
By \cref{lem:subcx}\cref{it:subcx_setminus}, this has the subset
\[
  \cup \set{ H \cap \sigma}
           {\sigma \in \triangle_f^*,\, H\cap \sigma \cap Z_p = \emptyset}
\]
as a strong deformation retract. All this yields
\begin{equation} \label{eq:ht_N}
  \Sigma \setminus Z_p
  \sim
  \cup\set{H\cap\sigma}
          {\sigma \in \triangle_f^*,\,
           \sigma\cap Z_p = \{0\}}.
\end{equation}

Using a projection, this time onto
$\partial \Sigma^\vee$ in $M$, having as center any element in
$(\Sigma^\vee \cap \Gamma_+(x^pf))^\circ$, we get a homotopy equivalence
\[
  \Gamma_+(x^p f) \setminus \Sigma^\vee
  \sim
  \Gamma_+(x^p f)^\circ \cap \partial\Sigma^\vee.
\]
By \cref{lem:subcx}\cref{it:subcx_cap}, we have a homotopy equivalence
\[
  \Gamma_+(x^p f)^\circ \cap \partial\Sigma^\vee
  \sim
  \cup \set{(\sigma^\perp \cap \Sigma^\vee)^\circ}
           {\sigma \in \triangle_\Sigma,\,
            \sigma \neq \{0\},\,
            (\sigma^\perp\cap\Sigma^\vee)^\circ \cap \Gamma_+(x^p f)^\circ
            \neq \emptyset}.
\]
Since, by assumption made above, $0 \notin \Gamma_+(x^p f)^\circ$,
and so the right hand side above has a free action by $\R_{>0}$
which has a section
given by intersection with $H^\vee$.
Furthermore, one checks that if $\sigma \in \triangle_f^*$, then
\[
  (\sigma^\perp\cap\Sigma^\vee)^\circ \cap \Gamma_+(x^p f)^\circ \neq \emptyset
  \quad\Leftrightarrow\quad
  \fa{\ell\in \sigma\setminus \{0\}}{\ell(p)+m_\ell < 0}.
\]
Here, the condition on the left is equivalent to
$\sigma\cap Z_p = \{0\}$, so
\begin{equation} \label{eq:ht_M}
  \Gamma_+(x^p f) \setminus \Sigma^\vee
  \sim
  \cup\set{H\cap(\sigma^\perp\cap\Sigma^\vee)^\circ}
          {\sigma \in \triangle_f^*,\,
           \sigma\cap Z_p = \{0\}}.
\end{equation}

Now, consider the CW structure $K$ given by the cells
$H\cap \sigma$ in $H\cap \partial \Sigma$ and $K'$ given by
cells $H^\vee \cap (\sigma^\perp\cap\Sigma^\vee)$
in $H^\vee \cap \partial \Sigma^\vee$.
Using barycentric subdivision, one obtains a homeomorphism
$\phi:H\cap \partial \Sigma \to H^\vee \cap \partial \Sigma^\vee$,
sending the center of a cell $H\cap \sigma$ to the center of the dual
cell $H\cap \sigma^\vee$, thus identifying $K$ with the dual of $K'$.
By this identification, the left hand side of \cref{eq:ht_M}
is a regular neighbourhood around the image under $\phi$ of the
left hand side of \cref{eq:ht_N}. This concludes
\ref{it:geom_genus}.

Next, we prove \ref{it:geom_genus_d-1}. By the above discussion,
the result is clear in the
cases  when $Z_p=\emptyset$ or  $Z_p \subset \partial \Sigma$.
Assuming that this
is not the case, the complex, say, $A$, on the right
hand side of \cref{eq:ht_N} is a closed subset of
$H \cap \partial \Sigma \sim S^{d-1}$. Then $h^{d-1}(A,\C) = 0$, unless
$A = H \cap \partial \Sigma$, in which case $h^{d-1}(A,\C) = 1$. But this
is equivalent to $\ell(p) + m_\ell < 0$ for all
$\ell \in \partial\Sigma \setminus 0$, that is,
$0\in\Gamma_+^*(x^p f)^\circ$.

For \ref{it:geom_genus_pointed}, we will show that
$\Gamma_+(x^pf) \setminus \Sigma^\vee$ has trivial homology in degrees
$i<d-1$ for all $p\in M$.
By assumption, there is a $q \in M_\Q$ so that for
$\sigma \in \triangle_\Sigma^{(1)}$ we have $m_\sigma = \ell_\sigma(q)$.
We can again assume that
$0 \in \Gamma_+(x^p f) \setminus \Gamma(x^p f)$. We must show that
$\tilde h^i(A,\C) = 0$ for $i<d-1$, where $A$ is the right hand side of
\cref{eq:ht_N}. We note that by definition, $A$ consists of cells
$H\cap \sigma$ for $\sigma\in\triangle_f^*$ satisfying
$\fa{\ell\in H\cap \sigma}{\ell(p) < -m_\ell}$.
Define similarly
\[
  A_\Sigma = \cup \set{H\cap \sigma}
                {\sigma\in \triangle_\Sigma^*,
                 \fa{\ell\in H\cap \sigma}
                    {\ell(p) < -\ell(q)}}.
\]
Define
\[
  A_q= \set{\ell\in H\cap \partial \Sigma}
           {\ell(p) < -\ell(q)}.
\]
This space can be either $S^{d-1}$, an $d-1$ dimensional
ball, or empty. In each case, $\tilde H^i(A_q,\C) =0$ for $i<d-1$.
We will show that $A_q \supset A_\Sigma \subset A$, and
that these inclusions
are homotopy equivalences. For the first one, in fact, this is clear
by definition and \cref{lem:subcx}\ref{it:subcx_setminus}.

For the second one, denote by $A_\Sigma^i$ the $i$-skeleton of the complex
$A_\Sigma$, and define similarly
\[
  A^i = A \setminus \cup
        \set{\sigma^\circ}
            {\sigma\in \triangle_\Sigma^{*(\geq i+2)}}.
\]
We will prove by induction on $i$ that $A^i_\Sigma \subset A^i$ and that
this is a homotopy equivalence. The case $i = 0$ follows from the pointed
condition: assuming $\sigma\in\triangle_\Sigma^{*(1)}$ is a ray, there is
a $t>0$ so that $H\cap\sigma = \{t\ell_\sigma\}$. By assumption, we have
$m_{\ell_\sigma} = \ell_\sigma(q)$,
so that $H\cap \sigma\subset A_\Sigma$ if and only if $H\cap \sigma \subset A$.
Since $A^0$ consists only of such zero-cells, we get
$A_\Sigma^0 = A^0$.

Next, assume that for some $i> 0$ we have an inclusion
$A^{i-1}_\Sigma \subset A^{i-1}$ which is a homotopy equivalence.
Let $\sigma \in \triangle_\Sigma^{*(i+1)}$
provide an $i$-cell $H\cap\sigma$ in $A_\Sigma$.
In this case, we want to show that $H\cap \sigma \subset A^i$.
In fact, we have $\partial(H\cap\sigma) \subset A^{i-1}_\Sigma$, hence
$\partial(H\cap\sigma) \subset A^{i-1}$, by induction.
But by the rationality assumption on the transverse type, it follows from
\ref{it:geom_genus_d-1} and \cref{lem:contains_orig} that we must have
$\sigma \subset A^i$, thus $A_\Sigma^i \subset A^i$.

To show that this inclusion is a homotopy equivalence, let
$\sigma \in \triangle_f^{*(i+1)}$ provide an $i$-cell $H\cap \sigma$
which is not in $A_\Sigma^i$. By definition, we see that
$\sigma \not\subset A^i$ as well. In fact, similarly as in the proof of
\ref{it:geom_genus}, the inclusion
$\partial(H\cap \sigma) \cap A^i \subset (H\cap \sigma) \cap A^i$
is a strong deformation retract. Since these cells, along with
$A_\Sigma^i$ provide a finite closed covering, these glue
together to form a strong deformation retract
$A^i \to A_\Sigma^i$.
\end{proof}

\begin{lemma} \label{lem:subcx}
Let $K,L \subset \R^N$. Assume that $K$ is given as a finite disjoint union
$K=\cup_{\alpha\in I} K_\alpha$ of relatively open convex
polyhedrons $K_\alpha$,
i.e. each $K_\alpha$ is given by a finite number of affine equations and strict
inequalities. Furthermore, assume the following two conditions:
\begin{itemize}
\item
If $F$ is the face of
$\overline K_\alpha$ for some $\alpha$, then $F = \overline K_\beta$ for
some $\beta$.
\item
For any $\alpha,\beta$, the intersection
$\overline K_\alpha \cap \overline K_\beta$ is a face of both
$\overline K_\alpha$ and $\overline K_\beta$.
\end{itemize}
Note that the polyhedrons $K_\alpha$ may be unbounded. In this case
\begin{enumerate}

\item \label{it:subcx_setminus}
Assume that $K$ is compact and $L$ is convex. Then the inclusion
\begin{equation} \label{eq:subcx_setminus}
  \bigcup_{\alpha\in I}
  \set{\overline K_\alpha}{\overline K_\alpha \cap L = \emptyset}
  \subset
  K \setminus L
\end{equation}
is a strong deformation retract.

\item \label{it:subcx_cap}
Assume that $L$ is convex. Then the inclusion
\[
  \bigcup_{\alpha\in I}
  \set{K_\alpha}{K_\alpha \cap L \neq \emptyset}
  \subset
  K \cap L
\]
is a strong deformation retract.

\end{enumerate}
\end{lemma}
\begin{proof}
We prove \cref{it:subcx_setminus}, similar arguments work for
\ref{it:subcx_cap}.
We use induction on the number of $\alpha$ with
$K_\alpha \cap L \neq \emptyset$. Indeed, if this number is zero,
then the inclusion in \cref{eq:subcx_setminus} is an equality.

Otherwise, there is an $\alpha_0$ with $K_{\alpha_0} \cap L \neq \emptyset$.
Define
\[
  I' = \set{\alpha \in I}
           {\overline K_\alpha \not\supset \overline K_{\alpha_0}}
     \subsetneq I,
  \quad
  K' = \cup_{\alpha \in I'} K_\alpha.
\]
Then the left hand side of \cref{eq:subcx_setminus} does not change if we
replace $I$ by $I'$. Therefore, using the induction hypothesis, it is enough
to show that the inclusion $K'\setminus L \subset K\setminus L$ is a
homotopy equivalence.
We do this by constructing a deformation retract
$h:K\setminus L \times [0,1] \to K\setminus L$.
For this, we use the finite closed covering
$\overline K_\alpha \setminus L$, $\alpha \in I$ of $K\setminus L$.
It is then enough to define the restriction $h_\alpha$ of $h$ to
$(\overline K_\alpha \setminus L)\times [0,1]$ for $\alpha \in I$ in such
a way that these definitions coincide on intersections.

For any $\alpha \in I'$, we define $h_\alpha(x,t) = x$.
Let $q\in K_{\alpha_0} \cap L$.
If $\alpha \in I \setminus I'$, then
$q \in \overline K_\alpha$, and we define $h_\alpha$ by projecting away from
$q$, that is, for any $x\in K_\alpha$ there is a unique
$y$ in the intersection of
$\partial \overline K_\alpha \setminus K_{\alpha_0}$ and they ray
starting at $q$ passing through $x$. We define
$h_\alpha(x,t) = (1-t)x + ty$.
One readily verifies that these functions
are continuous, agree on intersections of their domains and define
a strong deformation retract.
\end{proof}

\section{Canonical divisors and cycle} \label{s:can}

In this section we describe possible  canonical divisors for
$\tilde Y = Y_{\tilde \triangle_f}$ and $\X$.
Furthermore, in the case $d=2$, we give a formula for the canonical cycle.

\begin{definition}
Let $\X \to X$ be a resolution of singularities of an
$(r-1)$-dimensional singularity.
A \emph{canonical divisor} $K_{\X}$ on $\X$ is any divisor satisfying
$\O_{\X}(K_{\X}) \cong \Omega_{\X}^{r-1}$.

If $r=3$ then  let $E = \cup_{v\in\V}  E_v$ be the exceptional divisor of a
resolution $\X \to X$, where $E_v$ are the irreducible  components of $E$.
 Recall that we denoted by $L$  the lattice of integral
cycles in $\tilde X$ supported on the exceptional divisor $E$: that is,   $ L=\Z\gen{ E_v }{v\in\V}$.
We also set $L_\Q = L\otimes\Q$ and
\[
  L' = \Hom(L,\Z) \cong \set{l'\in L_\Q}{\fa{l\in L}{(l',l) \in \Z}},
\]
where $(\cdot,\cdot)$ denotes the intersection form, extended linearly
to $L_\Q$. Moreover, set  $E^*_v\in L' $ for  the unique rational cycle
 satisfying $(E_v, E_v^*) = -1$ and $(E_w, E_v^*) = 0$ for
$w\neq v$.

In this surface singularity case  the
 \emph{canonical cycle} $Z_K\in L'$   is the unique rational cycle on $\X$
supported on the exceptional divisor,
  satisfying the
\emph{adjunction formula}
\[
  (E_v, Z_K) = - b_v + 2 - 2g_v
\]
for any irreducible component $E_v$ of the exceptional divisor, where
$-b_v$ is the Euler number of the normal bundle of $E_v \subset \X$,
and $g_v$ is the genus of $E_v$ (we assume here that the components $E_v$
of the exceptional divisor are smooth).
\end{definition}

\begin{rem}
The cycles $Z_K$ and $E_v^*$ are well defined, since the intersection
matrix, with entries $(E_v, E_w)$,
associated with any resolution is negative definite.
Notice also that any two canonical divisors are linearly equivalent, and
that any canonical divisor $K$ is numerically equivalent to $-Z_K$.
However,
it can happen that $\O_{\X}(K_{\X}+Z_K)$ has infinite order in the Picard group.
\end{rem}

\begin{prop} \label{prop:cans} Fix any  $r$.
Let $(X,0) \subset (Y,0)$ be a Newton nondegenerate Weil divisor, and
$\tilde \triangle_f$ a subdivision of the normal fan $\triangle_f$ so
that $\tilde Y \to Y$ is an embedded resolution.
Then the divisors
\begin{equation} \label{eq:can_divs}
  K_{\tilde Y} = - \sum_{\sigma\in\tilde\triangle^{(1)}_f} D_\sigma
               \in \Div(\tilde Y),\quad
  K_{\X}       = - \sum_{\sigma\in\tilde\triangle^{(1)}_f}
                        (1 + m_\sigma)E_\sigma
               \in \Div(\X)
\end{equation}
are possible canonical divisors for $\tilde Y$ and $\X$, respectively.

Furthermore, in the surface case ($r=3$),
the canonical cycle on $\X$ is given by the formula
\begin{equation} \label{eq:can_cycle}
  Z_K - E = \wt(f) - \sum (m_n+1)E_v^*,
\end{equation}
where the sum to the right runs through edges $\{n,v\}$ in the graph $G^*$
so that $n\in\Nd^*\setminus\Nd$ and $v\in\V$ (and the identity is in $L$).
\end{prop}

\begin{proof}
For $K_{\Yt}$, see e.g. 4.3 of \cite{Fulton_toric}. Since the divisor
$\X + \sum_{\sigma\in\tilde\triangle_f^{(1)}} m_\sigma D_\sigma = (\pi^*f)$
is principal in $\tilde Y$
(and $D_\sigma|_{\X}=E_\sigma$), the adjunction formula gives
\[
  K_{\X} = \left.\left( K_{\tilde Y} + \X \right)\right|_{\X}
    = -\sum_{\sigma\in\tilde\triangle_f^{(1)}} (m_\sigma + 1) E_\sigma,
\]
which proves \cref{eq:can_divs}.
To prove \cref{eq:can_cycle}, it is enough to show that in $L$  for all $v\in \V$,
\begin{equation} \label{eq:can_cycle_pf}
  (Z_K-E, E_v) = \left(\wt(f) - \sum (m_n+1)E_v^*, E_v\right),
\end{equation}
where the sum is as in \cref{eq:can_cycle}.
Recall that $\wt(f) = \sum_{v\in \V} m_v E_v$.
We note that
the adjunction formula gives $(Z_K-E, E_v) = 2 - 2g_v - \delta_v$ for all
$v \in \V$, where $\delta_v$ is the valency of the vertex
$v$ in $G$, and $g_v$ is the genus of $E_v$.
Furthermore, it follows from \cref{def:okas_graph}
that if $v\in \V$, then
\begin{itemize}

\item
$\delta_v = 1$ if and only $v$ is on the end of a bamboo
joining a node $n\in \Nd$ and an extended node $n' \in \Nd^*\setminus \Nd$.
In this case, $v$ has exactly one neighbour in $\V^*\setminus \V$ in the
graph $G^*$.

\item
$\delta_v = 2$ if and only $v$ is on a bamboo joining two extended nodes,
and is not of the form described in the previous item.

\item
$\delta_v \geq 3$ if and only if $v$ is a node.

\end{itemize}

Consider first the case $\delta_v = 1$, and let $n$ be the unique
neighbour of $v$ in $\Nd^* \setminus \Nd$. It follows from
\cref{lem:bv} that $(\wt(f), E_v) = -m_n$, since $F_v$ is a
segment, and so has area zero. As a result, the right hand side
of \cref{eq:can_cycle_pf} is $1 = (Z_K-E,E_v)$.

Next, assume that $\delta_v = 2$. Then both sides of \cref{eq:can_cycle_pf}
vanish (use again \cref{lem:bv}).

Assume finally that $v \in \Nd$. Then, $v$ has no neighbours in
$\Nd^* \setminus \Nd$. Furthermore, $\delta_v$ coincides with the
number of integral points on the boundary of $F_v$, since
each edge adjacent to $v$ can be seen to correspond to a primitive
segment of the boundary.
By using Pick's theorem and \cref{lem:bv}, we therefore get
\[
  (Z_K-E, E_v) = 2 - 2g_v - \delta_v = -2\Vol_2(F_v) = (E_v, \wt(f)),
\]
which finishes the proof.
\end{proof}
\begin{rem}
As $m_\sigma$ depends on the choice of $f$ up to a $x^p$ multiplication,
the right hand side of the second formula from \cref{eq:can_divs} depends  on this choice too.
In fact,  the monomial rational function $x^p$ realizes the
  linear equivalence between the two   divisors $K_{\X}$ associated with two such choices.
\end{rem}

\section{Gorenstein surface singularities} \label{s:Gor}

In this section we prove \cref{thm:Gor},
which characterizes nondegenerate normal surface Gorenstein
singularities by their Newton polyhedron.
The key technical \cref{lem:can_wt_div,lem:g_Gor} provide the
tools for the proof. They are proved using
vanishing of certain cohomology groups calculated
by toric methods.
In the first lemma, the restriction $r=3$ is not needed.
However, the second lemma relies on the negative
definiteness of the intersection form, restricting our result to the
surface case.

\begin{definition} \label{def:pointedG}
Let $f$ and $\triangle_f$ be as above.
We say that $\Gamma_+(f)$, or $f$, is ($\Q$-)\emph{Gorenstein-pointed}
if there exists a $p\in M$ ($p\in M_\Q$) such  that $\ell_\sigma(p) = m_\sigma+1$
for all $\sigma \in \triangle^{*(1,1)}_f$.
\end{definition}

\begin{example}
Recall that
$(Y,0)$ is Gorenstein if and only if there is a $p \in M$
satisfying $\ell_\sigma(p) = 1$ for all $\sigma\in\triangle_\Sigma^{(1)}$,
see e.g. \cite{BruGub}, Theorem 6.32.
Therefore, if $(X,0)$ is Cartier, and
$\triangle_f^* = \triangle_\Sigma^*$, then
$f$ is Gorenstein pointed (since
$m_\sigma = 0$ for $\sigma\in\triangle_\Sigma^{(1)}$).
Furthermore,  $(X,0)$ is Gorenstein since $(Y,0)$ is Gorenstein and
$f$ forms a regular sequence.

Similarly,  $(Y,0)$ is $\Q$-Gorenstein if there is a $p \in M_\Q$
satisfying $\ell_\sigma(p) = 1$ for all $\sigma\in\triangle_\Sigma^{(1)}$,
see e.g. \cite{Altmann_QGor}.
Therefore, if $(X,0)$ is Cartier, and
$\triangle_f^* = \triangle_\Sigma^*$, then
$f$ is $\Q$-Gorenstein pointed.
\end{example}
\begin{rem}
Though the two combinatorial conditions in definitions \ref{def:pointed} and
\ref{def:pointedG}  look very similar,
they codify two rather different
geometrical properties. Being `pointed'  codifies an embedding property, namely
that $(X,0)\subset (Y,0)$ is Cartier,    see  \cref{prop:pg_Cartier}.
However, being `Gorenstein pointed'  codifies an abstract property of the germ $(X,0)$, namely its Gorenstein property, see
\cref{thm:Gor} below.
\end{rem}

\begin{block}
Recall also that $(X,0)$ is Gorenstein if it admits a Gorenstein form.
A Gorenstein form is a
nowhere vanishing section in $H^0(X\setminus 0, \Omega^2_{X\setminus 0}) =
H^0(\widetilde{X}\setminus E,\Omega^2_{\widetilde{X}\setminus E})$.
A Gorenstein pluri-form is a
nowhere vanishing section in
$H^0(\widetilde{X}\setminus E,(\Omega^2_{\widetilde{X}\setminus E})^{\otimes k})$
for some $k\in\Z_{>0}$.

In this section $K_{\tilde{Y}}$ and $K_{\X}$ are canonical divisors with a  choice as in
\cref{eq:can_divs}.
\end{block}

\begin{definition}
Let $\omega_f$ be some meromorphic 2-form on $\X$ whose divisor
$(\omega_f)$ is $K_{\X}$.
\end{definition}

\begin{thm} \label{thm:Gor}
Assume that $(X,0) \subset (Y,0)$ is a normal Newton nondegenerate surface
singularity (i.e. $r=3$).
The following conditions are equivalent:
\begin{enumerate}

\item \label{it:Gor_point}
$f$ is Gorenstein-pointed at some $p\in M$.

\item \label{it:Gor_ell}
There exists a $p\in M$ so that for all $v\in \V^*\setminus \V$ we have
$\ell_v(p) = m_v + 1$.

\item \label{it:Gor_ellG}
There exists a $p\in M$ so that for all $v\in \V$ we have
$\ell_v(p) = m_v + 1 - m_v(Z_K)$.

\item \label{it:Gor_form}
There exists a $p\in M$ so that $x^p\omega_f$ is a Gorenstein form.

\item \label{it:Gor_Gor}
$(X,0)$ is Gorenstein.

\end{enumerate}
When these conditions hold,
\ref{it:Gor_point},  \ref{it:Gor_ell},
\ref{it:Gor_ellG} and \ref{it:Gor_form} uniquely identify the same point $p$.
\end{thm}
In fact, the analogues of parts \cref{it:Gor_point}--\cref {it:Gor_form} are equivalent over
rational points $p\in M_\Q$ as well.

\begin{prop}\label{prop:QGor}
Under the assumption of \cref{thm:Gor},
the following conditions are equivalent, and imply that
$(X,0)$ is $\Q$-Gorenstein:
\begin{enumerate}

\item \label{it:Gor_point_Q}
$f$ is $\Q$-Gorenstein-pointed at some $p\in M_\Q$.

\item \label{it:Gor_ell_Q}
There exists a $p\in M_\Q$ so that for all $v\in \V^*\setminus \V$ we have
$\ell_v(p) = m_v + 1$.

\item \label{it:Gor_ellG_Q}
There exists a $p\in M_\Q$ so that for all $v\in \V$ we have
$\ell_v(p) = m_v + 1 - m_v(Z_K)$.

\item \label{it:Gor_form_Q}
There exists a $p\in M_\Q$ so that $x^{kp}(\omega_f)^{\otimes k}$ is a Gorenstein pluri-form
 for some $k\in \Z_{>0}$.

\end{enumerate}
Furthermore, all these  these conditions identify the very same $p$ uniquely.
\end{prop}
\begin{proof}
\ref{it:Gor_ell_Q} is a rephrasing of \ref{it:Gor_point_Q}, since $\triangle^{*(1,1)}_f=\V^*\setminus \V$.

\ref{it:Gor_ell_Q}$\Rightarrow$\ref{it:Gor_ellG_Q}\ For any $p\in M_\Q$
consider the cycles
\[
  Z_1:=\sum_{v\in\V}\ell_v(p)E_v\in L_\Q,\qquad
  Z_2:=\sum (m_n+1)E_v^*\in L_\Q,
\]
where the sum
runs over edges $\{n,v\}$ in $G^*$ so that $n\in\V^*$ and $v\in\V$
(as in \cref{eq:can_cycle}),
 and $Z^*:=\sum_{n\in\V^*\setminus \V}
\ell_n(p)E_n$ (where all these $E_n$'s are the noncompact curves in $\X$).

If $\{n,v\}$ is an edge as above, then $(Z_2, E_v)=-(m_n+1)$. Moreover, $(Z^*, E_v)_{\X}=\ell_n(p)$.
Therefore, by  assumption \ref{it:Gor_ell_Q},
$(Z^*+Z_2, E_u)_{\X}=0$ for any $u\in \V$. On the other hand, by \cref{lem:bv}, $(Z^*+Z_1, E_u)_{\X}=0$ for any $u\in \V$
as well. Hence $Z_1=Z_2$. But by \cref{eq:can_cycle} $m_u(Z_2)=m_v+1-m_v(Z_K)$.

\ref{it:Gor_ellG_Q}$\Rightarrow$\ref{it:Gor_ell_Q}\ With the above notations, \ref{it:Gor_ellG_Q} shows that
$Z_1=Z_2$. Let $\{n,v\}$ be an edge as above, let $w\in\V$ be the other neighbour of $v$, and note that
$E_v=b_vE^*_v- E^*_w$ in $L'$.
Then,
\[
  m_n+1=(Z_2, -E_v)= (Z_1, -b_vE^*_v+E^*_w)= \ell_v(p)b_v-\ell_w(p)=\ell_n(p)
\]
(in the last equality  use \cref{lem:bv}).

For \ref{it:Gor_ell_Q}$\Leftrightarrow$\ref{it:Gor_form_Q} use the second identity of \cref{eq:can_divs}.
\end{proof}

\begin{rem} \label{rem:Gor}
Similarly as in \cref{thm:Gor}, one may ask
whether the equivalent cases in \ref{prop:QGor}
are equivalent with the property that  $(X,0)$ is $\Q$-Gorenstein. If $f$ is
$\Q$-Gorenstein-pointed at $p\in M_\Q$, then \ref{it:Gor_form_Q} implies that
 $(X,0)$ is $\Q$-Gorenstein. {\it The converse does not hold}, as seen by the
following example.

 Let $N = \Z^3$ and
\[
  \Sigma = \R_{\geq 0}\langle (1,0,0),
                              (0,1,0),
                              (1,0,1),
                              (0,1,1) \rangle,\quad
  f(x) = x^{(0,0,2)} + x^{(1,0,1)} + x^{(0,2,0)} + 2x^{(1,2,-1)}.
\]
Write $\sigma_i$, $i=1,2,3,4$ for the rays generated by the vector specified
above
and denote by
$m_i$ the corresponding multiplicities. We find $m_1=m_2=m_3=0$ and
$m_4 = 1$. As a result, since the linear equation
\[
\left(
\begin{matrix}
  1 & 0 & 0 \\
  0 & 1 & 0 \\
  1 & 0 & 1 \\
  0 & 1 & 1
\end{matrix}
\right)
  \cdot p
  =
\left(
\begin{matrix}
  1 \\
  1 \\
  1 \\
  2
\end{matrix}
\right)
\]
has no solution, hence $f$ is not $\Q$-Gorenstein  pointed.

\begin{figure}[ht]
\begin{center}
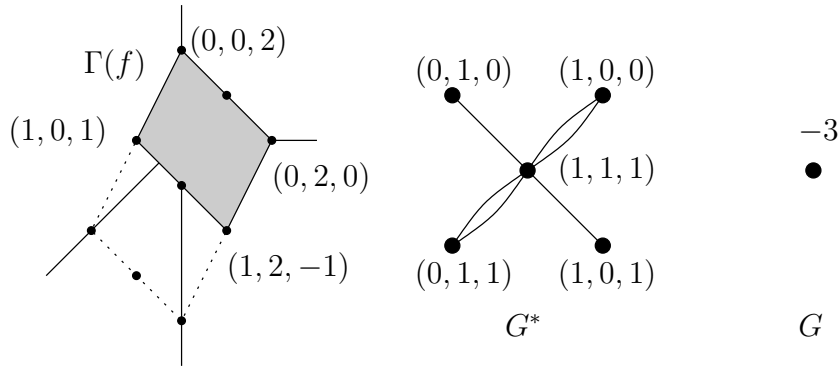
\caption{A Newton diagram, and the output of Oka's algorithm. The dotted line
shows the intersection of the affine hull of the only face of the diagram
intersected with $\partial \Sigma^\vee$.
For simplicity, here  in $G^*$  we have blown down
the $(-1)$-vertices constructed in the last paragraph
of \cref{block:two_can}.}
\label{fig:ex_3}
\end{center}
\end{figure}

On the other hand, one verifies that the Weil divisor defined by $f$
is normal
using \cref{thm:geom_genus}. Furthermore, Oka's algorithm shows that
this singularity has a resolution with an exceptional divisor
consisting of a single rational curve with Euler number $-3$.
Such a singularity is a cyclic quotient singularity.
In particular, it is $\Q$-Gorenstein.
\end{rem}

\begin{block}
Next,  we focus  on the proof of \cref{thm:Gor}.
The equivalences of the first four cases follow from (or, as) \cref{prop:QGor}.
For \cref{it:Gor_point}$\Rightarrow$\cref{it:Gor_Gor} note that
 if $f$ is Gorenstein-pointed at $p\in M$
then $x^{p} \omega_f$ trivializes the canonical bundle.
The implication \cref{it:Gor_Gor}$\Rightarrow$\cref{it:Gor_point}
will be proved below based on two lemmas.
\end{block}

\begin{lemma} \label{lem:can_wt_div}
Let $\,\overline g\in H^0(\X\setminus E,\O_{\X}(K_{\X}))$, that is,
$\overline g$ is
a meromorphic function on the complement of the exceptional divisor
in $\X$ satisfying
\begin{equation} \label{eq:can_div}
  (\overline g)
    \geq -K_{\X} |_{\X\setminus E}
    =    \sum_{v\in\V^*\setminus \V} (m_v+1)E_v.
\end{equation}
Then, there exists a Laurent series $g\in\O_{Y,0}[x^M]$
satisfying $(\pi^*g)|_{\X\setminus E} = \overline g$ and
\begin{equation} \label{eq:can_wt}
  \fa{\sigma \in \tilde\triangle_f^{*(1)}}
     {\wt_\sigma g \geq m_\sigma+1}.
\end{equation}
\end{lemma}

\begin{proof}
Let $I = H^0(\X\setminus E, \O_{\X}(K_{\X}))$
and let $J$ be the set of meromorphic functions obtained as a restriction
of Laurent series satisfying \cref{eq:can_wt}. We want to show that $I = J$.

We immediately see $J \subset I$. In fact, this inclusion fits into
an exact sequence as follows. Recall the notation
$D_m = \sum_{\sigma \in \tilde\triangle_f^{(1)}} m_\sigma D_\sigma$
from the proof of \cref{thm:geom_genus}, and
$K_{\Yt} = -\sum_{\sigma \in \tilde\triangle_f^{(1)}} D_\sigma$.
Also, define $\Dc$ as the union of compact divisors in $\tilde Y$, that is,
$\cup_\sigma  D_\sigma$ for $\sigma \not \in \tilde\triangle_f^{*(1)}$.
Since $(\pi^*f)=\X+D_m$,
we have a short exact sequence of sheaves
\[
  0
  \to
  \O_{\Yt\setminus\Dc}(K_{\Yt})
  \stackrel{\cdot f}{\to}
  \O_{\Yt\setminus\Dc}(-D_m + K_{\Yt})
  \to
  \O_{\X\setminus E}(-D_m + K_{\Yt})
  \to
  0
\]
yielding a long exact sequence of cohomology groups.
We have
\[
  I = H^0(\X \setminus E, \O_{\X\setminus E}(-D_m + K_{\Yt})),
\]
since $K_{\widetilde{X}}=(-D_m+K_{\widetilde{Y}})|_{\widetilde{X}}$.
Furthermore, since $\Yt$
is normal, $H^0(\Yt\setminus\Dc, \O_{\Yt\setminus\Dc}(-D_m+K_{\Yt}))$
is the set of Laurent
series satisfying \cref{eq:can_wt}. Thus, its image in
$I$ is $J$. Therefore, the quotient $I/J$ injects into
$H^1(\Yt\setminus\Dc, \O_{\Yt\setminus\Dc}(K_{\Yt}))$.
On the other hand,
\begin{equation} \label{eq:van_gp}
  H^1(\Yt\setminus\Dc, \O_{\Yt\setminus\Dc}(K_{\Yt}))
  \cong
  \bigoplus_{p\in M} H^1_{Z(p)}(\partial\Sigma, \C),
\end{equation}
where, following Fulton \cite{Fulton_toric},
$\psi_K:\partial\Sigma \to \R$ is the unique function restricting
to linear function on all $\sigma\in\tilde\triangle_f^*$, and satisfying
$\psi_K(\ell_\sigma) = 1$ for $\sigma\in\tilde\triangle_f^{(1)*}$,
and for $p\in M$ we set
\[
  Z(p) = \set{\ell\in\partial\Sigma}{\ell(p) \geq \psi_K(\ell)}.
\]
Firstly, since $\partial\Sigma$ is contractible, we find
\[
  H^1_{Z(p)}(\partial\Sigma, \C)
  \cong
  \tilde H^0(\partial\Sigma\setminus Z(p), \C).
\]
Secondly, define $Z'(p)$ as the union of those cones
$\sigma\in\tilde\triangle_f^*$ satisfying $p|_\sigma \geq 0$
(i.e. $\ell(p) \geq 0$ for all $\ell\in \sigma$),
and
let $Z''(p)$ be the set of $\ell\in\partial\Sigma$ satisfying
$\ell(p) \geq 0$. By \cref{lem:subcx}, the inclusions
\[
  \partial\Sigma \setminus Z(p)
  \subset
  \partial\Sigma \setminus Z'(p)
  \supset
  \partial\Sigma \setminus Z''(p)
\]
are strong deformation retracts. But the right hand side above is either a
contractible set, or it has the homotopy of $S^{r-2}$. In particular,
it is connected, by our assumption $r>2$, and so \cref{eq:van_gp} vanishes.
\end{proof}

\begin{lemma} \label{lem:g_Gor}
Assume that $(X,0)$ is a Gorenstein normal
surface singularity, i.e. $r=3$, and that
we have a Gorenstein form $\omega$ on $\X\setminus E$.
Thus, $-K_{\X}-Z_K$ is linearly trivial,
and there exists
\[
  \overline g\in H^0(\X,\O_{\X}(K_{\X}+Z_K)), \quad
  (\overline g) = (\omega) - (\omega_f) = -Z_K - K_{\X}.
\]
Then there is a $g \in \O_{Y,0}[x^M]$ satisfying
\begin{equation} \label{eq:g_Gor}
 (\pi^* g)|_{\X} = \overline g\quad\mathrm{and}\quad
  \fa{v\in\V^*}{\wt_v(g) = \div_v(\overline g)}.
\end{equation}
\end{lemma}

\begin{proof}
By the previous \cref{lem:can_wt_div}, we can find a $g$ satisfying
$g|_{\X\setminus E} = \overline g$ and \cref{eq:can_wt}.
Let $A = (\overline g)$ and $B = \sum_{v\in\V^*} \wt_v(g)E_v$.
We want to prove that $A = B$. Both $A$ and $B$ are supported in the exceptional
divisor and the noncompact curves $E_v$ for $v\in\V^*\setminus \V$, and
by our assumptions, they have the same multiplicity along this
noncompact part. Thus, $A-B$ is supported on the exceptional divisor.
Furthermore, we have $\wt_v(g) \leq \div_v(\overline g)$ for $v\in\V$, thus
$B-A \leq 0$.

For the reverse inequality, note first that $(A,E_v) = 0$ for all
$v\in\V$ since $A$ is principal. For any $v\in\V$, let $q\in M$ be an
element of the support of the principal part of $g$ with respect to $\ell_v$,
i.e. $q \in \supp(g)$ and $\ell_v(q) = \wt_v(g)$. By definition, we also
have $\ell_u(q) \geq \wt_u(g)$ for all $u\in\V^*_v$. Therefore,
\[
  (B,E_v) = -b_v\wt_v(g) + \sum \set{\wt_u(g)}{u\in\V^*_v}
    \leq -b_v\ell_v(q) + \sum \set{\ell_u(q)}{u\in\V^*_v} = 0.
\]
As a result, $B-A$ is in the Lipman cone, and so, $B-A \geq 0$, proving
\cref{eq:g_Gor}.
\end{proof}

\begin{proof}[Proof of \cref{thm:Gor}]

The first four conditions are equivalent by \cref{prop:QGor}, and
\cref{it:Gor_form} clearly implies \cref{it:Gor_Gor}.

Assuming that $(X,0)$ is Gorenstein, let $\omega$ be a Gorenstein form.
Then there is meromorphic $\overline g$ so that $\overline g\omega_f = \omega$
on $\X\setminus E$.
By \cref{lem:g_Gor}, $\overline g$ is the restriction
of a Laurent series $g \in \O_{Y,0}[x^M]$ satisfying \cref{eq:g_Gor}.

For any $v\in \V$, denote by $g_v$ the principal part of $g$
with respect to the weight $\ell_v$.
We make the

\vspace{1mm}

\noindent \emph{Claims:}

\begin{enumerate}[(a)]
\item \label{cl:Gor_n}
For any $n\in \Nd$, $g_n$ is
a monomial, that is, there is a $p_n \in M$ so that $g_n = a_nx^{p_n}$
for some $a_n\in\C^*$.

\item \label{cl:Gor_bam}
If $v$ is a vertex on a bamboo connecting $n\in\Nd$ and some
other node in $\Nd^*$, then $g_v = a_nx^{p_n}$.

\end{enumerate}

By (\ref{cl:Gor_bam}), the exponent $p = p_n$ does not depend on $n$,
finishing the proof since
hence $x^p\omega_f$ is a Gorenstein form.

(\ref{cl:Gor_n}) is proved  as follows.  Set $q \in \supp(g_n)$ arbitrarily. We then have
$\wt_n(g) = \ell_n(q)$, and also $\wt_u(g) \leq \ell_u(q)$, for any
other $u$, since
$\supp(g_n) \subset \supp(g)$. In particular,
\[
  -b_n \wt_n(g) + \sum_{u\in\V_n} \wt_u(g)
  \leq
  -b_n \ell_n(q) + \sum_{u\in\V_n} \ell_u(q).
\]
The right hand side is sero since $\ell_n + \sum_{u\in\V_n} \ell_u = 0$ for $n\in \Nd$.
On the other hand, by the \cref{lem:g_Gor}, we have $\wt_v(g) = \div_v(\overline g)$
for all $v$, thus,
the left hand side above equals $(\div(g),E_n)$. Furthermore,
since $(\overline{g}) = (\omega) - (\omega_f)$, $g$ does not have any zeroes or
poles outside the exceptional divisor, in a neighbourhood around $E_n$,
hence  $(\div(\overline{g}),E_n) = ((g),E_n) = 0$.
 Therefore,
the inequality above is an equality, and we have
$\wt_u(g) = \ell_u(q)$ for $u\in\V_n$.

This fact is true for any choice of $q$, therefore,
$\ell_u(q') =\wt_u(g)= \ell_u(q)$ for any  $u\in\V_n$ and
for any other choice $q'$.
But the vectors $\{\ell_u\}_{u\in \V_n}$ form
a generator set, hence necessarily $q=q'$.

For (\ref{cl:Gor_bam}), assume that $n$ and $n'\in\Nd^*$ are joined by
a bamboo, consisting of vertices $v_1,\ldots, v_s$, with
$v_1\in \V_n$ and $v_s \in \V_{n'}$, and $v_i,v_{i+1}$ neighbours for
$i=1,\ldots,s-1$. For convenience, we set $v_0 = n$ and $v_{s+1} = n'$.
We start by showing that $\wt_i(g) = \ell_i(p_n)$ using induction (we
replace the subscript $v_i$ by just $i$ for legibility).
Indeed, for $i=0$ this is clear, and we showed in the proof of
(\ref{cl:Gor_n}) that this holds for $i=1$. For the induction step
we use the recursive formulas
\[
  \ell_{i+1}   - b_i \ell_i   + \ell_{i-1}   = 0,\quad
  \wt_{i+1}(g) - b_i \wt_i(g) + \wt_{i-1}(g) = 0.
\]
The first one holds by \cref{lem:bv}, and the second one follows
from $\wt_i(g) = \div_i(g)$ similarly as above, although for the case
$i=s$, we may have to use a component of the noncompact curve $E_{n'}$.

We now see that for any $1\leq i\leq s$, the support of $g_i$ consists of points
$q\in M$ for which $\ell_i(q) = \ell_i(p_n)$ and
$\ell_{i\pm 1}(q) \geq \ell_{i\pm 1}(p_n)$.
But these
equations are equivalent to $\ell_n(q) = \ell_n(p_n)$ and
$\ell_{n'}(q) = \ell_{n'}(p_n)$. Therefore, $\supp(g_i) = \supp(g_n)$ for
these $i$. 
\end{proof}

\section{The geometric genus and the diagonal computation sequence} \label{s:comp_seq}

In this section we construct the diagonal
computation sequence, and show that it computes the
geometric genus of any Newton nondegenerate,
$\Q$-Gorenstein pointed,   normal surface
singularity having a rational homology sphere link.
Any computation sequence provides an upper bound for the geometric genus.
The smallest such bound is a topological invariant, and we show that
this is realized by this diagonal sequence.
This is done by showing that the diagonal computation sequence counts
the lattice points ``under the diagram'', whose number is precisely
the geometric genus, according to \cref{cor:pg_normal_surface}.

\begin{block} {\bf Discussions regarding general normal surface singularities.}
Throughout this section, when not mentioned specifically,
$\pi:(\tilde X,E) \to (X, 0)$ denotes a resolution
of a normal surface singularity $(X,0)$ with exceptional divisor $E$, whose
 irreducible decomposition is  $E = \cup_{v\in\V} E_v$.

{\em We assume that $(X,0)$ has a rational  homology sphere link}; thus
 $E_v \cong \CP^1$ for all $v\in\V$.

  We use the notations $L$, $L'$ and $E^*_v$ as in
\cref{s:can}.
For  $Z = \sum_v r_v E_v$ with $r_v\in\Q$  we write
$\lfloor Z \rfloor = \sum_v \lfloor r_v \rfloor E_v$.
$Z_K$ denotes the canonical cycle.
Note that $Z_K=0$ if and only if $(X,0)$ is an $ADE$
germ. Otherwise, it is known that
in the minimal resolution, or, even in the minimal good resolution, all the coefficients of
$Z_K$ are strictly positive. However, usually this is not the case in non-minimal resolutions, i.e.
in our $G$ it is not automatically guaranteed.
\end{block}

\begin{lemma}\label{lem:pg_lemma}
 In any resolution $\X\to X$  of a normal surface singularity with $\lfloor Z_K \rfloor\geq 0$
we have
\begin{equation} \label{eq:pg_lemma}
  p_g = \dim_\C \frac{H^0(\X,\O_{\X}(K_{\X} + \lfloor Z_K \rfloor))}
                     {H^0(\X,\O_{\X}(K_{\X}                      ))}.
\end{equation}
\end{lemma}
\begin{proof}
By the  generalized version of Grauert--Riemenschneider
vanishing we have the two vanishings
\begin{equation} \label{eq:GR_van}
  H^1(\X, \O_{\X}(K_{\X})) =0, \ \ \
  H^1(\X, \O_{\X}(-\lfloor Z_K \rfloor)) = 0.
\end{equation}
Hence, if  $\lfloor Z_K \rfloor= 0$ then $p_g=0$ too. Otherwise,
from the long exact sequence of cohomology groups
associated with 
\[
      0
  \to \O_{\X}(K_{\X})
  \to \O_{\X}(K_{\X}+\lfloor Z_K \rfloor)
  \to \O_{\lfloor Z_K \rfloor}(K_{\X} + \lfloor Z_K \rfloor)
  \to 0,
\]
we obtain that the right hand side of \cref{eq:pg_lemma} equals
$\dim\, H^0(\lfloor Z_K\rfloor , \O_{\lfloor Z_K \rfloor}(K_{\X} + \lfloor Z_K \rfloor))$.
By Serre duality,
this equals $H^1(\lfloor Z_K \rfloor, \O_{\lfloor Z_K \rfloor})$.
Now, the short exact sequence
\[
      0
  \to \O_{\X}(-\lfloor Z_K \rfloor)
  \to \O_{\X}
  \to \O_{\lfloor Z_K \rfloor}
  \to 0,
\]
with the above vanishing  \cref{eq:GR_van} 
gives
$H^1(\lfloor Z_K \rfloor, \O_{\lfloor Z_K \rfloor}) \cong H^1(\X, \O_{\X}) \cong \C^{p_g}$.
\end{proof}

\begin{definition}\label{def:CS}
A \emph{computation sequence} is a sequence of cycles $(Z_i)_{i=0}^k$ from $Z_K+L$,
\[
  Z_K - \lfloor Z_K \rfloor = Z_0 < \ldots < Z_k
\]
such  that
\begin{blist}

\item\label{it:CS1}
for all $0\leq i < k$ there is a $v(i)\in\V$ so that
$Z_{i+1} = Z_i + E_{v(i)}$, and

\item\label{it:CS2}
$Z_k\geq Z_K$ and $Z_k - Z_K$ is the union of some reduced and
non-intersecting rational $(-1)$-curves
\end{blist}
Given such a sequence $(Z_i)_{i=0}^k$, we define
\[
  \Lb_i = \O_{\X}(K_{\X} + Z_K - Z_i),\quad
  \Qb_i = \Lb_i/\Lb_{i+1}.
\]
Then $\Qb_i$ is a line bundle on $E_{v(i)}$. Denote by $d_i$ its degree.
Since $K_{\X} + Z_K$ is numerically equivalent to zero,
we have $d_i = (-Z_i, E_{v(i)})$. In particular, since
$E_{v(i)} \cong \CP^1$, we get $\Qb_i=\O_{E_{v(i)}}(-d_i)$ and
\[
  h^0(E_{v(i)}, \Qb_i) = \max\{0,(-Z_i,E_{v(i)}) + 1\}.
\]
\end{definition}

\begin{block}
Given a computation sequence $(Z_i)_i$,
the inclusion  $\O_{\X}(K_{\X}+Z_K-Z_k)\hookrightarrow \O_{\X}(K_{\X})$
induces an isomorphism
\[
  H^0(\X, \O_{\X}(K_{\X}+Z_K-Z_k)
  \stackrel{\cong}{\longrightarrow}
  H^0(\X,  \O_{\X}(K_{\X})).
\]
Indeed, let $\UU \subset \V$ be such that
$Z_k - Z_K = \sum_{u \in \UU} E_u$. Then we have a short exact sequence
\[
  0
  \to
  \O_{\X}(K_{\X} - E_\UU)
  \to
  \O_{\X}(K_{\X})
  \to
  \bigoplus \O_{E_u}(K_{\X})
  \to
  0,
\]
which induces an exact sequence
\[
  0
  \to
  H^0(\X, \O_{\X}(K_{\X}+Z_K-Z_k)
  \to
  H^0(\X,  \O_{\X}(K_{\X}))
  \to
  \bigoplus H^0(E_u,\O_{E_u}(K_{\X})),
\]
and the right hand side vanishes,
since $(E_u, K_{\X}) = -2 - 2g_u + b_u =  -1$.
\end{block}

\begin{cor} \label{cor:pg_bound}
Let $(Z_i)_{i=0}^k$ be a computation sequence. Then
\begin{equation} \label{eq:pg_bound}
  p_g =\sum_{i=0}^{k-1} \dim\, \frac{H^0(\X,\Lb_i)}{H^0(\X,\Lb_{i+1})}
  \leq \sum_{i=0}^{k-1} \max\{0, d_i+1\}.
\end{equation}
with equality if and only if the map
$H^0(\X,\Lb_i) \to H^0(E_{v(i)}, \Qb_i)$ is surjective for all $0\leq i < k$.
\qed
\end{cor}

\begin{rem} \label{rem:no_contrib}
\begin{blist}

\item
We note in particular that if there exists a computation sequence
$(Z_i)_{i=0}^k$
so that $(Z_i, E_{v(i)}) > 0$ for all $i$, then $p_g = 0$, that is,
$(X,0)$ is rational. In general, if $(Z_i, E_{v(i)}) > 0$ for some
$i$, then the inequality between the $i^{\textrm{th}}$ terms in the
sums \cref{eq:pg_bound} is an equality.

\item
Let $S(Z_i)$ be the sum
$\sum_i \max\{0, d_i+1\}$ from the right hand side of \cref{eq:pg_bound}
associated with $(Z_i)$. Then we have
\begin{equation} \label{eq:NEW}
p_g\leq \min_{(Z_i)} S(Z_i),
\end{equation}
where the minimum is taken over all computation sequences.
Note that $\min_{(Z_i)} S(Z_i)$ is an
invariant associated with the topological type (graph),
hence in this way we get a
\emph{topological upper bound} for the geometric genus of all possible
analytic types supported on a fixed topological type.

On the other hand we emphasize the following facts.
In general it is hard to identify a sequence
which minimizes $\{S(Z_i)\}$.
Also, for an arbitrary fixed topological type, it is not even true that
there exists an analytic type supported on the fixed topological type for
which \cref{eq:NEW} holds.
Furthermore, it is even harder to identify those analytic structures which
maximize $p_g$,
e.g., if  \cref{eq:NEW} holds for some analytic structure, then which
are these maximizing analytic structures, see e.g. \cite{Nem_Oku_pg}.

In the sequel our aim is the following: in our toric Newton nondegenerate case
we construct combinatorially a sequence (it will be called `diagonal sequence'),
which satisfies \cref{eq:pg_bound} with equality (in particular it
minimizes  $\{S(Z_i)\}$ as well).
This also shows that if a topological type is realized by a Newton nondegenerate Weil divisor,
then this germ maximizes the geometric genus of analytic types supported by that topological type.
\end{blist}
\end{rem}

\begin{block} \label{block:Laufer_op}
We recall the construction of the {\em Laufer operator} and {\em generalized
Laufer sequences} with respect to $\Nd \subset \V$.
We claim that for any cycle $Z\in L'$, there is a smallest cycle
$x(Z) \in Z+L$ satisfying
\begin{equation} \label{eq:Laufer_op}
\left \{\begin{array}{l}
  \fa{n\in\Nd}{m_n(x(Z)) = m_n(Z)},\\
  \fa{v\in\V\setminus\Nd}{(x(Z),E_v)\leq 0}.\end{array}\right.
\end{equation}
The existence and uniqueness
of such an element is explained in \cite{Nemethi_OzsSzInv} in the case
when $|\Nd| = 1$ and in general in \cite{Laszlo_th,NS-hyper,Baldur_th}.
The name comes from a construction of Laufer in
\cite[Proposition 4.1]{Lauf_rat}.
Note that $x(Z)$ only depends on the multiplicities $m_n(Z)$ of $Z$ for
$n\in\Nd$ and the class $[Z] \in H = L'/L$.

The following properties hold for the operator $x$,
assuming $Z_1- Z_2 \in  L$:

\vspace{2mm}

\emph{Monotonicity:} If $Z_1\leq Z_2$ then $x(Z_1) \leq x(Z_2)$.

\emph{Idempotency:} We have $x(x(Z)) = x(Z)$ for any $Z\in L'$.

\emph{Lower bound by intersection numbers:}
If $Z\in L'$ and $Z'\in L_\Q$ so that
$m_n(Z) = m_n(Z')$ for $n\in \Nd$ and $(Z',E_v) \geq 0$ for all
$v\in\V\setminus\Nd$, then $x(Z) \geq Z'$.   

\emph{Generalized Laufer sequence:} Assume that $Z \leq x(Z)$. First note that if  $(Z,E_v) > 0$
for some $v\in\V\setminus\Nd$, then we have
$Z+E_v \leq x(Z)$ as well, similarly as in the proof of
Proposition 4.1 \cite{Lauf_rat}.
We claim that there exists a  generalized
Laufer sequence which connects $Z$ with $x(Z)$.
It  is determined recursively as follows. Start by
setting $Z_0 = Z$. Assume that we have constructed $Z_i$. By induction,
we then have $Z_i \leq x(Z)$.
If $(Z,E_v) \leq  0$ for all  $v\in\V\setminus\Nd$
then by the minimality of $x(Z)$ we get $Z_i=x(Z)$;
hence  the construction
is finished and we stop. Otherwise, there is a $v\in\V\setminus\Nd$ so that
$(Z,E_v) > 0$. We then define $Z_{i+1} = Z_i + E_v$ (for some choice of such $v$).
\end{block}

\begin{rem} The computation sequence $(Z_i)_{i=0}^k$ (as in corollary \ref{cor:pg_bound}),
what we will construct,  will have several intermediate parts formed by generalized Laufer sequences
as above. Note that if
$Z_i$ and  $Z_{i+1} = Z_i + E_v$ are two consecutive elements in
a Laufer sequence, then  $-d_i=(Z_i, E_v) > 0$, hence $\max\{0, d_i+1\}=0$, and the comment from
\cref{rem:no_contrib} applies: this step does not contribute in
the sum on the right hand side of  \cref{eq:pg_bound}.
Informally, we say that parts given by  Laufer sequences ``do not contribute to the geometric
genus''.
\end{rem}

\begin{block} \label{block:diagonal_newton}
{\bf The Newton nondegenerate case.}
Let us consider again the resolution
 $\tilde X \to X$ of  Newton nondegenerate Weil divisor as in \cref{s:res}.
Let  $K_{\X}$ denote a canonical
divisor as in \cref{s:can}.
In this section we will assume that in the dual resolution graph $G$
we have $m_n(Z_K) \geq 1$ for any node $n$.
This assumption will be  justified in \cref{s:rem_face}.

From the assumption  $m_n(Z_K) \geq 1$, valid for any node $n$,
an immediate application of the construction of $G$  from \cref{s:res} gives that
$Z_K\geq 0$. Thus
$\lfloor Z_K \rfloor>0$.
\end{block}

\begin{lemma} \label{lem:seq_coda} \begin{blist}
\item\label{it:seq_coda1}
 $x(Z_K-\lfloor Z_K \rfloor )\geq Z_K-\lfloor Z_K \rfloor$.

\item\label{it:seq_coda2}
Let $\UU \subset \V$ be the set of $(-1)$-vertices appearing on bamboos
joining $n, n' \in \Nd^*$ with $\alpha(\ell_n, \ell_{n'}) = 1$ in
\cref{def:okas_graph}. Then $x(Z_K) = Z_K + \sum_{u\in \UU} E_u$.
In particular, the sequence constructed in \cref{def:comp_seq} satisfies
\cref{it:CS2} in \cref{def:CS}.
\end{blist}
\end{lemma}
\begin{proof}
\cref{it:seq_coda1} Since $x(Z)-Z\in L$ for any $Z\in L'$,
it is enough to show that
$x(Z_K-\lfloor Z_K \rfloor )\geq 0$. We can analyse each component of  $G\setminus \Nd$ independently,
let $G_B$ be such a bamboo formed from $E_1, \ldots, E_s$, with dual vectors in $G_B$ denoted by $E_i^*$.
If $a\geq 0$ and $b\geq 0$ are the multiplicities of  $Z_K-\lfloor Z_K \rfloor$ along the
neighboring nodes of $G_B$ in $G$ (with convention that  $a=0$ if there is only one such node),
we search for a cycle $x$ with $(x, E_i)\leq (aE_1^*+bE_s^*, E_i)$ for all $i$. Thus,
$x-(aE_1^*+bE_s^*)$ is in the Lipman cone of $G_B$, hence $x\geq aE_1^*+bE_s^*\geq 0$.

\cref{it:seq_coda2}
Using the lower bound by intersection numbers, we find that
$x(Z_K)  \geq Z_K - E + \sum_{n\in \Nd} E_n$.
Since $x(Z_K)=x(Z_K-E+\sum_{n\in\Nd}E_n)$, there exists a Laufer sequence from
 $Z_K - E + \sum_{n\in \Nd} E_n$ to $x(Z_K)$.
Now, one verifies that  the construction/algorithm  of this sequence
chooses each vertex $v\in \V\setminus (\Nd \cup \UU)$ once, and each
vertex in $\UU$ twice.
\end{proof}

\begin{definition} \label{def:comp_seq}
A \emph{(coarse) diagonal computation sequence $(\bar{Z}_i)_{i=0}^{\bar {k}}$ with respect to} $\Nd$
is defined as follows. Start with $Z_0= Z_K - \lfloor Z_K \rfloor$, and define
 $\bar Z_0 =x( Z_K - \lfloor Z_K \rfloor)$.
Assuming $\bar Z_i$ ($i\geq 0$) has been defined, and that $\bar Z_i|_\Nd < Z_K|_\Nd$,
choose a $\bar v(i) \in \Nd$
minimizing the ratio
\begin{equation} \label{eq:ratio}
n\mapsto r(n):=  \frac{m_{n}(\bar Z_i)}{m_{n}(Z_K-E)}, \ \ \ n\in\Nd.
\end{equation}
Then set $\bar Z_{i+1} = x(\bar Z_i+E_{\bar{v}(i)})$.
If $\bar Z_i|_\Nd = (Z_K-E)|_\Nd$, then we record $\bar k' = i$.
If $\bar Z_i|_\Nd = Z_K|_\Nd$, then we stop, and set $\bar k = i$.

We refine the above choice as follows.
Choose some node $n_0 \in \Nd$ and define a partial order $\leq$ on
the set $\Nd$: for $n,n'\in\Nd$, define $n\leq n'$ if $n$ lies on the
geodesic joining $n'$ and $n_0$ (here we make use of the assumption that
the link is a rational homology sphere, in particular, $G$ is a tree).
When choosing $\bar v(i)$, if given a choice of several nodes minimizing
$\{r(n)\}_n$,  and $\min_n\{r(n)\}< 1$,
then, we choose $\bar v(i)$ minimal of those with respect to $(\Nd,\leq)$.
If $\min_n\{r(n)\}=1$, let $\Nd'\subset \Nd$ be the set of
nodes $n$ for which $r(n)= 1$. If $\Nd'$ has one element we have to chose that one.
Otherwise, let  $G'$ be the minimal connected
subgraph of $G$ containing $\Nd'$, and
we choose $\bar v(i)$ as a leaf of $G'$.

Note that by \cref{lem:seq_coda}\cref{it:seq_coda1}, $Z_0= Z_K-\lfloor Z_K \rfloor\leq
x(Z_K-\lfloor Z_K \rfloor )=\bar{Z}_0$, hence there exists a Laufer sequence connecting
$Z_0$ with $\bar{Z}_0$. Furthermore,
using idempotency and monotonicity of the Laufer operator
\cref{block:Laufer_op}, we find
\[
  \bar Z_i + E_{\bar v(i)}
  = x(\bar Z_i) + E_{\bar v(i)}
  \leq x(\bar Z_i + E_{\bar v(i)})
  = \bar Z_{i+1}.
\]
As a result, we can join $\bar Z_i + E_{\bar v(i)}$
and $\bar Z_{i+1}$ by a Laufer sequence.
This way, we obtain a computation sequence $(Z_i)_i$,
connecting  $Z_K-\lfloor Z_K\rfloor$ with  $x(Z_K)$.
Finally, by \cref{lem:seq_coda}\cref{it:seq_coda2},
 $x(Z_k)$ satisfies the requirement \cref{def:CS}\cref{it:CS2} too,
hence \cref{cor:pg_bound} applies.
\end{definition}

\begin{block}
For  a diagonal computation sequence as above at  each
step, except for the step from $\bar Z_i$ to
$\bar Z_i + E_{\bar v(i)}$, we have $d_i < 0$, we find, using
\cref{lem:pg_lemma,lem:seq_coda}
\begin{equation} \label{eq:pg_bound_bar}
  p_g \leq \sum_{i=0}^{\bar k-1} \max\{0, (-\bar Z_i, E_{\bar v(i)})+1\}.
\end{equation}
\end{block}

\begin{thm} \label{thm:diagonal}
Let $(X,0)$ be a normal Newton nondegenerate Weil divisor given by
a function $f$, with a rational
homology sphere link, and assume that the polyhedron $\Gamma_+(f)$ is
$\Q$-Gorenstein  pointed at $p\in M_\Q$.
Then, a diagonal computation sequence $(Z_i)_i$  constructed above
computes the geometric
genus, that is, equality holds in \cref{eq:pg_bound_bar}.
\end{thm}

In this sequel we prove the theorem under the assumption 10.9 regarding the
multiplicities of $Z_K$,
by the results of the next section this assumption can be removed.

\begin{definition}
Let $n\in\Nd$, corresponding to the face $F_n \subset \Gamma(f)$.
Denote by $C_n$ the convex hull of the union of $F_n$ and $\{p\}$.
Set also
\[
  C^-_n = C_n \setminus \bigcup_{n'\geq n} C_{n'},
\]
where we use the partial ordering $\leq$ on $\Nd$ defined in
\cref{def:comp_seq}.
For $i=0, \ldots, \bar k-1$, let $H_i$ be the hyperplane in
$M_\R$ defined as the set of points $q\in M_\R$ satisfying
$\ell_n(q-p) = m_{\bar v(i)}(\bar Z_i)$.
For $i=0,\ldots, \bar k-1$, we set
\[
  F_i = C_{\bar v(i)} \cap H_i,\qquad
  F^-_i = C_{\bar v(i)}^- \cap H_i.
\]
\end{definition}

\begin{rem} \label{rem:Fi_nonempty}
The affine plane $H_i$ contains an \emph{affine lattice} $M \cap H$, that is
there is an affine isomorphism
$H \to \R^2$, inducing a bijection $H\cap M \to \Z^2$.
The polyhedron $F_i$ is then the image of a lattice polyhedron
with no integral integer points
under a homothety with ratio in $[0,1[$ if $i < \bar k'$.
These properties allow us to apply \cref{lem:polygonal_point_count}
in the proof of \cref{thm:diagonal}.
Furthermore, the polygon $F_i$ is always nonempty, even if $F_i^-$ may
be empty.
\end{rem}

\begin{block}
The sets $C^-_n$ form a partitioning of the union of segments starting
at $p$ and ending in points on $\Gamma(f)$, that is,
$\cup_{n\in\Nd} C_n$.
This follows from the construction as follows.
The partially ordered set $(\Nd, \leq)$ is an \emph{lower semilattice},
i.e. any subset has a largest lower bound.
If $q \in \cup_{n\in\Nd} C_n$, and
$\mathcal{I} \subset \Nd$ is the set of nodes $n$ for which $q\in C_n$,
then $q\in C_{n_q}^-$, and $p\notin C^-_n$ for $n\neq n_q$,
where $n_q$ is the largerst lower bound of $\mathcal{I}$.

The integral points $q$ in the union of the sets $C_n^-\setminus F_n$ are
presicely the integral points satisfying
$\ell_\sigma(q) > m_\sigma$ for all $\sigma\in \tilde\triangle_f^{*(1,1)}$
and $\ell_\sigma(q) \leq m_\sigma$ for some
$\sigma \in \triangle_f \setminus \triangle_f^*$.
Indeed, by the rational homology sphere assumption,
any integral point on the Newton diagram $\Gamma(f)$ must lie on
the boundary $\partial\Gamma(f)$,
see \cref{rem:qhs3_oka}.
These are the points
``under the Newton diagram''; by \cref{thm:geom_genus},
the number of these points is $p_g$. It follows from construction that
the family $(F^-_i \cap M)_{i=0}^{\bar k-1}$ forms a partition
of these points. We conclude:
\begin{equation} \label{eq:pg_F}
  p_g = \sum_{i=0}^{\bar k'-1} |F^-_i \cap M|.
\end{equation}
\end{block}

\begin{definition}
For $r,x\in \R$, denote by $\lceil r \rceil_x$ the smallest real number
larger or equal to $r$ and congruent to $x$ modulo $\Z$. That is,
\[
  \lceil r \rceil_x = \min\set{a\in \R}{a\geq r,\; a \equiv x\; (\mod \Z)}
\]
\end{definition}

\begin{rem}
The number $\lceil r \rceil_x$ depends on $x$ only up to an integer.
For all $i$, we have $\bar Z_i \equiv Z_K\;(\mod L)$. In particular,
given an $n\in \Nd$, we have $m_n(\bar Z_i) \equiv m_n(Z_K-E)\;(\mod \Z)$.
\end{rem}

\begin{lemma} \label{lem:m_nbr}
Let $Z \in L'$ and take $n,n'\in \Nd^*$ connected by a bamboo, and
$u\in \V$ a neighbour of $n$ on this bamboo. Then
\begin{equation} \label{eq:m_nbr}
  m_u(x(Z)) =
  \left\lceil
  \frac{\beta m_n(Z) + m_{n'}(Z)}{\alpha}
  \right\rceil_{m_u(Z)}
\end{equation}
where $\alpha = \alpha(\ell_n, \ell_{n'})$ and
$\beta = \beta(\ell_n, \ell_{n'})$
(see \cref{def:canon_subdiv} and \cref{rem:alpha_beta}).
Furthermore, if for all $v \in \V$ lying on the bamboo joining
$n, n'$, we have $(Z, E_v) = 0$, then $x(Z) = Z$ along the bamboo and
\begin{equation} \label{eq:m_nbr_eq}
  m_u(x(Z)) =
  \frac{\beta m_n(Z) + m_{n'}(Z)}{\alpha}.
\end{equation}
\end{lemma}

\begin{proof}
We prove \cref{eq:m_nbr}, \cref{eq:m_nbr_eq} follows similarly.
Let $u = u_1, \ldots, u_s$ be the vertices on the bamboo with Euler numbers
$-b_1, \ldots, -b_s$ as in \cref{fig:bamboo}. Set
$\tilde m_0 = m_0 = m_n(Z)$ and
$\tilde m_{s+1} = m_{s+1}= m_{n'}(Z)$. There exists a unique set of numbers
$\tilde m_1, \ldots, \tilde m_s \in \Q$ so that the equations
\begin{equation} \label{eq:tilde_m_eq}
  \tilde m_{i-1} -b_i \tilde m_i + \tilde m_{i+1} = 0, \quad i=1, \ldots, s
\end{equation}
are satisfied. This follows from the fact that the intersection matrix
of the bamboo is invertible over $\Q$. In fact, it follows from
\cite[Lemma 20.2]{EisNeu} that in fact,
\[
  \tilde m_1 = \frac{\beta m_0 + m_{s+1}}{\alpha}.
\]
This, and the lower bound by intersection numbers from
\cref{block:Laufer_op}, implies that $m_u(x(Z)) \geq \tilde m_1$, and
therefore $m_u(x(Z)) \geq \lceil \tilde m_1 \rceil_{m_u(Z)}$, since
$x(Z) - Z \in L$.

For the inverse inequality, we must show that there exist numbers
$m_1, \ldots, m_s$ satisfying
\begin{equation} \label{eq:m_ineq}
  m_{i-1} -b_i m_i + m_{i+1} \leq 0, \qquad
  m_i \equiv m_i(Z) \; (\mod \Z),
\end{equation}
for $i=1,\ldots, s$,
and so that $m_1$ is the right hand side of \cref{eq:m_nbr}.
Let $\ell_n = \ell_0, \ell_1, \ldots, \ell_s, \ell_{s+1} = \ell_{n'}$
be the canoncial primitive sequence as in \cref{def:okas_graph}, and
note that $\beta = \alpha(\ell_1, \ell_{s+1})$.
Set recursively
\[
  m_i =
  \left\lceil
  \frac{\alpha(\ell_i,     \ell_{s+1}) m_{i-1} + m_{s+1}}
       {\alpha(\ell_{i-1}, \ell_{s+1})}
  \right\rceil_{m_{u_i}(Z)}
  \qquad
  i=1,\ldots, s.
\]
Note that, by definition, $m_i \equiv m_i(Z)$.
The assumption $Z \in L'$ therefore implies that the
left hand side of \cref{eq:m_ineq} is an integer.
It is then enough to prove \cref{eq:m_ineq} for $i=1$.
This equation is clear if $s = 1$, so we assume that $s > 1$.
Setting $\gamma = \alpha(\ell_2, \ell_s)$, we find
\[
  m_2 - \tilde m_2
  =
  \left\lceil
  \frac{\gamma m_1 + m_{s+1}}{\beta}
  \right\rceil_{m_{u_2}(Z)}
  -
  \frac{\gamma \tilde m_1 + \tilde m_{s+1}}{\beta}
  = \frac\gamma\beta (m_1 - \tilde m_1) + r
\]
where $0\leq r < 1$.
In order to prove \cref{eq:m_ineq}, we start by subtracting zero, i.e.
the left hand side of \cref{eq:tilde_m_eq}.
The left hand side of \cref{eq:m_ineq} equals
\[
  m_0 - \tilde m_0 - b_1(m_1 - \tilde m_1) + m_2 - \tilde m_2
  = \left(-b_1 + \frac\gamma\beta\right) (m_1 - \tilde m_1) + r < 1,
\]
since $\gamma/\beta < 1$. Since the left hand side is an integer,
\cref{eq:m_ineq} follows.
\end{proof}

\begin{lemma} \label{lem:kprimek}
If $\bar k' \leq i < \bar k$, then $(\bar Z_i, E_{\bar v(i)}) > 0$.
As a result, the corresponding terms in \cref{eq:pg_bound_bar} vanish.
\end{lemma}

\begin{proof}
Let $u\in \V_n$ be a neighbour of $\bar v(i)$. Assume first that
$u$ lies on a bamboo connecting $\bar v(i)$ and $n\in \Nd$.
We then have $m_{\bar v(i)}(\bar Z_i) =  m_{\bar v(i)}(Z_K-E)$.
Furthermore, $m_n(\bar Z_i) = m_n(Z_K-E) + \epsilon$,
where $\epsilon$ equals $0$ or $1$. By the previous lemma, we find
\[
  m_u(\bar Z_i) =
  \left\lceil
  \frac{\beta m_{\bar v(i)}(\bar Z_i) + m_n(Z_K-E) + \epsilon}{\alpha}
  \right\rceil_{m_u(Z_K)}
  =
  m_u(Z_K-E) + \epsilon.
\]
with $\alpha, \beta$ as in the lemma.

Next, assume that $u$ lies on a bamboo connecting $\bar v(i)$ and
$n'\in \Nd^*\setminus \Nd$. Name the vertices on the bamboo $u_1, \ldots, u_s$
as in the proof of the previous lemma. We then have $(Z_K-E,E_{u_j}) = 0$
for $j=1,\ldots, s-1$, and $(Z_K-E,E_{u_s}) = 1$. By the lower bound on
intersection numbers, we find $x(Z_K-E) \geq Z_K-E$. A Laufer sequence
which computes $x(Z_K-E)$ from $Z_K-E$ may start with
$E_{u_s}, E_{u_{s-1}}, \ldots, E_{u_1}$. This shows that
$m_u(x(Z_K-E)) \geq m_u(Z_K-E) + 1$ in this case.

As a result, for every $u\in \V_{\bar v(i)}$, we have
$m_u(x(Z_K-E)) \geq m_u(Z_K-E)$, with an equality for at most one neighbour.
As a result, since $(Z_K-E, E_v) = 2 - \delta_v$ we find
\[
  (\bar Z_i, E_{\bar v(i)})
  \geq (Z_K-E, E_{\bar v(i)}) + (\delta_{\bar v(i)} - 1)
  = 1.
\]
The final statement of the lemma is now clear.
\end{proof}

\begin{lemma} \label{lem:polygonal_point_count}
Let $F \subset \R^2$ be an integral polygon with no internal integral points.
Let $S_1, \ldots, S_r$ be the faces of $F$ and let $c_j$ be the
integral lenght of $S_j$.
Let $0 \leq \rho < 1$, $J \subset \{1, \ldots, r\}$.
Then let $a_i:\R^2 \to \R$ be the unique integral affine function whose
minimal set on $\rho F$ is $\rho S_j$ and this minimal value is
$\lambda_j \in ]-1,0]$ if $j \notin J$ and $\lambda_j \in [-1,0[$ if
$j \in J$. Set $F_\rho^- = \rho F \setminus \cup_{j\in J} \rho S_j$. Then
there exists an $a \in \Z$ satifying
\[
  \sum_{j=1}^s c_j a_j \equiv a, \qquad
  |F_\rho^- \cap \Z^2| = \max\{0,a+1\}.
\]
\end{lemma}

\begin{proof}
This is \cite[Theroem 4.2.2]{Baldur_th}.
\end{proof}

\begin{proof}[Proof of \cref{thm:diagonal}]
Recall the order $\leq$ on the set $\Nd$ defined in \cref{def:comp_seq}.
We extend this order in the obvious way to all of $\V$.
Also, by assumption, $f$ is $\Q$-pointed at the point $p\in M_\Q$.
Fix an $0 \leq i \leq \bar k'$ and set $H = H_i$.
For $u\in \V_{\bar v(i)}$, define
\[
  \lambda_u = \inf \set{ \ell_u(q)}{q\in F_i}
\]
(recall that $F_i$ is nonempty, see \cref{rem:Fi_nonempty}) and
\[
  \nu_u =
\begin{cases}
  \lambda_u + 1 & \mathrm{if}\; u \leq \bar v(i) \;
                  \mathrm{and}\; \lambda_u \in \Z,\\
  \lceil \lambda_u \rceil  & \mathrm{else.}
\end{cases}
\]
Define the affine functions $a_u:H \to \R$, $a_u=\ell_u|_H - \nu_u$.
By construction, these are primitive integral functions on $H$ with respect
to the affine lattice $H \cap M$. It now follows from
\cref{lem:polygonal_point_count}
that there is an $a\in \Z$ so that $\sum_u a_u \equiv a$ and
$|F^-_i \cap M| = \max\{0,a+1\}$.

On the other hand, we claim that
$\nu_u - \ell_u(p) \leq m_u(\bar Z_i)$ for $u\in \V_{\bar v(i)}$.
Using \cref{lem:bv}, and the definition of $H_i$, i.e.
$\ell_{\bar v(i)}(q - p) |_H = m_{\bar v(i)}(\bar Z_i)$ for $q\in H$,
it follows that
\[
  a
  = \sum_{u} a_u(q)
  = \sum_u \ell_u(q-p) - (\nu_u - \ell_u(p))
  \geq b_{\bar v(i)} \ell_{\bar v(i)}(q-p) - \sum_u m_u(\bar Z_i)
  = (-\bar Z_i, E_{\bar v(i)}).
\]
where $q$ is any element of $H$. As a result, using
\cref{eq:pg_bound_bar} and \cref{lem:kprimek},
as well as \cref{eq:pg_F},
we have
\[
  p_g
  = \sum_{i=0}^{\bar k'-1} |F^-_i\cap M|
  \geq \sum_{i=0}^{\bar k-1} \max\{ 0, (-\bar Z_i, E_{\bar v(i)}) + 1\}
  \geq p_g,
\]
and so these inequalities are in fact equalities.

We are left with proving the claim
$\nu_u \leq m_u(\bar Z_i) + \ell_u(p)$ for $u\in \V_{\bar v(i)}$.
Fix $u$, and let $n\in \Nd^*$ so that $u$ lies on a bamboo connecting
$\bar v(i)$ and $n$.
Let $S = F_{\bar v(i)} \cap F_n$. Then
$S$ is the minimal set of $\ell_u$ on $F_{\bar v(i)}$, i.e.,
$S = F_u$.
Let $A$ be the affine hull of $S \cup \{p\}$. Since the two affine functions
\[
  \frac{\ell_{\bar v(i)} - \ell_{\bar v(i)}(p)}{m_{\bar v(i)}(Z_K-E)},\quad
  \frac{\ell_n - \ell_n(p)}{m_n(Z_K-E)},
\]
both take value $0$ on $p$ and $1$ on $S$,
by \cref{thm:Gor}\cref{it:Gor_ellG}, they conincide on $A$.
Let
\[
  r = \frac{m_{\bar v(i)}(\bar Z_i)}{m_{\bar v(i)}(Z_K-E)}
\]
Using the minimality of \cref{eq:ratio}, we get for any
$q\in p+ r(S-p) \subset H \cap A$
\begin{equation} \label{eq:point_count_ineq}
  \frac{\ell_n(q-p)}{m_n(Z_K-E)}
  = \frac{\ell_{\bar v(i)}(q-p)}{m_{\bar v(i)}(Z_K-E)}
  = \frac{m_{\bar v(i)}(\bar Z_i)}{m_{\bar v(i)}(Z_K-E)}
  \leq \frac{m_n(\bar Z_i)}{m_n(Z_K-E)},
\end{equation}
and so $\ell_n(q-p) \leq m_n(\bar Z_i)$.
In the case when $n \leq \bar v(i)$, or equivalently,
$u \leq \bar v(i)$, this inequality is strict.
It follows, using \cref{lem:m_nbr}, that
\begin{equation} \label{eq:diagonal_claim}
\begin{split}
  m_u(\bar Z_i)
  &=
  \left\lceil
    \frac{\beta(\ell_{\bar v(i)}, \ell_n) m_{\bar v(i)}(\bar Z_i)
           + m_n(\bar Z_i)}
         {\alpha(\ell_{\bar v(i)}, \ell_n)}
  \right\rceil_{m_u(\bar Z_i)} \\
  &\geq
    \frac{\beta(\ell_{\bar v(i)}, \ell_n) \ell_{\bar v(i)}(q-p)
           + \ell_n(q-p)}
         {\alpha(\ell_{\bar v(i)}, \ell_n)} \\
  &= \ell_u(q-p) \\
  &= \lambda_u - \ell_u(p).
\end{split}
\end{equation}
Therefore, since
$m_u(\bar Z_i) \equiv m_u(Z_K) \equiv -\ell_u(p)$ $(\mod \Z)$, we find
\[
  m_u(\bar Z_i)
  \geq
  \lceil
    \lambda_u
  \rceil
  - \ell_u(p).
\]
This proves the claim, unless
$u\leq \bar v(i)$ and $\lambda_u \in \Z$.
In that case, the numbers $\ell_{\bar v(i)}(q)$ and $\ell_u(q) = \lambda_u$
are both integers. Since $\ell_n, \ell_u$ form a part of an integral basis of
$N = M^\vee$, we can assume that $q \in M$,
hence,
\[
  \frac{\beta(\ell_{\bar v(i)}, \ell_n) \ell_{\bar v(i)}(q-p)
         + \ell_n(q-p)}
       {\alpha(\ell_{\bar v(i)}, \ell_n)} \\
  =
  \ell_u(q-p)
  \equiv
  -\ell_u(p)
  \equiv m_u(Z_K)
  \equiv m_u(\bar Z_i)
  \quad
  (\mod \Z).
\]
As a result,
since we have a strict inequality $m_n(\bar Z_i) > \ell_n(q-p)$ we get
a strict inequality in \cref{eq:diagonal_claim} as well.
Therefore, we have
\[
  m_u(\bar Z_i) > \lambda_u - \ell_u(p)\qquad
  \mathrm{and} \qquad
  m_u(\bar Z_i) \equiv \lambda_u - \ell_u(p) \quad (\mod \Z),
\]
and so
$m_u(\bar Z_i) \geq \lambda_u - \ell_u(p) + 1 = \nu_u - \ell_u(p)$,
which finishes the proof of the claim.
\end{proof}

\section{Removing $B_1$-facets} \label{s:rem_face}

In this section we consider only surface singularities, i.e. we
assume that $r=3$. We consider \emph{removable} $B_1$-facets of
two dimensional Newton diagrams
and show that they can be \emph{removed} without affecting certain invariants
of nondegenerate Weil divisors.
This is stated in \cref{prop:removable}.
In parallel we also prove \cref{prop:ZK-E_geq_0}, which allows us to assume
that the divisor $Z_K-E$ on the resolution provided by Oka's algorithm
has nonnegative multiplicities on nodes,
cf. \cref{block:diagonal_newton} and the sentence after
\cref{thm:diagonal}.
Similar computations are given in
\cite{Newton_nondeg}, providing a
stronger result in the case of a hypersurface singularity in $\C^3$
with rational homology sphere link.

The concept of a $B_1$-facet appears in \cite{Denef_poles} in the
case of hypersurfaces in $K^r$, where $K$ is a $p$-adic field,
and is further studied in \cite{Lem_Pro_mon,Bories_Veys}.

\begin{definition} \label{def:B1}
Let $F \subset \Gamma(f)$ be a compact facet,
i.e. of dimension $2$. Then $F$ is a
\emph{$B_1$-facet} if $F$ has exactly $3$ vertices $p_1,p_2,p_3$ so that
there is a $\sigma\in \triangle^{(1)}_\Sigma$ so that
$m_\sigma = \ell_\sigma(p_1) = \ell_\sigma(p_2) = \ell_\sigma(p_3)-1$.
A $B_1$-facet $F$ is \emph{removable} if furthermore,
the segment $[p_2,p_3]$ is contained in the boundary
$\partial\Gamma(f)$ of $\Gamma(f)$.
\end{definition}

\begin{figure}[ht]
\begin{center}
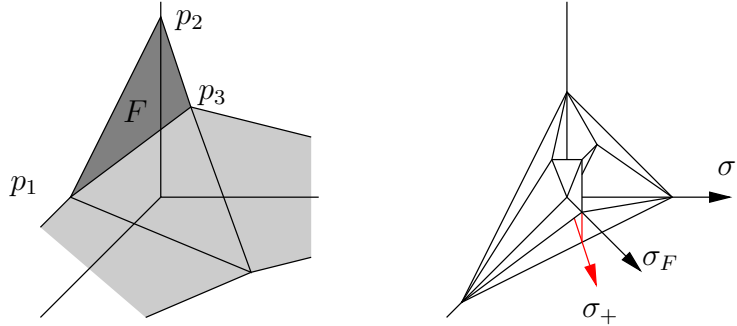
\caption{On the left we have a Newton diagram in $\R^3_{\geq 0}$
with a removable $B_1$ facet $F$.
To the left, we see the $2$-skeleton of the dual fan, and an intersection
with a hyperplane. In this example we have
$\ell_\sigma(p_1) = \ell_\sigma(p_2) = 0$ and $\ell_\sigma(p_3) = 1$.}
\label{fig:B1_dem}
\end{center}
\end{figure}

\begin{definition} \label{def:trop}
Let $T(f)$ be closure in $N_\R$ of
the union of cones in $\triangle_f$ which correspond to
compact facets of $\Gamma_+(f)$ which have dimension $>0$.
This is the \emph{tropicalization} of $f$.
We say that $\Sigma$ is
\emph{generated by the tropicalization of $f$}, if $\Sigma$ is
generated as a cone by the set $T(f)$.

Let $\Sigma'$ be the cone generated by $T(f)$. This is a finitely generated
rational strictly convex cone, and if $(X,0)$ is not rational, then
$\Sigma'$ has dimension $r = 3$.
This cone induces an affine toric variety $Y' = Y_{\Sigma'}$, and the function
$f$ defines a Weil divisor $(X',0) \subset (Y',0)$.
Furthermore, the inclusion $\Sigma' \subset \Sigma$ induces a morphism
$Y' \to Y$, which restricts to a morphism $(X',0) \to (X,0)$.
\end{definition}

\begin{rem}
The closure of $T(f)$ in a certain partial compactification of $N_\R$
is called the \emph{local tropicalization} of $(X,0)$ \cite{pop_step_trop}.
\end{rem}

\begin{figure}[ht]
\begin{center}
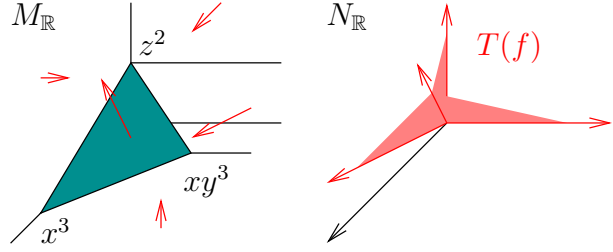
\caption{
Here, $\Sigma = \R^3_{\geq 0}$ is the positive octant,
and $f(x,y,z) = x^3 + xy^3 + z^2$ is the $E_7$ singularity in normal form.
In this case, $T(f)$ does not generate $\Sigma$, but the cone
generated by $(2,0,1)$, $(0,1,0)$ and $(0,0,1)$.}
\label{fig:E7_dem}
\end{center}
\end{figure}

\begin{lemma} \label{lem:orbit_trop}
Let $\sigma \in \tilde\triangle_f$. Then the orbit $O_\sigma$ intersects
$\tilde X$ if and only if $\sigma \subset T(f)$.
\end{lemma}

\begin{proof}
The orbit $O_\sigma$ is an affine variety
$O_\sigma = \Spec(\C[M(\sigma)])$ (recall $M(\sigma) = M \cap \sigma^\perp)$,
and if $p_\sigma$ is an element of the affine hull of $F_\sigma$, then
$x^{-p_\sigma} f_\sigma \in \C[M(\sigma)]$
and
\[
  \X \cap O_\sigma
  \cong
  \Spec
  \left(
    \frac{\C[M(\sigma)]}
         {(x^{-p_\sigma} f_\sigma)}
  \right).
\]
Therefore, $\X\cap O_\sigma$ is empty if and only if $x^{-p_\sigma} f_\sigma$
is a unit in $\C[M(\sigma)]$,
which is equivalent to $f_\sigma$ being a monomial,
i.e. $\dim F_\sigma = 0$.
\end{proof}

\begin{lemma} \label{lem:trop_normal}
Let $(X,0)$ and $(X',0)$ be as in \cref{def:trop}.
If $(X,0)$ is normal, then the morphism $(X',0) \to (X,0)$ is an isomorphism.
\end{lemma}
\begin{proof}
We can assume that the smooth subdivision $\tilde\triangle_f$ subdivides
the cone $\Sigma'$, so that we get a subdivision
$\tilde\triangle_f' = \tilde\triangle_f|_{\Sigma'}$ of the cone $\Sigma'$.
Let $\tilde Y'$ be the corresponding toric variety.
Let $\triangle_{T(f)}$ be the fan consisting of cones
$\sigma \in \tilde\triangle_f$ which are contained in $T(f)$.
We then get open inclusions
\[
  Y_{T(f)} \subset \tilde Y' \subset \tilde Y
\]
where $Y_{T(f)}$
is the toric variety associated with the fan $\triangle_{T(f)}$.

It follows from \cref{lem:orbit_trop} that the strict transforms
$\X$ and $\X'$ of
$X$ and $X'$, respectively, are contained in $Y_{T(f)}$, and so
$\X' = \X$. As a result,
$X'\setminus \{0\} \cong \X' \setminus \pi^{-1}(0)
= \X \setminus \pi^{-1}(0) \cong X\setminus \{0\}$. Since $(X,0)$ is normal,
the morphism $(X',0) \to (X,0)$ is an isomorphism.
\end{proof}

\begin{block} \label{block:remove}
Assume that $F\subset \Gamma(f)$ is a removable $B_1$-face,
and let
$\sigma \in \triangle^{(1)}_\Sigma$ and $p_i$ be as in
\cref{def:B1}.
If $F$ is the only facet of $\Gamma(f)$, then we leave as
an exercise to show that the graph
$G$ is equivalent to a string of rational curves, and so
$(X,0)$ is rational.
We will always assume that $F$ is not the only facet of $\Gamma(f)$.
There exists an element of $\Sigma^\circ$ which is
constant on the segment $[p_1, p_3]$ (e.g. the normal vector to $F$).
As a result, the boundary
$\partial \Sigma$ intersects the hyperplane of elements $\ell\in N_\R$ which
are constant on $[p_1,p_3]$ in two rays, $\sigma_+$ and $\sigma_-$, where
$\ell \in \sigma_+$ satisfies $\ell|_{[p_1,p_3]} \equiv \max_F \ell$, and
$\ell \in \sigma_-$ satisfies $\ell|_{[p_1,p_3]} \equiv \min_F \ell$.

Let $\ell_+ \in N$ be a primitive generator of $\sigma_+$, set
$m_+ = \max_F \ell$ and define
\[
  \bar f(x) = \sum \set{ a_p x^p }{ p\in M, \, \ell_+(p) \geq m_+ },
\]
where $a_p$ are the coefficients of $f$ as in \cref{eq:expansion}.
Let $(\bar X,0)$ be the Weil divisor defined by $\bar f$.
We get a Newton polyhedron $\Gamma_+(\bar f)$, from which we calculate
invariants of $(\bar X,0)$ as described in previous sections.
It follows from this construction that
$\Gamma(\bar f) = \overline{\Gamma(f)\setminus F}$,
and that  $\bar f$ is Newton nondegenerate.

Now, assume that $\Sigma$ is generated by the tropicalization of $f$.
Let
$\sigma_1$ and $\sigma_3 \in \triangle_f^{(1)}$ be the rays corresponding
to the noncompact faces of $\Gamma_+(f)$ containing the segments
$[p_2,p_3]$ and $[p_1,p_2]$, respectively. Let $\ell_1, \ell_3$ be
primitive generators of $\sigma_1, \sigma_3$.
By construction, and the above assumption that $\Sigma$ is generated by
$T(f)$,
we have
$\R_{\geq 0}\langle \ell_1, \ell_3 \rangle \subset \partial \Sigma$, and so
$\ell_+ \in \R_{\geq 0}\langle \ell_1, \ell_3 \rangle \in \triangle_f$.

In fact, we have $\ell_+ = \ell_1 + t \ell_3$ where
$t = \ell_1(p_1 - p_2)$.
Indeed, $\ell_+$ is the unique positive linear combination of
$\ell_1$ and $\ell_3$ which vanishes on $p_1-p_3$, and is primitive.
Since $\ell_3 = \ell_\sigma$, by definition of $F$,
and since $\ell_1(p_3) = \ell_1(p_2)$, we have
\[
  (\ell_1+t\ell_3)(p_1 - p_3)
  =
  \ell_1(p_1-p_3) + \ell_1(p_1 - p_2) \cdot \ell_3(p_1 - p_3)
  =
  \ell_1(p_1-p_2) - \ell_1(p_1 - p_2)
  =
  0.
\]
Furthermore, we have $\ell_1(p_3-p_2) = 0$ and $\ell_3(p_3-p_2) = 1$,
and so by \cref{rem:alpha}, $\ell_1, \ell_3$ form a part of an integral
basis, which implies that $\ell_1 + t\ell_3$ is primitive.

Now, define $t'$ as the combinatorial length of the segment
$[p_1,p_2]$.
We have $t'|t$ and via Oka's algorithm (\cref{def:okas_graph}),
this segment corresponds
to $t'$ bamboos in $G$, each consisting of a single $(-1)$-curve, whereas
$[p_2,p_3]$ corresponds to one bamboo with determinant $t/t'$.
\end{block}

\begin{prop} \label{prop:removable}
Let $f$, $F$ and $\bar f$ be as above, and assume that $f$ is Newton
nondegenerate. Assume also that $\Sigma$ is generated by the
tropicalization $T(f)$ as
described in \cref{def:trop}. Then
\begin{enumerate}

\item \label{it:removable_nondeg}
$\bar f$ is Newton nondegenerate.

\item \label{it:removable_Gamma}
$\Gamma(\bar f) = \overline{\Gamma(f) \setminus F}$.

\item \label{it:removable_link}
The singularities $(X,0)$ and $(\bar X,0)$ have diffeomorphic links.

\item \label{it:removable_pg}
The singularities $(X,0)$ and $(\bar X,0)$ have equal geometric genera
and $\delta$-invariants.

\item \label{it:removable_normal}
If $(X,0)$ is normal, then $(\bar X,0)$ is normal.

\item \label{it:removable_Gor}
If $f$ is $\Q$-Gorenstein-pointed at $p \in M_\Q$, then so is $\bar f$.
In particular, if $(X,0)$ is Gorenstein,
then $(\bar X,0)$ is also Gorenstein.

\end{enumerate}
\end{prop}

\begin{proof}
\cref{it:removable_nondeg} and \cref{it:removable_Gamma} follow from
definition.

We now prove \cref{it:removable_link}. We have $G$, the output of
Oka's algorithm for the Newton polyhedron $\Gamma_+(f)$, and
$\bar G$, the output of Oka's algorithm for $\Gamma_+(\bar f)$.
Let $\sigma_F \in \triangle_f$ be the ray dual to $F$ and let $F'$
be the unique face of $\Gamma_+(f)$ adjacent to $F$, i.e.
$F'\cap F = [p_1, p_3]$.
Then
$\sigma_F \subset \R_{\geq 0}\langle \ell_{F'}, \ell_+ \rangle
\in \triangle_{\bar f}$, and we can subdivide the canonical subdivision
of $\R_{\geq 0}\langle \ell_F', \ell_+ \rangle$ so that
we can assume that $\sigma_F \in \tilde\triangle_{\bar f}$.
We can therefore identify vertices $v_F$ of $G$ and $\bar G$
corresponding to the same ray
$\sigma_F \in \tilde\triangle^{(1)}_f$ and
$\sigma_F \in \tilde\triangle^{(1)}_{\bar f}$.
It is then clear from construction that the components of
$G\setminus v_F$ and $\bar G\setminus v_F$
in the direction of $v_{F'}$ are isomorphic. After blowing down the
$(-1)$-curves corresponding to the segment $[p_1, p_2]$, we must show
\begin{itemize}

\item
The two bamboos joining $\ell_F$ with $\ell_+$ on one hand, and with
$\ell_1$ on the other, are isomorphic.

\item
The vertex $v_F$ has the same Euler number in $G$ and in $\bar G$.

\end{itemize}

\begin{figure}[ht]
\begin{center}
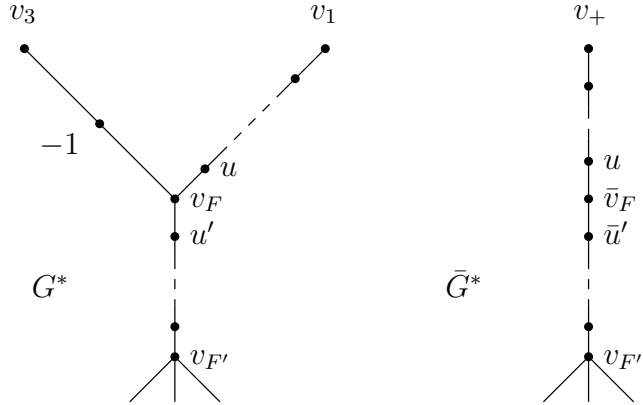
\caption{The $(-1)$-curve to the left is blown down, so that the two
graphs $G$ and $\bar G$, obtained by deleting $v_3, v_1, v_+$ and their
adjacent edges, look topologically the same.
To the right, the bamboo
connecting $\bar v_{F'}$ and $v_+$ corresponds to a subdivision of
the cone generated by $\ell_+$ and $\ell_{F'}$ which contains the
ray generated by $\ell_F$.}
\label{fig:m1}
\end{center}
\end{figure}

For the first of these, we prove that
\[
  \alpha(\ell_F, \ell_+) = \alpha(\ell_F, \ell_1), \quad
  \beta(\ell_F, \ell_+) = \beta(\ell_F, \ell_1).
\]

We calculate $\alpha(\ell_F, \ell_+)$ as the greatest common divisor
of maximal minors of the matrix having coordinate vectors for
$\ell_F$ and $\ell_+ = \ell_1+t\ell_3$ as rows.
But $\alpha(\ell_F, \ell_1) = t/t'$, and so adding a multiple of $t$
to $\ell_1$ does not modify the greatest common divisor of
these determinants, hence
$\alpha(\ell_F, \ell_+)
=\alpha(\ell_F, \ell_1+t\ell_3)
=\alpha(\ell_F, \ell_1)$.

The invariant $\beta(\ell_F, \ell_+)$ can be calculated as the unique
number $0\leq \beta < \alpha(\ell_F, \ell_+)$ so that
$\beta \ell_F + \ell_+$ is a multiple of $\alpha(\ell_F, \ell_+)$.
On the other hand, we find, setting $\beta = \beta(\ell_F, \ell_1)$
and $\alpha = \alpha(\ell_F, \ell_+) = \alpha(\ell_F, \ell_1) = t/t'$,
\[
  \frac{\beta \ell_F + \ell_+}{\alpha}
  = \frac{\beta \ell_F + \ell_1 + t\ell_3}{\alpha}
  = \frac{\beta \ell_F + \ell_1}{\alpha} + t'\ell_3 \in N.
\]

Finally, we show that $v_F$ has the same Euler number in the graphs
$G$ and $\bar G$. Denote these by $-b_F$ and $-\bar b_F$. After blowing
down the $(-1)$ curves associated with the segment $[p_1, p_2]$,
the vertex $v_F$ has two neighbors
in either graph $G$ or $\bar G$. Denote by $v_{-1}$ and $\bar v_{-1}$
the neighbor of $v_F$ contained in the same component of $G\setminus v_F$
and $\bar G_\setminus v_F$ as $v_{F'}$. It is then clear that
$\ell_{v_{-1}} = \ell_{\bar v_{-1}}$.

Denote by $u, \bar v$ the neighbours of $v_F, \bar v_F$
in the direction of $v_1, v_+$, respectively, and
$u', \bar u'$ the other neighbours, as in
\cref{fig:m1}.
Then we have $\ell_{u'} = \ell_{\bar u'}$ and
\[
  \ell_u = \frac{\beta \ell_F + \ell_1}{\alpha},\quad
  \ell_{\bar u} = \frac{\beta \ell_F + \ell_+}{\alpha}
  = \ell_u + t'\ell_3,
\]
where $\alpha, \beta$ are as above. The two numbers $-b_F$ and $-\bar b_F$ are
identified by \cref{lem:bv}
\[
  -b_F \ell_F + \ell_u + \ell_{u'} + t'\ell_3 = 0,\quad
  -\bar b_F \ell_F + \ell_{\bar u} + \ell_{\bar u'} = 0,
\]
which leads to their equality.

Next, we prove \cref{it:removable_pg} and \cref{it:removable_normal}.
By \cref{thm:geom_genus}, it suffices to show that
\[
  \Gamma_+(f) \setminus (\Sigma^\vee + q),
  \qquad
  \Gamma_+(\bar f) \setminus (\Sigma^\vee + q)
\]
have the same cohomology for all $q\in M$. By shifting $\Gamma_+(f)$, we
simplify the following proof by assuming $q=0$.
The inclusion
\[
  \Gamma(f) \setminus \Sigma^\vee
  \subset
  \Gamma_+(f) \setminus \Sigma^\vee
\]
is a homotopy equivalence. Indeed, one can construct a suitable vectorfied
on $\Gamma_+(f) \setminus \Sigma^\vee$ pointing in the direction
of $-\Sigma^\vee$, whose trajectories end up in
$\Gamma(f) \setminus \Sigma$, thus giving a homotopy inverse to the above
inclusion.

Now, let $K$ be the union of faces of $\Gamma(f)$ which do
not intersect $\Sigma^\vee$. By \cref{lem:subcx}, the inclusion
$K\subset \Gamma(f) \setminus \Sigma$ is a homotopy equivalence.
Define $\bar K$ similarly, using $\bar f$.
Thus it suffices to prove that
$\tilde H^i(K,\bar K;\Z)$ vanish for all $i$. By excision, this is equivalent
to showing
\begin{equation} \label{eq:red_vanishing}
  \fa{i\in \Z_{\geq 0}}{\tilde H^i(K\cap F,\bar K\cap F;\Z) = 0}.
\end{equation}

If $\Sigma^\vee$ does not intersect the face $F$, then
$K\cap F = F = \bar K \cap F$.
Also, if $p_2 \in \Sigma^\vee$, then
$K\cap F = \bar K \cap F$. In either case,
\cref{eq:red_vanishing} holds.
We can therefore assume that $p_2 \in K$ and $F \not\subset K$.
With these assumptions at hand, it is then enough to prove that excactly
one of the segements $[p_1,p_2]$ and $[p_2,p_3]$ is contained in $K$,
i.e. it cannot happen that either both or neither is contained in $K$.

Let $A$ be the affine hull of $F_n$, i.e. the
hyperplane in $M_\R$ defined by $\ell_n = m_n$, and let
$C = \Sigma^\vee \cap A$.
Define a point $r \in A$ by
\[
  \ell_3(r) = 0,\quad
  \ell_1(r) = 0,\quad
  \ell_n(r) = m_n.
\]
This is well defined, since the functions $\ell_1, \ell_3, \ell_n$ are
linearly independent.
Then $C$ is a convex polygon in $A$, and
$r$ is a vertex of $C$. Furthermore, $r$ is the unique point in
$C$ where both functions $\ell_1|_C$ and $\ell_3|_C$ take their minimal values.

If neither of the segments $[p_1,p_2]$, $[p_2,p_3]$ are contained in
$K$, i.e. both intersect $\Sigma^\vee$, then
we can choose $r_1 \in C\cap[p_1,p_2]$ and $r_2\in C \cap [p_2,p_3]$.
Furthermore,
we have $\ell_3(r) \leq \ell_3(r_1) = \ell_3(p_2)$, and
$\ell_1(r) \leq \ell_1(r_2) = \ell_1(p_2)$. Therefore, $p_2$ is
in the convex hull of $r,r_1,r_2$, and so $p_2\in C$, contrary to the
assumption $p_2\in K$.
\begin{figure}[ht]
\begin{center}
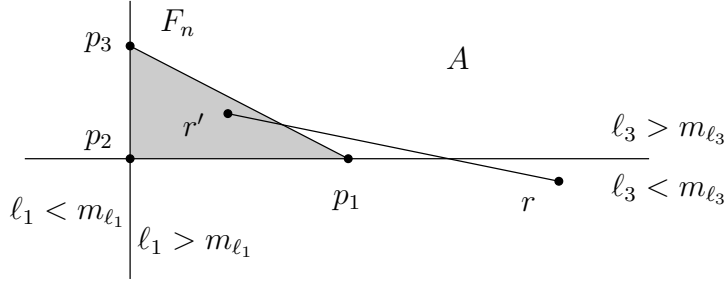
\caption{The segment $[r',r]$ intersects neither $[p_1,p_2]$ nor $[p_2,p_3]$.}
\label{fig:segment}
\end{center}
\end{figure}

Next, assume that both segments $[p_1,p_2]$, $[p_2,p_3]$ are contained
in $K$. We start by showing that in this case, we have $r\in F_n$.
By assumption, we can choose $r' \in C \cap F_n$.

We have $\ell_3(r') > m_{\ell_3}$.
One verifies (see \cref{fig:segment})
that if $\ell_3(r) \leq m_{\ell_3}$, then we would
have $\ell_1(r') < \ell_1(r)$, but $r$ is a minimum for $\ell_1|_C$.
Therefore, we can assume that $\ell_3(r) > m_{\ell_3}$, similarly,
$\ell_1(r) > m_{\ell_1}$. It follows, since $C \cap F \neq \emptyset$,
that $r \in F$, so we can assume that $r' = r$.
But, since $r \notin [p_1,p_2] \cup [p_2,p_3]$, we find
\[
  \ell_3(p_2) < \ell_3(r) < \ell_3(p_3) = \ell_3(p_2)+1,
\]
and so $\ell_3(r) \notin \Z$. But this is a contradiction, since
$\ell_3(r) = \ell_3(q) \in \Z$.

Next we prove \cref{it:removable_Gor}. Assume that
$\Gamma_+(f)$ is $\Q$-Gorenstein pointed at $p\in M_\Q$.
It suffices to show that
$\ell_+(p) = \bar m_{\ell_+} + 1$, where $\bar m_{\ell_+}$ is the minimal
value of $\ell_+$ on $\Gamma_+(\bar f)$. We immediately find
\[
  \bar m_{\ell_+} = \ell_+(p_3)
  = \ell_1(p_3) + t\ell_3(p_3)
  = m_{\ell_1} + t(m_{\ell_3}+1)
  = \ell_1(p)-1 + t\ell_3(p)
  = \ell_+(p)-1.\qedhere
\]
\end{proof}

\begin{example}
Consider the cone $\Sigma = \R^3_{\geq 0}$ and the function
\[
  f(x,y,z) = x^3 + xy^3 + z^5 + y^{10}z,
\]
which defines a nonrational singularity $(X,0)$.
In this case, $\Gamma(f)$ has a $B_1$-facet
\[
  F = \convx \{ (1,3,0), (0,10,1), (0,0,5) \},
\]
corresponding to a node $n\in \Nd$.
The normal vector to $F$ is $(19,2,5)$ and
\cref{eq:can_cycle} gives $m_n(Z_K-E) = -1$.
By the above computations, removing the monomial $y^{10}z$ from $f$
gives another singularity with the same link and geometric genus, but
$Z_K-E$ is nonnegative on the other node. After removing $F$ we find
\[
  \bar f(x,y,z) = x^3 + xy^3 + z^5.
\]
Note that $\Sigma$ is generated by the tropicalization of $f$, but the
tropicalization of $\bar f$ generates the cone
$\R_{\geq 0}\langle(5,0,1), (0,1,0), (0,0,1)\rangle$.

\begin{figure}[ht]
\begin{center}
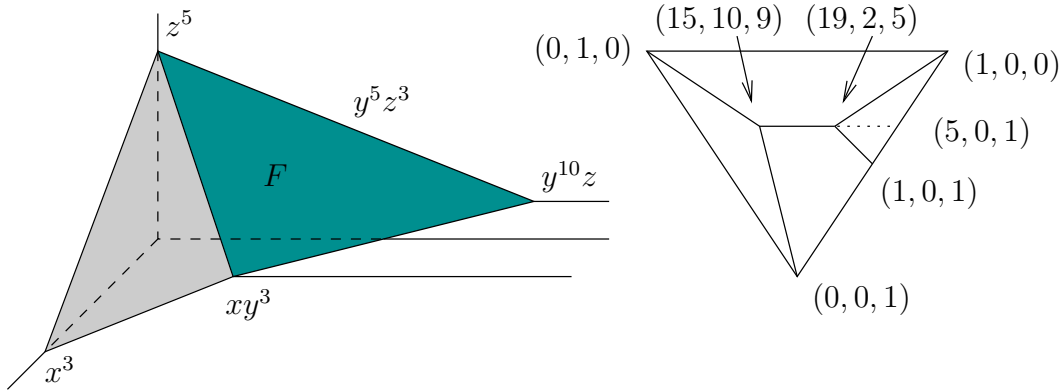
\caption{A diagram with a $B_1$-facet $F$ and its dual.
The dotted line to the right replaces
its two neighbouring segments if the $B_1$-facet is removed.}
\label{fig:B1_ex}
\end{center}
\end{figure}
\end{example}

\begin{block} \label{block:leaf_graph}
In what follows, we connect the above construction with the coefficients
of $Z_K-E$.
We introduce a simplified graph, whose vertices are the nodes of $G$.
whose vertices are the nodes of $G$, and a  bamboo of $G$
connecting two nodes of $G$ is replaced in $G_\Nd$ by an edge.
Then $G_\Nd$ is a tree, with an edge connecting
$n,n'$ if and only if $F_n$ and $F'_n$ intersect in a segment
(of length $1$).
Recall that a \emph{leaf} of a tree is a vertex with exactly one neighbour.
If we assume that $|\Nd| > 1$, then we see that
the following are equivalent, since $G_\Nd$ is a tree:
\begin{itemize}

\item
$n\in \Nd$ is a leaf in $G_\Nd$,

\item
$\Gamma(f) \setminus F_n$ is connected,

\item
all edges of $F_n$, except for one, lie on the
boundary $\partial\Gamma(f)$ of the Newton diagram.

\end{itemize}
If $|\Nd| = 1$, then there is a unique $n\in \Nd$, and
$\Gamma(f) = F$, in particular, $\partial\Gamma(f) = \partial F_n$.
Finally, if $|\Nd| = 0$, and if we assume that $(X,0)$ is normal, then
$(X,0)$ is rational.
\end{block}

\begin{figure}[ht]
\begin{center}
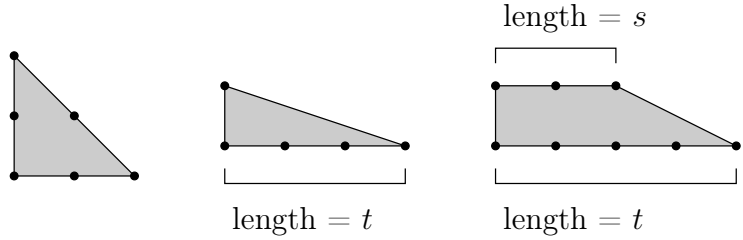
\caption{A big triangle, a small triangle of type $t=3$, and a
trapezoid of type $(t,s) = (4,2)$.}
\label{fig:empty}
\end{center}
\end{figure}

The following lemma is elementary:

\begin{lemma} \label{lem:empty_polygon}
Let $F$ be an integral polyhedron in $\R^2$, having no integral interior
points. Then, up to an integral
affine automorphism of $\R^2$, $F$ is one one the following:
\begin{itemize}
\item {\bf Big triangle}
The convex hull of $(0,0)$, $(2,0)$, $(0,2)$.
\item {\bf Small triangle of type $t$}
The convex hull of $(0,0)$, $(t,0)$, $(0,1)$.
\item {\bf Trapezoid of type $(t,s)$}
The convex hull of $(0,0)$, $(t,0)$, $(0,1)$, $(s,1)$, where
$t,s \in \Z$, $t \geq s > 0$ and $t > 0$.
\qed
\end{itemize}
\end{lemma}

\begin{lemma} \label{lem:removable_leaf}
Assume that $(X,0)$ is normal, Gorenstein-pointed at $p\in M$,
and not rational.
If $n\in\Nd$ is a leaf in $G_\Nd$ and $m_n(Z_K-E) < 0$,
then $F_n$ is a removable $B_1$-facet of $\Gamma(f)$
(See \cref{block:leaf_graph} for the definition of $G_\Nd$).
\end{lemma}
\begin{proof}
By assumption, $F_n$ has two adjacent edges contained in $\partial \Gamma(f)$,
say $[q_1, q_2]$ and $[q_2, q_3]$.
Let $F_1, F_2$ be the noncompact faces of $\Gamma_+(f)$ containing
the segments $[q_1, q_2]$ and $[q_2, q_3]$, respectively, and let
$\ell_1, \ell_2 \in \partial \Sigma$ be the primitive functions
having $F_1, F_2$ as minimal sets on $\Gamma_+(f)$,
denote these minimal values by $m_{\ell_1}, m_{\ell_2}$.

Let $l_1 = \length([q_2, q_3])$ and
$\alpha_1 = \ell_1(q_3-q_2)/l_1$ and
$l_2 = \length([q_2, q_3])$ and
$\alpha_2 = \ell_2(q_1-q_2)/l_2$.
Then, the bamboos corresponding to the segements $[q_1,q_2]$ and
$[q_2,q_3]$ have determininats $\alpha_1, \alpha_2$, see \cref{rem:alpha}.

Assume first that $F_n$ is a small triangle of type $t$,
that the segment $[q_1, q_2]$ has length $t$,
and that $\alpha_1 = 1$.
This implies that $F_n$ is a removable $B_1$-facet.

Otherwise, let $A$ be the affine hull of $F_n$.
If $F_n$ is a big triangle, a trapezoid,
or a small triangle as above, but with $\alpha_1 > 1$, then the square
\[
  \set{q\in A}{m_{\ell_1} \leq \ell_1(q) \leq m_{\ell_1}+1,\,
               m_{\ell_2} \leq \ell_2(q) \leq m_{\ell_2}+1}
\]
is contained in $F_n$.
In particular, its vertex $q_0$, the unique point in
$A$ satisfying $\ell_i(q_0) = m_{\ell_i}+1$ for $i=1,2$, is contained in
$F_n$. The set
\[
  R = \set{q \in \Sigma^\vee}{\ell_i(q) = 0,\;i=1,2}
\]
is a one dimensional face of $\Sigma^\vee$
(here we use the condition that $\Sigma$ is generated by the
tropicalization of $(X,0)$).
By our assumption $m_n \leq \ell_n(p)$
we have $p \in q_0 + R^\circ \subset \Gamma_+(f)^\circ$, contradicting
the assumption that $(X,0)$ is not rational.
\end{proof}

\begin{prop} \label{prop:has_removable}
Assume that $(X,0)$ is normal, Gorenstein-pointed at $p\in M$,
and not rational.
If there is an $n\in \Nd$ so that $m_n(Z_K-E) < 0$, then
$\Gamma(f)$ has a removable $B_1$-facet.
\end{prop}
\begin{proof}
If $n$ is a leaf in $G_\Nd$ (see \cref{block:leaf_graph}),
then $F_n$ is removable by
\cref{lem:removable_leaf}.
So let us assume that $n$ is not a leaf in $G_\Nd$, i.e. that
$\Gamma(f) \setminus F_n$ is
disconnected. The inclusion
\[
  \Gamma(f)^\circ \setminus F_n \subset
  \Gamma_+^*(f)^\circ \setminus (\{\ell_n \leq m_n\} \cup \Gamma_+(f)^\circ )
\]
is a strong homotopy retract
(here we set $\Gamma(f)^\circ = \Gamma(f) \setminus \partial\Gamma(f)$).
In particular, the right hand side is
disconnected as well.
But it follows from our assumptions that
the point $p$ is in the right hand side above. Let $C$ be a component
of $\Gamma(f)\setminus F_n$ contained in a component of the right hand side
which does not contain $p$. Then, for any $n'$ so that
$F_{n'} \subset \overline C$ we have $\ell_{n'}(p) > m_{n'}$, i.e.
$m_{n'}(Z_K-E) < 0$.
Let $G_C$ be the induced subgraph of $G_\Nd$ having vertices
$n'$ for $F_{n'} \subset \overline C$.
This graph is a nonempty tree, and
so has either exactly one vertex, or
at has least two leaves.
In the first case, the unique vertex $n'$ of $G_C$ is a leaf of $G$.
In the second case, $G_C$ has at least two leaves,
so we can choose a leaf $n'$ of $G_C$ which is
not adjacent to $n$ in $G$.
In either case, $F_{n'}$
is a removable $B_1$-facet by \cref{lem:removable_leaf}.
\end{proof}

\begin{prop} \label{prop:ZK-E_geq_0}
Assume that $f$ defines a normal Newton nondegenerate Weil divisor $(X,0)$,
which is not rational.
Then there exists a normal Newton nondegenerate Weil divisor
$(\bar X,0)$, defined by a function $\bar f$ and a cone
$\Sigma'$ (possibly different than $\Sigma$)
satisfying the following conditions:
\begin{itemize}

\item
$(\bar X,0)$ and $(X,0)$ have diffeomorphic links.

\item
$p_g(\bar X,0) = p_g(X,0)$.

\item
If $(X,0)$ is Gorenstein or pointed
at $p\in M_\Q$, then so is $(\bar X,0)$.

\item
If $F_n \subset \Gamma_+(\bar f)$ is a compact facet, then
$m_n(Z_K-E) \geq 0$.

\end{itemize}
In fact, $\Gamma(\bar{f})$ is the union of those facets $F_n$ of
$\Gamma(f)$ for which $m_n(Z_K-E) \geq 0$.
\end{prop}
\begin{proof}
By \cref{lem:trop_normal}, we can assume that $\Sigma$ is generated by $T(f)$,
since $(X,0)$ is normal (see \cref{def:trop}). 
The result therefore follows, using induction on the number of
facets of $\Gamma(f)$, and
\cref{prop:removable,prop:has_removable} below.
\end{proof}

\section{Examples} \label{s:ex}

\begin{example} \label{ex:Brieskorn}
Let $N = M = \Z^3$ and let $a,b,c\in\N$ be natural numbers with no common
factor, and let $0\leq r<s \in\N$ be coprime with $s \leq rc$.
Take
\[
  \Sigma^\vee = \R_{\geq 0}
      \left\langle
        \begin{array}{@{(}rrr@{)}}
	   ra,&  0,&-s  \\
	    0,& rb,&-s  \\
	    0,&  0,& 1
	\end{array}
      \right\rangle, \quad
  f = x_1^a + x_2^b + x_3^c.
\]
The cone $\Sigma$ is then generated by
\[
  \ell_1 = (1,0,0),\quad
  \ell_2 = (0,1,0),\quad
  \ell_3 = \frac{1}{\gcd(ab,s)}(bs, as, abr).
\]
Corresponding to these, we
have irreducible invariant divisors $D_1, D_2, D_3 \subset Y$
and multiplicities
\[
  m_1 = 0,\quad m_2 = 0,\quad m_3 = \frac{abs}{\gcd(ab,s)}.
\]
The Newton diagram $\Gamma(f)$ consists of a single face with normal
vector $\ell_0 = (bc,ac,ab)$ and $m_0 = abc$.
Fulton shows in 3.4 of \cite{Fulton_toric}
that the group of Weil divisors modulo linear
equivalence on $Y$ is generated by $D_1, D_2, D_3$, and that
$\sum_{j=1}^3 a_i D_i$ is Cartier if and only if there is a
$p = (p_1, p_2, p_3) \in M = \Z^3$ so that $a_j = \ell_j(p)$ for
$j=1,2,3$.

In our case, $X$ is equivalent to $-\sum_{i=1}^3 m_i D_i = -m_3 D_3$.
Therefore, if $X$ is Cartier, then there is a $p = (p_1, p_2, p_3) \in M$
so that $\ell_i(p) = m_i$. Therefore, we find $p_1 = p_2 = 0$, and
\[
  \frac{abr}{\gcd(ab,s)} p_3 = \frac{abs}{\gcd(ab,s)}.
\]
Therefore, $X$ is Cartier if and only if $r|s$, i.e. $r = 1$.

\begin{figure}[ht]
\begin{center}
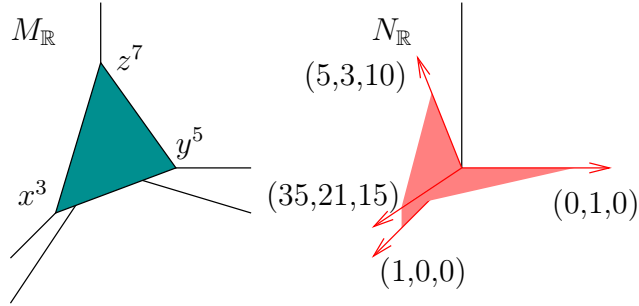
\caption{In the above examples, we have
$a=3$, $b=5$, $c=7$, $r=2$ and $s=3$. The cone $\Sigma$ is generated by the
vectors $(1,0,0)$, $(0,1,0)$ and $(5,3,10)$. Furthermore,
$(35,21,15)$ is the normal vector to the unique facet of $\Gamma(f)$.}
\label{fig:13}
\end{center}
\end{figure}
\end{example}

\begin{example}
In \cite{Nem_Oku_pg}, N\'emethi and Okuma analyse upper and lower bounds
for the geometric genus of
singularities with a specific topological type, namely,
whose link is given by the plumbing graph in
\cref{fig:13}.

\begin{figure}[ht]
\begin{center}
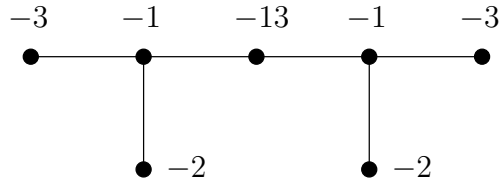
\caption{A resolution graph}
\label{fig:13}
\end{center}
\end{figure}

They show that for this graph, the path lattice cohomology is $4$, but that
the maximal geometric genus among analytic structures with this topological
type is $3$. As a result, this graph is not the topological type of a
Newton nondegenerate Weil divisor in a toric affine space.

On the other hand, this topological type is realized by
the complete intersection given by the splice equations
\[
  X = \set{z\in\C^4}
          {z_1^2 z_2 + z_3^2 + z_4^3 =
           z_1^3 + z_2^2 + z_4^2 z_3 = 0}.
\]
This singularity is in fact a Newton nondegenerate isolated
complete intersection \cite{Oka_princ_zeta}.
As a result, the methods of \cref{s:comp_seq}
do not generalize in the most straightforward way to Newton nondegenerate
complete intersections.
\end{example}


\end{document}